\documentclass[tikz]{amsart}
\usepackage{amsmath, amsthm, amssymb, amsfonts}
\usepackage[normalem]{ulem}
\usepackage{hyperref}
\usepackage{tikz-cd}
\usepackage{combelow}

\usepackage[
backend=biber,
style=numeric,
maxnames=99,
maxbibnames=99,
giveninits=true,
doi=false,
isbn=false,
url=false,
eprint=false,
sorting=nyt
]{biblatex}
\DeclareFieldFormat[article,inbook,incollection,inproceedings,thesis,unpublished,misc]{title}{#1\isdot}

\renewbibmacro{in:}{}
\usepackage{enumerate}

\usepackage{geometry}
\newgeometry{margin=1.1in}

\usepackage{verbatim} 
\usepackage{longtable}
\usepackage{caption}
\usepackage{mathtools}

\usepackage{tikz}
\usetikzlibrary{positioning,arrows.meta,calc}

\tikzset{
  vertex/.style = {circle, draw, very thin, minimum size=20pt, inner sep=1pt, font=\small},
  leg/.style    = {draw, very thin},
  edge/.style   = {draw, very thin},
  lab/.style    = {font=\footnotesize}
}

\usepackage{caption}

\theoremstyle{plain}
\newtheorem{thm}{Theorem}

\newtheorem{cor}[thm]{Corollary}

\newtheorem{lemma}[thm]{Lemma}
\newtheorem{prop}[thm]{Proposition}
\newtheorem{conj}[thm]{Conjecture}

\theoremstyle{definition}
\newtheorem{defn}[thm]{Definition}
\newtheorem{example}[thm]{Example}

\theoremstyle{remark}
\newtheorem{rmk}[thm]{Remark}

\newcommand{\ellt}{{a}}
\newcommand{\BA}{{\mathbb{A}}}

\newcommand{\BC}{{\mathbb{C}}}

\newcommand{\BE}{{\mathbb{E}}}
\newcommand{\BF}{{\mathbb{F}}}
\newcommand{\BG}{{\mathbb{G}}}

\newcommand{\BI}{{\mathbb{I}}}

\newcommand{\BL}{{\mathbb{L}}}
\newcommand{\BM}{{\mathbb{M}}}
\newcommand{\BN}{{\mathbb{N}}}

\newcommand{\BQ}{{\mathbb{Q}}}
\newcommand{\BR}{{\mathbb{R}}}

\newcommand{\BT}{{\mathbb{T}}}

\newcommand{\BZ}{{\mathbb{Z}}}

\newcommand{\CC}{{\mathcal C}}
\newcommand{\CD}{{\mathcal D}}
\newcommand{\CE}{{\mathcal E}}
\newcommand{\CF}{{\mathcal F}}

\newcommand{\CI}{{\mathcal I}}

\newcommand{\CL}{{\mathcal L}}

\newcommand{\CO}{{\mathcal O}}

\newcommand{\CT}{{\mathcal T}}

\newcommand{\CV}{{\mathcal V}}
\newcommand{\CW}{{\mathcal W}}
\newcommand{\CX}{{\mathcal X}}

\newcommand{\CZ}{{\mathcal Z}}

\newcommand{\FN}{{\mathfrak{N}}}
\newcommand{\FNfl}{\mathfrak{N}^{\mathrm{fl}}}
\newcommand{\FQfl}{\mathfrak{Q}^{\mathrm{fl}}}

\newcommand{\Fq}{{\mathfrak{q}}}
\newcommand{\FQ}{{\mathfrak{Q}}}

\newcommand{\Ft}{{\mathfrak{t}}}

\newcommand{\pt}{{\mathsf{p}}}
\newcommand{\ch}{{\mathrm{ch}}}

\DeclareMathOperator{\Hilb}{Hilb}

\DeclareFontFamily{OT1}{rsfs}{}
\DeclareFontShape{OT1}{rsfs}{n}{it}{<-> rsfs10}{}
\DeclareMathAlphabet{\curly}{OT1}{rsfs}{n}{it}
\renewcommand\hom{\curly H\!om}
\newcommand\ext{\curly Ext}
\newcommand\Ext{\operatorname{Ext}}
\newcommand\Hom{\operatorname{Hom}}
\newcommand{\p}{\mathbb{P}}

\newcommand\id{\operatorname{id}}

\newcommand\Spec{\operatorname{Spec}}

\newcommand{\Mbar}{{\overline M}}

\newcommand{\Coh}{\mathrm{Coh}}
\newcommand{\Pic}{\mathop{\rm Pic}\nolimits}
\newcommand{\PT}{\mathsf{PT}}
\newcommand{\DT}{\mathsf{DT}}
\newcommand{\GW}{\mathsf{GW}}
\newcommand{\Sym}{{\mathrm{Sym}}}

\newcommand{\ev}{{\mathrm{ev}}}
\newcommand{\DR}{\mathsf{DR}}

\newcommand{\Aut}{\operatorname{Aut}}

\newcommand{\pr}{\mathrm{pr}}

\newcommand{\vacuum}{v_{\varnothing}}

\newcommand{\vir}{\text{vir}}
\newcommand{\red}{\text{red}}

\newcommand{\Fr}{\mathrm{Fr}}

\newcommand\Quot{\operatorname{Quot}}
\newcommand\Cone{\operatorname{Cone}}
\newcommand\Ker{\operatorname{Ker}}
\newcommand\Coker{\operatorname{Coker}}

\newcommand\Pmain{\mathrm{P}^{\mathrm{main}}}
\newcommand\Prubber{\mathrm{P}^{\mathrm{rubber}}}

\newcommand{\rk}{\mathrm{rk}}

\newcommand{\taut}{\mathrm{taut}}
\newcommand{\tr}{\mathrm{tr}}

\newcommand{\td}{\mathrm{td}}

\newcommand\Perf{\operatorname{Perf}}

\begin{document}
\title{The multiple cover formula for $K3$ and abelian surfaces}

\dedicatory{To Jun Li on the occasion of his 65{th} birthday}
\date{\today}

\begin{abstract}
All reduced descendent Gromov-Witten invariants of $K3$ and abelian surfaces in primitive curve classes
can be calculated by the methods of \cite{BOPY,MPT}. To handle the imprimitive curve classes, a multiple cover formula was conjectured in \cite{ObPand} for $K3$ surfaces and in \cite{O_NLGW} for abelian surfaces.
We prove here that both descendent multiple cover formulas are implied by the conjectural families GW/PT correspondence 
for semipositive relative 3-folds with primary insertions. The implication is proven by 
showing that the multiple cover formula for $S$ can be recast as a property of an appropriate localization vertex for
the relative 3-fold Gromov-Witten theory of 
$(S\times \mathbb{P}^1/S_0 \cup S_\infty)$.  The families GW/PT correspondence
then transfers the multiple cover formula from the Gromov-Witten side to 
the stable pairs side where the formula is proven geometrically by studying cosections and applying
 universality properties.
Along the way, we prove a DT/PT correspondence for the reduced theories of $(S\times \mathbb{P}^1/S_0 \cup S_\infty)$
using the wallcrossing techniques of
Kuhn-Liu-Thimm \cite{KLT2,KLT}.
\end{abstract}
\author{Georg Oberdieck}
\author{Rahul Pandharipande}

\maketitle
\baselineskip=14pt

\setcounter{tocdepth}{1} 
\tableofcontents

\section{Introduction}
A {\em symplectic surface} $S$ is a nonsingular, irreducible, projective surface satisfying the symplectic condition
$\omega_S \cong \CO_S$.
By Kodaira's classification, $S$ is either a ${K3}$ or an abelian surface.
Let $\beta \in H_2(S,\BZ)$ be an effective{\footnote{Effective curve classes $\beta\in H_2(S,\mathbb{Z})$ here are {\em always} nonzero. If $\beta=0$, reduction of the Gromov-Witten theory is not possible, and the standard Gromov-Witten theory is well-known to be determined by Hodge integrals \cite{GetzlerP}.}} curve class.
For cohomology classes 
$$\gamma_1,\ldots,\gamma_n \in H^{\ast}(S,\mathbb{Q})$$ and integers $\ellt_1,\ldots,\ellt_n \geq 0$,
the {\em reduced Gromov-Witten invariants} of $S$ are defined by:
\begin{equation}\label{redgw}
\big\langle \tau_{\ellt_1}(\gamma_1) \cdots \tau_{\ellt_n}(\gamma_n) \big\rangle^{S}_{g,\beta}
=
\int_{[ \Mbar_{g,n}(S,\beta) ]^{\red}} \prod_{i=1}^{n} \ev_i^{\ast}(\gamma_i) \psi_i^{\ellt_i}\, .
\end{equation}
The integration here is over the {\em reduced} virtual class of the moduli space of genus $g$ degree $\beta$ stable maps.
We will always take the classes $\gamma_i$ to be homogeneous{\footnote{If $S$ is an abelian surface, $\deg_{\mathbb{C}}(\gamma_i)$ may be a half-integer.}} with 
$\gamma_i\in H^{2 \deg_\mathbb{C}(\gamma_i)}(S,\mathbb{Q})$.

The reduced virtual counts of curves
through points on symplectic surfaces were first introduced by Bryan-Leung \cite{BL} 
to prove the Yau-Zaslow formula \cite{YZ} for rational curves in primitive classes on $K3$ surfaces. A connection between the Yau-Zaslow
formula and
sheaf counting was made earlier by Beauville \cite{Beau}.
Over the past 30 years, reduced Gromov-Witten invariants have been studied
intensively for $K3$ surfaces \cite{BB,vIOP,
KMPS,LL1,LL2,MPT,O_K3xP1,O_K3,OP,PTKKV,PY} and for abelian surfaces \cite{Blomme2,Blomme1,Blomme3,Blomme4,CarBlomme2,BL2,BOPY,OP2}. Several different techniques have been used:
log degenerations, mirror symmetry, tropical counting, 
the double ramification cycle, and the connection to sheaf counting theories. See \cite{Dedieu, Pand-Survey, PT-Survey} for surveys of results.

The following basic structural property for reduced Gromov-Witten theory was conjectured in \cite[Conjecture C2]{ObPand} for $K3$ surfaces
and in \cite[Appendix A]{O_NLGW} for abelian surfaces.

\begin{conj}[Multiple Cover Formula] \label{conj:MCF}
Let $S$ be a symplectic surface, and let $\beta \in H_2(S,\BZ)$ be an effective curve class.
For every positive integer $k$ dividing $\beta$, let
\[ \varphi_k : H^{\ast}(S,\BQ) \to H^{\ast}(S_k,\BQ) \]
be a degree-preserving $\mathbb{Q}$-algebra isomorphism, where $S_k$ is a symplectic surface{\footnote{The symplectic surface $S_k$ must be of the same type as $S$ (since $K3$ surfaces have Euler characteristic 24 and abelian surfaces have Euler characteristic 0).}}, satisfying:
\begin{enumerate}
\item[\emph{(i)}] $\varphi_k(\pt)=\pt$, where $\pt\in H^4(S,\mathbb{Q})$ is the class of a point,
\item[\emph{(ii)}] $\varphi_k(\beta/k)$ is an effective primitive curve class.
\end{enumerate}
Then, we have
\[
\langle \tau_{\ellt_1}(\gamma_1) \cdots \tau_{\ellt_n}(\gamma_n) \rangle^{S}_{g,\beta}
=
\sum_{k | \beta} k^{2g-3 + \sum_{i=1}^n \deg_{\BC}(\gamma_i)}
\langle \tau_{\ellt_1}(\varphi_k(\gamma_1)) \cdots \tau_{\ellt_n}(\varphi_k(\gamma_n)) \rangle^{S_k}_{g,\varphi_k(\beta/k)}\, .
\]
If such a $\varphi_k$ does not exist, then the corresponding summand on the right is defined
to vanish.
\end{conj}

As reviewed in Section \ref{revsec}, the
$\mathbb{Q}$-algebra isomorphisms $\varphi_k$ can be made explicit using the geometry of
elliptically fibered surfaces in the $K3$ case and products of elliptic curves in the abelian surface case.
Conjecture \ref{conj:MCF} expresses {\em all} reduced descendent
Gromov-Witten invariants \eqref{redgw} of symplectic surfaces in terms of invariants with {\em primitive} curve classes. For primitive curve classes,
effective methods for computing  all
invariants \eqref{redgw} are known by \cite{MPT} in the $K3$ case and
\cite{BOPY} in the case of abelian surfaces. Conjecture \ref{conj:MCF} may therefore be viewed as the last step in 
a formally complete calculation of the reduced descendent Gromov-Witten invariants of symplectic surfaces.

Our main result here is a derivation of  Conjecture \ref{conj:MCF}  from the correspondence between Gromov-Witten and Pandharipande-Thomas theories
of 3-folds \cite{MNOP1, MNOP2, PT1}. Let 
$$S_0 \cup S_\infty \subset S\times \p^1$$
be the union of the fibers over $0\in \p^1$ and $\infty\in \p^1$, and let $(S\times \p^1, S_0 \cup S_\infty)$ be the relative threefold geometry.

\begin{thm} \label{mainn}
Let $S$ be a symplectic surface.
If the GW/PT correspondence holds for
the reduced theory of $(S \times \p^1,S_0\cup S_{\infty})$
with primary insertions,
then the multiple cover formula holds for $S$.
\end{thm}

The GW/PT correspondence for the reduced theories of 
$(S \times \p^1,S_{0} \cup S_{\infty})$ is a {\em families}
GW/PT correspondence for 1-parameter families of semipositive relative threefolds.   
Pardon \cite{Pardon} has proven the families GW/PT correspondence with primary insertions for
semipositive threefolds, see \cite{P-Beijing} for a discussion of families correspondences.
The promotion of Pardon's results to the semipositive relative  case is a natural direction of study.
The strategy of using the GW/PT correspondence and the geometry of stable pairs
to study the reduced Gromov-Witten theory of $K3$ surfaces has already appeared in special cases in 
\cite{MPT,PTKKV}. Theorem \ref{mainn} may be viewed as the strongest formulation of
these ideas for the reduced Gromov-Witten theory of $S$. The consequences of the conjectural
families GW/PT correspondence 
over the {\em full} moduli
space of $K3$ surfaces are discussed in Section \ref{modk3}.

The GW/PT correspondence for $(S \times \p^1,S_0\cup S_{\infty})$ with primary insertions in primitive curve classes for $K3$ surfaces $S$ 
was proven{\footnote{The case where $S$ is an abelian surface was not considered in \cite{Marked}, but the argument is
parallel.}}
in \cite{Marked}.
The converse of Theorem~\ref{mainn} is also true: {\em the multiple cover formula for $S$ implies the GW/PT  correspondence
for $(S \times \p^1,S_0\cup S_{\infty})$ with primary insertions for all curve classes on $S$}.
The proof of Theorem  \ref{mainn} and the discussion of the converse will be presented in Section \ref{gwt}.

Using a direct application of Pardon's
families GW/PT result in the semipositive absolute case, 
we obtain the following unconditional result.

\begin{thm}\label{hodgemc}
The multiple cover formula holds for all integrals of the form
\[ \langle \tau_{0}(\pt)^m \lambda_{g-m} \rangle^{S}_{g,\beta}\, , \]
where $\lambda_i = c_i(\BE)$ are the Chern classes of the Hodge bundle and $\pt \in H^4(S,\BQ)$ is the class of a point.
\end{thm}

The form of the multiple cover formula for Hodge classes (and other insertions of tautological classes
from the moduli of curves) related to Theorem \ref{hodgemc} will be explained in 
Section \ref{revsec}. The $m=0$ case of Theorem \ref{hodgemc} is the Katz-Klemm-Vafa formula proven in \cite{MPT,PTKKV}.

Theorem \ref{hodgemc}, in case $S$ is an abelian surface, has been proven by Blomme and
Carocci \cite{CarBlomme2}
by a very different approach entirely within Gromov-Witten theory: their proof uses degeneration together with
a beautiful refinement of the double ramification cycle formula \cite{CarBlomme1,JPPZ}. The results of
\cite{CarBlomme2}
cover more general integrands for the reduced theory of abelian surfaces than those appearing in Theorem \ref{hodgemc}.
In particular, the multiple cover
formula for all stationary
descendents is proven there in the abelian surface case. Hopefully, the methods of Blomme-Carocci in the future  will also be applicable to $K3$ surfaces.

\vspace{8pt}
\noindent{\bf Acknowledgements.}
Several of the ideas presented here began at the Oberwolfach workshop on
{\em Recent trends in algebraic geometry} in June 2025. 
We thank John Pardon for many discussions about his work on the GW/PT correspondence. We have also benefitted from conversations with Thomas Blomme, Pierrick Bousseau, 
Jim  Bryan, Francesca Carocci, Henry Liu, Davesh Maulik, Sam Molcho, Dragos Oprea,  Richard Thomas, and Qizheng Yin.  

The paper was written, in part, while R.P. was
visiting the University of Cambridge in March 2026. He thanks Dhruv Ranganathan and
the Geometry group at Cambridge for the wonderful mathematical environment (and for many discussions related to the Gromov-Witten theory of $K3$ surfaces).
G.O. was supported by the starting grant {\em Correspondences in enumerative geometry: Hilbert schemes, K3 surfaces and modular forms}, No 101041491 of the European Research Council.
R.P. was supported by SNF-200020-219369 and Swiss\-MAP.

\section{Multiple cover formulas} \label{revsec}

\subsection{Multiple cover formula for $K3$ surfaces}
Let $S$ be a $K3$ surface, and 
let $\beta \in H_2(S,\BZ)$ be an effective curve class.
The only deformation invariants of the pair $(S,\beta)$ are the self-intersection{\footnote{Since the intersection form $\langle , \rangle$ of a $K3$ surface is even, $\langle \beta,\beta \rangle$ is always an even integer.}} $\langle \beta,\beta \rangle\in 2\mathbb{Z}$ and the
divisibility $\mathsf{div}(\beta)\in \mathbb{Z}_{>0}$ of $\beta\in H_2(S,\mathbb{Z})$.\footnote{The claim follows from Eichler's criterion and the description on the moduli space of $K3$ surfaces via the period map, see \cite{HuyK3}.} Since the integrals
\begin{equation}\label{redgw2}
\big\langle \tau_{\ellt_1}(\gamma_1) \cdots \tau_{\ellt_n}(\gamma_n) \big\rangle^{S}_{g,\beta}
=
\int_{[ \Mbar_{g,n}(S,\beta) ]^{\red}} \prod_{i=1}^{n} \ev_i^{\ast}(\gamma_i) \psi_i^{\ellt_i}\, 
\end{equation}
are deformation invariants
of the pair $(S,\beta)$,
we have the freedom to choose a convenient $K3$ surface for which to study the reduced Gromov-Witten theory.

Starting from the original paper of Bryan-Leung \cite{BL}, the standard choice has been to take an {\em elliptically fibered} $K3$ surface with a section:
$$\pi:S_E\rightarrow \mathbb{P}^1\, , \ \ \mathsf{s}: \mathbb{P}^1 \rightarrow S_E\, .$$
The generic such $K3$ surface $S_E$ has a Picard lattice
generated by the section class $\mathsf{s}$ and the fiber class $\mathsf{f}$ of $\pi$:
$$\mathsf{Pic}(S_E)= \mathbb{Z} \mathsf{s} + \mathbb{Z} \mathsf{f}\, , \ \ \langle \mathsf{s},\mathsf{s}\rangle = -2\, , \ \ \langle \mathsf{s},\mathsf{f}\rangle = 1\,, \ \ 
\langle \mathsf{f},\mathsf{f}\rangle = 0\, .
$$
For $m\in 2\mathbb{Z}$ and $r\in \mathbb{Z}_{>0}$,
the pairs $(S_E, \beta_{m,r})$, where $$\beta_{m,r}= r\mathsf{s}+r\left(\frac{m}{2}+1\right)\mathsf{f}\in \mathsf{Pic}(S_E)\, ,\ \ \langle \beta_{m,r},\beta_{m,r} \rangle= r^2m\,, \ \ 
\mathsf{div}(\beta_{m,r})=r\, ,$$
cover all deformation classes of pairs $(S,\beta)$ with $\beta$ effective.
Therefore, all reduced Gromov-Witten invariants of $K3$ surfaces can be calculated on the single surface $S_E$.

For $m\in 2\mathbb{Z}$ and $r\in \mathbb{Z}_{>0}$,
there is an isomorphism of $\mathbb{Q}$-algebras
$$\varphi_{m,r}: H^*(S_E,\mathbb{Q}) \rightarrow 
H^*(S_E,\mathbb{Q})$$
defined by the following rules:
\begin{itemize}
    \item [$\bullet$] In degrees 0 and 4, $\varphi$ restricts to the identity,
    $$\varphi_{m,r}|_{H^0(S_E,\mathbb{Q})}= \mathsf{Id}|_{H^0(S_E,\mathbb{Q})}\,,  \ \ \
    \varphi_{m,r}|_{H^4(S_E,\mathbb{Q})}= \mathsf{Id}|_{H^4(S_E,\mathbb{Q})}\, .
    $$
    \item[$\bullet$] In degree 2,  let $V\subset H^2(S_E,\mathbb{Z})$ be the orthogonal complement of $\mathsf{Pic}(S_E)$, and let
    $$H^2(S_E,\mathbb{Q})= \mathsf{Pic}(S_E)\otimes_\mathbb{Z} \mathbb{Q} \, \oplus\,  V \otimes_\mathbb{Z} \mathbb{Q}\, .$$
    The isomorphism $\varphi_{m,r}$ satisfies
$$\varphi_{m,r}(\beta_{m,r})= \mathsf{s} + \left(\frac{r^2m}{2}+1\right) \mathsf{f}\, , \ \ 
\varphi_{m,r}(\mathsf{f})= r\mathsf{f}\,, \ \ \varphi_{m,r}|_{V\otimes_\mathbb{Z} \mathbb{Q}} = \mathsf{Id}|_{V\otimes_\mathbb{Z} \mathbb{Q}}\, .$$
\end{itemize}

If $m\geq 0$ or if $(m,r)=(-2,1)$, then $\varphi_{m,r}(\beta_{m,r})$
is an effective curve class. If $$m<0 \ \ \text{and} \ \ (m,r)\neq (-2,1)\,,$$ then 
$\varphi_{m,r}(\beta_{m,r})$ is {\em not} effective.

The multiple cover formula of Conjecture \ref{conj:MCF} can then be written explicitly as:
\begin{eqnarray*}
\langle \tau_{\ellt_1}(\gamma_1) \cdots \tau_{\ellt_n}(\gamma_n) \rangle^{S_E}_{g,\beta_{m,r}}
& =
&\sum_{k | r} k^{2g-3 + \sum_i \deg_{\BC}(\gamma_i)}
\langle \tau_{\ellt_1}(\varphi_{m,r/k}(\gamma_1)) \cdots \tau_{\ellt_n}(\varphi_{m,r/k}(\gamma_n)) \rangle^{S_E}_{g,\varphi_{m,r/k}(\beta_{m,r/k})}\, \\
& =
&\sum_{k | r} k^{2g-3 + \sum_i \deg_{\BC}(\gamma_i)}
\langle \tau_{\ellt_1}(\varphi_{m,r/k}(\gamma_1)) \cdots \tau_{\ellt_n}(\varphi_{m,r/k}(\gamma_n)) \rangle^{S_E}_{g,\beta_{(r/k)^2m,1}}\,,
\end{eqnarray*}
where summands on the right side with curve classes $\beta_{(r/k)^2m,1}$
are defined to vanish if
$m<0$ and $(m,r/k)\neq (-2,1)$.

\subsection{Multiple cover formula for abelian surfaces}
Let $S$ be an abelian surface, and 
let $\beta \in H_2(S,\BZ)$ be an effective curve class.
The only deformation invariants of the pair $(S,\beta)$ are the self-intersection{\footnote{Since the intersection form $\langle , \rangle$ of an abelian surface is even, $\langle \beta,\beta \rangle$ is always an even integer.}} $\langle \beta,\beta \rangle\in 2\mathbb{Z}$ and the
divisibility $\mathsf{div}(\beta)\in \mathbb{Z}_{>0}$ of $\beta\in H_2(S,\mathbb{Z})$, see 
\cite[Section 1]{BOPY}.\footnote{The discussion in \cite{BOPY} is in terms of type $(d_1,d_2)$ of a curve class. The type of a curve class $\beta$ is determined by the self-intersection and divisibility by $d_1 = \mathsf{div}(\beta)$, $2 d_1 d_2 = \langle \beta, \beta \rangle$.} Since the integrals
\begin{equation}\label{redgw2Ab}
\big\langle \tau_{\ellt_1}(\gamma_1) \cdots \tau_{\ellt_n}(\gamma_n) \big\rangle^{S}_{g,\beta}
=
\int_{[ \Mbar_{g,n}(S,\beta) ]^{\red}} \prod_{i=1}^{n} \ev_i^{\ast}(\gamma_i) \psi_i^{\ellt_i}\, 
\end{equation}
are deformation invariants
of the pair $(S,\beta)$,
we can again choose a convenient abelian surface for which to study the reduced Gromov-Witten theory.

The standard choice in the abelian surface case 
is the product $A=E_1 \times E_2$
of two elliptic curves, see \cite[Appendix A]{O_NLGW}. 
The generic such product $A$ has a Neron-Severi lattice
generated by the two fiber classes:
$$\mathsf{NS}(A)= \mathbb{Z} \mathsf{f_1} + \mathbb{Z} \mathsf{f_2}\, , \ \ \langle \mathsf{f}_i,\mathsf{f}_i\rangle = 0\, , \ \ \langle \mathsf{f}_1,\mathsf{f}_2\rangle = 1\,.
$$
For $m\in 2\mathbb{Z}_{\geq 0}$ and $r\in \mathbb{Z}_{>0}$,
the pairs $(A, \beta_{m,r})$, where $$\beta_{m,r}= r\mathsf{f}_1+\frac{rm}{2}\mathsf{f}_2\in \mathsf{NS}(A)\, ,\ \ \langle \beta_{m,r},\beta_{m,r} \rangle= r^2m\,, \ \ 
\mathsf{div}(\beta_{m,r})=r\, ,$$
cover all deformation classes of pairs $(S,\beta)$
with $\beta$ effective{\footnote{An effective curve on an abelian surface has non-negative self-intersection.}}.
Therefore, all reduced Gromov-Witten invariants of abelian surfaces can be calculated on $A$.

For $r\in \mathbb{Z}_{>0}$,
there is an isomorphism of $\mathbb{Q}$-algebras
$$\varphi_{r}: H^*(A,\mathbb{Q}) \rightarrow 
H^*(A,\mathbb{Q})$$
defined as follows.
Let 
$$H^1(A,\mathbb{Q})= V_1 \oplus V_2\,, \ \ \
V_1 = \mathbb{Q}dx_1 + \mathbb{Q} dy_1\,, \ \ \ 
V_2 = \mathbb{Q}dx_2 + \mathbb{Q} dy_2\, ,$$
where $x_i,y_i$ are flat coordinates on $E_i$.
In particular $\mathsf{f}_i = dx_i \wedge dy_i$.
Define 
$$\varphi_{r}: H^1(A,\mathbb{Q}) \rightarrow 
H^1(A,\mathbb{Q})$$
on generators by
$$\varphi_r(dx_1)=  dx_1\, ,\ \ 
\varphi_r(dy_1)= \frac{1}{{r}} dy_1\, ,\ \
\varphi_r(dx_2)=  dx_2\, ,\ \ 
\varphi_r(dy_2)= r dy_2
\,. $$
Since $H^*(A,\mathbb{Q})= \Lambda^* H^1(A,\mathbb{Q})$, we can extend $\varphi_r$ canonically to an isomorphism
of $\mathbb{Q}$-algebras satisfying $\varphi_r(\mathsf{p})=\mathsf{p}$.
Since $m\geq 0$, $$\varphi_{r}(\beta_{m,r})= \mathsf{f}_1+ \frac{r^2m}{2}\mathsf{f}_2$$
is an effective curve class. 

The multiple cover formula of Conjecture \ref{conj:MCF} can then be written explicitly as:
\begin{eqnarray*}
\langle \tau_{\ellt_1}(\gamma_1) \cdots \tau_{\ellt_n}(\gamma_n) \rangle^{A}_{g,\beta_{m,r}}
& = &\sum_{k | r} k^{2g-3 + \sum_i \deg_{\BC}(\gamma_i)}
\langle \tau_{\ellt_1}(\varphi_{r/k}(\gamma_1)) \cdots \tau_{\ellt_n}(\varphi_{r/k}(\gamma_n)) \rangle^{A}_{g,\varphi_{r/k}(\beta_{m,r/k})}\, \\
& =
&\sum_{k | r} k^{2g-3 + \sum_i \deg_{\BC}(\gamma_i)}
\langle \tau_{\ellt_1}(\varphi_{r/k}(\gamma_1)) \cdots \tau_{\ellt_n}(\varphi_{r/k}(\gamma_n)) \rangle^{A}_{g,\beta_{(r/k)^2m,1}}\,,
\end{eqnarray*}
for $m\in 2\mathbb{Z}_{\geq 0}$ and $r\in \mathbb{Z}_{>0}$.

\section{Cycle-valued multiple cover formulas}

\subsection{Insertions from moduli space} \label{secdefint}
Let $S$ be a symplectic surface, and 
let $\beta \in H_2(S,\BZ)$ be a (non-zero) effective curve class.
If $2g-2+n>0$,  there is a morphism
$$\epsilon:\overline{M}_{g,n}(S,\beta)\rightarrow \overline{M}_{g,n}\, $$
determined by the domain
of the stable map.
Given a 
class{\footnote{The cohomology theory of Deligne-Mumford stacks will always be taken with $\mathbb{Q}$-coefficients.}}
$\Theta \in H^*(\overline{M}_{g,n})$, we can define more general integrals:
\begin{equation}\label{redgwto}
\big\langle \tau_{\ellt_1}(\gamma_1) \cdots \tau_{\ellt_n}(\gamma_n)\, \Theta \big\rangle^{S}_{g,\beta}
=
\int_{[ \Mbar_{g,n}(S,\beta) ]^{\red}} \prod_{i=1}^{n} \ev_i^{\ast}(\gamma_i) \psi_i^{\ellt_i}\, \cdot \, \epsilon^*(\Theta)\, .
\end{equation}
The class $\Theta$ is conjectured to be a spectator in the multiple cover formula.

\begin{conj}
\label{conj:Taut}
Let $S$ be a symplectic surface, and let $\beta \in H_2(S,\BZ)$ be a (non-zero) effective curve class.
For every positive integer $k$ dividing $\beta$, let
\[ \varphi_k : H^{\ast}(S,\BQ) \to H^{\ast}(S_k,\BQ) \]
be a degree-preserving $\mathbb{Q}$-algebra isomorphism, where $S_k$ is a symplectic surface, satisfying:
\begin{enumerate}
\item[\emph{(i)}] $\varphi_k(\pt)=\pt$, where $\pt\in H^4(S,\mathbb{Q})$ is the class of a point,
\item[\emph{(ii)}] $\varphi_k(\beta/k)$ is an effective primitive curve class.
\end{enumerate}
Then, we have
\[
\langle \tau_{\ellt_1}(\gamma_1) \cdots \tau_{\ellt_n}(\gamma_n)\, \Theta \rangle^{S}_{g,\beta}
=
\sum_{k | \beta} k^{2g-3 + \sum_i \deg_{\BC}(\gamma_i)}
\langle \tau_{\ellt_1}(\varphi_k(\gamma_1)) \cdots \tau_{\ellt_n}(\varphi_k(\gamma_n)) \,
\Theta \rangle^{S_k}_{g,\varphi_k(\beta/k)} \, .
\]
If such a $\varphi_k$ does not exist, then the corresponding summand on the right is defined
to vanish.
\end{conj}

\subsection{Cycle-valued reduced Gromov-Witten theory}
Conjecture \ref{conj:Taut} is formulated for integrals. 
An {\em equivalent} multiple cover formula can be formulated for cycle-valued reduced
Gromov-Witten classes.
While we are principally interested  here in numerical invariants, several basic compatibilities are easier to see on the cycle level.

For $2g-2+n>0$ and $\gamma_1,\ldots, \gamma_n \in H^*(S,\BQ)$,
the Gromov-Witten cycle classes of $S$ are defined by
\[ \CC_{g,\beta}(\gamma_1,\ldots,\gamma_n) = \epsilon_{\ast}\left( \prod_{i=1}^{n} \ev_i^{\ast}(\gamma_i) \cap [\Mbar_{g,n}(S,\beta)]^{\red} \right)\in H^*(\Mbar_{g,n})\, . \]
Since $\Mbar_{g,n}(S,\beta)$ has reduced virtual dimension $g+n$, the (complex) degree is
\[ \deg_{\BC} \CC_{g,\beta}(\gamma_1,\ldots,\gamma_n) = 3g-3+n - (g+n-\sum_i \deg_{\BC}(\gamma_i)) = 2g-3+\sum_i \deg_{\BC}(\gamma_i)\, , \]
which is precisely the exponent in the multiple cover formula!

The case where the curve class{\footnote{The symbol $\beta$ is reserved for a (nonzero) effective curve class.}} is zero will appear in our study.
The standard virtual class may then be non-zero (as the standard cosection construction does not apply), and we define
\[
\CC_{g,0}(\gamma_1,\ldots,\gamma_n) = \tau_{\ast}\left( \prod_{i=1}^{n} \ev_i^{\ast}(\gamma_i) \cap [\Mbar_{g,n}(S,0)]^{\vir} \right)\in H^*(\Mbar_{g,n})\, . \]

\begin{lemma} \label{lemma:observation}
For all $g$ and $\gamma_1,\ldots,\gamma_n\in H^*(S,\BQ)$,
the class $\CC_{g,0}(\gamma_1,\ldots,\gamma_n)$ either vanishes or has cohomological degree zero.
\end{lemma}
\begin{proof}
We have the explicit description
\begin{equation} \label{Dfcsdfre3r}
[ \Mbar_{g,n}(S,0) ]^{\vir}
=
\begin{cases}
[ \Mbar_{0,n} \times S] & \text{ if } g=0 \\
c_2(S) \cap [ \Mbar_{1,n} \times S] & \text{ if } g=1 \\
0 & \text{ if } g \geq 2,
\end{cases}
\end{equation}
see \cite[Section 2.2]{GetzlerP}.
After capping with the insertions $\gamma_i$ and pushing forward to $\Mbar_{g,n}$, we either obtain 0 {\em or} a multiple of the fundamental class of the moduli space of curves.
\end{proof} 

A cycle-valued  multiple cover formula conjecture is easily formulated. We follow
the conventions of Conjectures \ref{conj:MCF} and \ref{conj:Taut} for $\varphi_k$.

\begin{conj}[Cycle-valued Multiple Cover Formula] \label{conj:cycle}
We have
\begin{equation} \label{cycleMCF}
\CC_{g,\beta}(\gamma_1,\ldots,\gamma_n)
=
\sum_{k|\beta} k^{d} \, \CC_{g,\varphi_k(\beta/k)}( \varphi_k(\gamma_1), \ldots, \varphi_k (\gamma_n)) \in H^{2d}(\Mbar_{g,n})\, ,
\end{equation}
where $d=\deg_{\mathbb{C}} \CC_{g,\beta}(\gamma_1,\ldots,\gamma_n)$.
\end{conj}

By Poincar\'e duality on $\overline{M}_{g,n}$, Conjecture \ref{conj:Taut} is
{\em equivalent} to the multiple cover formula of Conjecture \ref{conj:cycle} for cycle-valued reduced Gromov-Witten classes.


\subsection{Tautological classes}
The interaction of the multiple cover formula with tautological 
classes{\footnote{See \cite{Pan-Calculus} for an introduction to tautological classes on the moduli spaces of curves.}} 
$$RH^*(\Mbar_{g,n})\subset H^*(\Mbar_{g,n})$$  has several special properties.

\begin{defn}
{\em Let
$\Theta \in RH^{\ast}(\Mbar_{g,n})$ be a tautological class.
The cycle-valued multiple cover formula (MCF) {\em holds numerically} for $\Theta$ if the equality
\eqref{cycleMCF}  holds for all $\gamma_1,\ldots,\gamma_n \in H^*(S,\BQ)$ after intersecting both sides with $\Theta$ and taking the integral over $\Mbar_{g,n}$.}
\end{defn}


\begin{prop} \label{prop:psi to all taut}
If the cycle-valued MCF holds numerically for all tautological classes of the form $$\Theta = \prod_{i} \psi_i^{\ellt_i}\,, $$
then the cycle-valued MCF holds numerically for all tautological classes.
\end{prop}
\begin{proof}
Let $\gamma=(\gamma_1,\ldots,\gamma_n)$.
We denote the standard gluing maps by
$$\iota : \Mbar_{g-1,n+2} \to \Mbar_{g,n}\, , \ \ \ 
\xi : \Mbar_{g_1,n_1 + 1} \times \Mbar_{g_2,n_2+1} \to \Mbar_{g,n}\, .$$ 
The splitting formulas for the reduced virtual class take the form
\begin{align*}
\iota^{\ast}\CC_{g,\beta}(\gamma) & = \CC_{g-1,\beta}(\gamma, \Delta_S)\, , \\
\xi^{\ast}\CC_{g,\beta}(\gamma) & = \sum_i \CC_{g_1,\beta}( \gamma_{1},\ldots,\gamma_{n_1}, \delta_i)
\boxtimes \CC_{g_2,0}(\delta_i^{\vee}, \gamma_{n_1+1}, \ldots, \gamma_n) \\
& \ \ \ \ + \sum_i \CC_{g_1,0}( \gamma_{1},\ldots,\gamma_{n_1}, \delta_i) 
\boxtimes \CC_{g_2,\beta}(\delta_i^{\vee}, \gamma_{n_1+1}, \ldots, \gamma_n)\, ,
\end{align*}
where 
$\sum_i \delta_i \boxtimes \delta_i^{\vee}$ is a K\"unneth decomposition of the class $\Delta_S \in H^{\ast}(S \times S)$ of the diagonal.{\footnote{In case $S$ is an abelian surface, care must be taken in the splitting formula to keep track of signs related to odd insertions.}}

Assuming the hypothesis of Proposition \ref{prop:psi to all taut},
we will prove, by induction on $g$ and the number of markings $n$, the equality
\[ \CC_{g,\beta}(\gamma) \cdot \Theta = \sum_{k|\beta} k^d \CC_{g,\varphi_k(\beta/k)}(\varphi_k (\gamma)) \cdot \Theta\, , \]
where $d=\deg \CC_{g,\beta}(\gamma)$ and
$\Theta \in RH^*(\Mbar_{g,n})$.
We have used the notation  
$$
\varphi_k (\gamma)= ( \varphi_k(\gamma_1), \ldots, \varphi_k (\gamma_n))\, , \ \ \ 
\alpha \cdot \alpha' = \int_{\Mbar_{g,n}} \alpha \cup \alpha'$$ 
for brevity.

\vspace{5pt}
\noindent $\bullet$
If $\Theta = \iota_{\ast}(\Theta')$ for some $\Theta' \in RH^{\ast}(\Mbar_{g-1,n+2})$, then
\[
\CC_{g,\beta}(\gamma)\cdot \Theta =   \CC_{g-1,\beta}(\gamma, \Delta_S)\cdot \Theta'\, .
\]
By induction, the right side satisfies the multiple cover formula with exponent
$$\deg_{\mathbb{C}} \CC_{g-1}(\gamma,\Delta_S) = \deg_{\mathbb{C}} \CC_{g,\beta}(\gamma)\, .$$

\vspace{5pt}
\noindent $\bullet$ If $\Theta = \xi_{\ast}(\Theta_1 \boxtimes \Theta_2)$, then
$\CC_{g,\beta}\cdot \Theta$ is a linear combination of terms
\[ \Big( \CC_{g_1,\beta}( \gamma_{1},\ldots,\gamma_{n_1}, \delta_i)\cdot \Theta_1 \Big) 
\cdot \Big(  \CC_{g_2,0}(\delta_i^{\vee}, \gamma_{n_1+1}, \ldots, \gamma_n)\cdot \Theta_2 \Big)\, , \]
\[ \Big( \CC_{g_1,0}( \gamma_{1},\ldots,\gamma_{n_1}, \delta_i)\cdot \Theta_1 \Big) 
\cdot \Big(  \CC_{g_2,\beta}(\delta_i^{\vee}, \gamma_{n_1+1}, \ldots, \gamma_n)\cdot \Theta_2 \Big)\, . \]
By induction, the first factor (in the first product) satisfies the multiple cover with exponent
$$\deg_{\mathbb{C}}( \CC_{g_1,\beta}( \gamma_{1},\ldots,\gamma_{n_1}, \delta_i))\, .$$
By Lemma~\ref{lemma:observation}, the  class $\CC_{g_2,0}(\delta_i^{\vee}, \gamma_{n_1+1}, \ldots,\gamma_n)$, when nonzero, has degree $0$.
We see
\[ \deg_{\mathbb{C}}( \CC_{g_1,\beta}( \gamma_{1},\ldots,\gamma_{n_1}, \delta_i)) = 
\deg_{\mathbb{C}}\left( 
\CC_{g_1,\beta}( \gamma_{1},\ldots,\gamma_{n_1}, \delta_i) \boxtimes \CC_{g_2,0}(\delta_i^{\vee}, \gamma_{n_1+1}, \ldots, \gamma_n) \right)
= 
\deg_{\mathbb{C}} \CC_{g,\beta}(\gamma). \]
Hence, the exponent in the MCF formula matches again. The analysis is the same
for the second product.

\vspace{5pt}

The above analysis proves the induction step for all tautological classes $\Theta$ pushed forward from the boundary of $\Mbar_{g,n}$. We must now prove the induction step for
the interior tautological classes $\Theta$.

The tautological ring of the interior $M_{g,n}$ is generated by $\psi$ and $\kappa$-classes. The
$\kappa$-classes are defined by pushforward:
$$\kappa_r = p_{\ast}(\psi_{n+1}^{r+1})\,,\  \ \ \ p : \Mbar_{g,n+1} \to \Mbar_{g,n}\, ,
$$
where $p$ is the morphism forgetting the last point.
We have
\begin{gather*} p^{\ast}(\psi_i) = \psi_i - D_{i,n+1}\, , \\
p^{\ast}(\kappa_r) = \kappa_r - \psi_{n+1}^{r}\, , \end{gather*}
where $D_{i,n+1}$ is the boundary divisor parameterizing curves with a rational tail{\footnote{$D_{i,n+1}$ is the image of an appropriate gluing map $\xi : \Mbar_{g,n} \times \Mbar_{0,3} \to \Mbar_{g,n+1}$.}}
carrying the markings $\{i,n+1\}$.
We also have the vanishing 
\[ \psi_{n+1} \cdot D_{i,n+1}=0\, . \]

Let $\Theta = \kappa_{r_1} \ldots \kappa_{r_{m}} \prod_{i} \psi^{\ellt_i}$ be an arbitrary monomial in $\psi$ and $\kappa$-classes.
By the  projection formula for $\kappa_{r_m}$, we obtain
\[
\int_{\Mbar_{g,n}}  \CC_{g,\beta}(\gamma)\cdot \Theta = 
\int_{\Mbar_{g,n+1}} \CC_{g,\beta}(\gamma,1) \cdot 
\psi_{n+1}^{r_{m}+1} \cdot \prod_{i=1}^{n}  \psi_i^{\ellt_i} \cdot \prod_{i=1}^{m-1} (\kappa_{r_i} - \psi_{n+1}^{r_i})\, .
\]
Repeating the construction for the other $\kappa$-classes, we find 
\begin{equation}
\label{j12}
\int_{\Mbar_{g,n}} \CC_{g,\beta}(\gamma) \cdot \Theta = \int_{\Mbar_{g,n+{m}}} \mathsf{P} \cdot \CC_{g,\beta}(\gamma,1^{\times \ell}) 
\end{equation}
where $\mathsf{P}$ is  polynomial in the $\psi$-classes.
By our assumption, the right side of \eqref{j12} satisfies the MCF with exponent 
$\deg_{\mathbb{C}} \CC_{g,\beta}(\gamma,1^{\times \ell}) = \deg_{\mathbb{C}} \CC_{g,\beta}(\gamma)$.
\end{proof}

Let $2g-2+n>0$ and
$\Theta \in RH^*(\overline{M}_{g,n})$. We have defined, in Section \ref{secdefint},  the integrals:
\begin{equation}\label{redgwt}
\big\langle \tau_{\ellt_1}(\gamma_1) \cdots \tau_{\ellt_n}(\gamma_n)\, \Theta \big\rangle^{S}_{g,\beta}
=
\int_{[ \Mbar_{g,n}(S,\beta) ]^{\red}} \prod_{i=1}^{n} \ev_i^{\ast}(\gamma_i) \psi_i^{\ellt_i}\, \cdot \, \epsilon^*(\Theta)\, ,
\end{equation}
where the $\psi_i$ are the cotangent line classes on $\Mbar_{g,n}(S,\beta)$.
If all $\ellt_i=0$, we write
$\langle \gamma_1,\ldots,\gamma_n, \Theta \rangle^{S}_{g,\beta}$ for the invariant. If $2g-2+n\leq 0$, 
then we only consider the case $\Theta=1$ in the
integral \eqref{redgwt} by convention.



\begin{lemma} \label{lemma:desc vs ansc}
Let $2g-2+n>0$ and $\ellt_1\geq 1$, then
\begin{multline*} 
\langle \tau_{\ellt_1}(\gamma_1) \cdots \tau_{\ellt_n}(\gamma_n)\, \Theta \rangle^{S}_{g,\beta}
=
\langle  \tau_{\ellt_1-1}(\gamma_1) \cdots \tau_{\ellt_n}(\gamma_n)\, \Theta \cdot \psi_1 \rangle^{S}_{g,\beta} \\
+ \sum_i \langle \tau_{\ellt_1-1}(\gamma_1) \tau_0(\delta_i)  \rangle^{S}_{0,\beta} \cdot \langle  \tau_0(\delta_i^{\vee}) \tau_{\ellt_2}(\gamma_2) \cdots \tau_{\ellt_n}(\gamma_n)\, \Theta \rangle^S_{g,0} \, ,
\end{multline*}
where $\sum_i \delta_i \boxtimes \delta_i^{\vee} \in H^{\ast}(S \times S)$ is a K\"unneth decomposition of the class of the diagonal.
\end{lemma}
\begin{proof}
We have $\psi_1 = \epsilon^{\ast}(\psi_1) + \sum_{\beta_0 > 0} D_{1,\beta_0}$ where $D_{i,\beta_0}$ is the (virtual) Cartier divisor parameterizing stable maps
$$f : C_0 \cup C_1 \to S$$ where $C_0$ is a genus $0$ curve carrying the marked point $i$ and satisfying $f_{\ast}[C_0] = \beta_0$.
The claim then follows by the splitting formula for the reduced class, which forces $\beta_0=\beta$ and $f_{\ast}[C_1] = 0$.
\end{proof}

We say that {\em the MCF holds for $\Theta$-tautological invariants} \eqref{redgwt} if
\[
\langle \tau_{\ellt_1}(\gamma_1) \cdots \tau_{\ellt_n}(\gamma_n)\, \Theta \rangle^{S}_{g,\beta}
=
\sum_{k | \beta} k^{2g-3 + \sum_i \deg_{\BC}(\gamma_i)}
\langle  \tau_{\ellt_1}(\varphi_k(\gamma_1)) \cdots \tau_{\ellt_n}(\varphi_k(\gamma_n))\, \Theta \rangle^{S_k}_{g,\varphi_k(\beta/k)}
\]
for $\Theta\in RH^*(\Mbar_{g,n})$ for $2g-2+n>0$ and $\Theta=1$ for $2g-2+n\leq 0$.

\begin{prop} \label{prop:equivalent MCF}
The following are equivalent:
\begin{enumerate}
\item[(i)] The MCF holds for all invariants of the form
$\langle \tau_{\ellt_1}(\gamma_1) \cdots \tau_{\ellt_n}(\gamma_n) \rangle^{S}_{g,\beta}$.
\item[(ii)] The MCF holds for all $\Theta$-tautological invariants of the form
$\langle \tau_{\ellt_1}(\gamma_1) \cdots \tau_{\ellt_n}(\gamma_n)\, \Theta \rangle^{S}_{g,\beta}$ .
\item[(iii)] The MCF holds for all $\Theta$-tautological invariants of the form
$\langle \tau_0(\gamma_1)\cdots \tau_0(\gamma_n)\, \Theta
\rangle^{S}_{g,\beta}$.
\end{enumerate} 
\end{prop}

\noindent If any of the properties (i)-(iii) of Proposition 
\ref{prop:equivalent MCF} is satisfied, we say that {\em the MCF holds for $S$}.

\begin{proof}
Clearly (ii) implies (i).
To show that (i) implies (iii), we can assume $2g-2+n > 0$.
We have
\[ \langle \tau_0(\gamma_1)\cdots \tau_0(\gamma_n)\,  \Theta \rangle^{S}_{g,\beta} = \int_{\Mbar_{g,n}}  \CC_{g,\beta}(\gamma)\cdot \Theta\, . \]
By Proposition~\ref{prop:psi to all taut}, we need only  prove (iii) for $\Theta = \prod_i \psi_i^{k_i}$.
Using Lemma~\ref{lemma:desc vs ansc}, we have 
$$\langle \tau_0(\gamma_1)\cdots \tau_0(\gamma_n)\, \prod_i \psi_i^{\ellt_i} \rangle^S_{g,\beta} = \langle \prod_i \tau_{\ellt_i}(\gamma_i) \rangle^S_{g,\beta} \, + \, \text{corrections terms,}$$
where the correction terms
satisfy the MCF 
with the required matching exponent. The required
exponent calculus is given in  Lemma~\ref{lemma:mc exponent degree}.

For the direction (iii) to (ii), we first show that (iii) implies the MCF for $\langle \tau_{\ellt}(\gamma_1) \tau_0(\gamma_2) \rangle_{0,\beta}^{S}$ with $\ellt>0$. By dimension reasons, the possible cases for the insertions are 
$$\tau_1(\alpha) \tau_0(\alpha')\, ,\ \  \tau_1(D) \tau_0(1)\, ,\ \ \tau_1(1) \tau_0(D)\, \ \ \text{or}\ \ \tau_2(1) \tau_0(1)\, ,$$
where $D \in H^2(S)$ and $\alpha,\alpha'\in H^1(S)$.
By using the string, dilaton, and divisor equations, we immediately reduce to (iii).
Implication (iii) to (ii) then follows from Lemma~\ref{lemma:desc vs ansc}. \end{proof}

We record a few elementary properties of Gromov-Witten invariants of $S$ in curve class 0 which
will also appear later in the proof of Theorem \ref{mainn}.

\begin{lemma} We have\label{lemma:mc exponent degree}
\begin{enumerate}
\item[\em{(a)}] If $\langle \tau_{\ellt_1}(\gamma_1) \cdots \tau_{\ellt_n}(\gamma_n)\, \Theta \rangle^S_{g,0} \neq 0$, then
$2(g-1) + \sum_i \deg_\BC(\gamma_i) = 0$.\\
\item[\em{(b)}] Assume we have the nonvanishing:
\begin{equation} \label{abc} \langle   \prod_{j} \tau_{\ellt_{0j}}(\gamma_{0j})\, \Theta_0 \rangle^{S}_{g_0,\beta}
\cdot 
\prod_{i=1}^{r} \langle  \prod_{j} \tau_{\ellt_{ij}}(\gamma_{ij}) \, \Theta_i \rangle^{S}_{g_i,0} \neq 0\, . \end{equation}
Let $g=g_0 + \ldots + g_r - r$ be the total genus.\footnote{If $C = C_0 \sqcup \ldots \sqcup C_r$, then $g(C) = \sum_i g(C_i) - r$.} Then,
\[ 2g-3+\sum_{i,j} \deg_\BC(\gamma_{ij}) = 2g_0 - 3 + \sum_{j} \deg_\BC(\gamma_{0j})\, . \]
\end{enumerate}
\end{lemma} 
\begin{proof}
The properties (a) and (b) follow directly from  the explicit description of the virtual class of the space $[\Mbar_{g,n}(S,0)]^{\vir}$ of constant maps given 
in Lemma \ref{lemma:observation}.
\end{proof}

\subsection{An example} \label{subsec:an example}
For a symplectic surface $S$,
the MCF holds for the integrals 
$$\Big\langle \tau_0(\mathsf{p})^{m} \lambda_{g-m}  \Big\rangle^{S}_{g, \beta}=
\int_{[\Mbar_{g,m}(S,\beta)]^{\red}} 
\prod_{i=1}^m \text{ev}_i^*(\mathsf{p}) \cdot \lambda_{g-m}$$
by Theorem \ref{hodgemc} proven in Section \ref{pr-hodgemc}.

A complete evaluation in the primitive case was obtained for $K3$ surfaces{\footnote{In the $m=g$ case, the evaluation of
$\Big\langle \tau_0(\mathsf{p})^{g}   \Big\rangle^{S_E}_{g, \beta_{2h-2,1}}$
goes back to the paper of Bryan-Leung \cite{BL}.}} in \cite[Theorem 3]{MPT},
\[ \sum_{g= 0}^\infty \sum_{h = 0}^\infty
\Big\langle \tau_0(\mathsf{p})^{m}\lambda_{g-m} \Big\rangle^{S_E}_{g, \beta_{2h-2,1}} z^{2g-2} q^{h-1}
=
\frac{1}{\Theta(z,q)^2 \Delta(q)} S(z,q)^m, \]
and for abelian surfaces in \cite[Theorem 2 and Proposition 2]{BOPY},
\[
\sum_{g=0}^\infty \sum_{h=0}^\infty
\Big\langle \tau_0(\mathsf{p})^{m}\lambda_{g-m} \Big\rangle^{A}_{g, \beta_{2h-2,1}} z^{2g-2} q^{h-1}
= m S(z,q)^{m-1},
\]
with the following modular and Jacobi forms:
\begin{enumerate}
\item[$\bullet$]
the discriminant modular form
\[ \Delta(q) = q \prod_{n \geq 1} (1-q^n)^{24}\, , \]
\item[$\bullet$]
the Taylor expansion of the normalized odd Jacobi theta function
\[
\Theta(z,q) 
=  z \exp\left(-2\sum_{ k\geq 1} G_{2k} \frac{z^{2k}}{(2k)!}\right)\, ,
\]
where $G_{2k}(q) = - \frac{B_{2k}}{2 \cdot 2k} + \sum_{n \geq 1} \sum_{d|n} d^{2k-1} q^n$
is the classical Eisenstein series,
\item[$\bullet$] the negative of the logarithmic derivative of the Jacobi theta function
\[ S(z,q) = -q \frac{d}{dq} \log \Theta(z,q) =
\sum_{n \geq 1} \sum_{d|n} \frac{n}{d} (e^{dz} - 2 + e^{-dz}) q^n\, . \]
\end{enumerate}

Using the MCF, we then obtain a complete evaluation for all curve classes for both $K3$ and abelian surfaces:
\begin{eqnarray*}
\Big\langle \tau_0(\mathsf{p})^{m}\lambda_{g-m} \Big\rangle^{S}_{g,\beta_{2h-2,r}}
& =
&\sum_{k | r} k^{2g-3 + 2m}
\Big\langle \tau_0(\mathsf{p})^{m}\lambda_{g-m} \Big \rangle^{S}_{g,\beta_{(r/k)^2\cdot (2h-2),1}}\,.
\end{eqnarray*}
A discussion of the modular properties of the generating series in the imprimitive case can be found in
\cite[Section 7.5]{MPT}.

\section{Quot schemes on surfaces and wallcrossing} \label{quotstuff}
\subsection{Overview}
Let $S$ be a nonsingular projective surface.
In Section \ref{quotstuff},
we will study $\ell$-nested Quot schemes $\Quot_{\alpha}(\CE)$,
which parameterize flags of quotients
\begin{equation} \label{flagE}
 [E_0 \overset{j_0}{\twoheadrightarrow} E_1 \overset{j_1}{\twoheadrightarrow} \ldots \twoheadrightarrow E_{\ell-1} \overset{j_{\ell}}{\twoheadrightarrow} E_{\ell}] 
 \end{equation}
of coherent sheaves  $E_i \in \Coh(S)$ with fixed Chern characters $\alpha=(\alpha_0,\ldots,\alpha_\ell)$, where the initial sheaf $E_0$ is allowed to move in a moduli space $M$ of stable sheaves on $S$.
We will also consider a perverse version of the Quot scheme parameterizing flags of quotients in the abelian category $\Coh^{\sharp}(S)$ which is obtained from $\Coh(S)$ by tilting along the subcategory of 0-dimensional sheaves. 

If $S$ is  a symplectic surface, we will endow $\Quot_{\alpha}(\CE)$ and $\Quot^{\sharp}_{\alpha}(\CE)$ with perfect obstruction theories and show that there are as many linearly independent surjective cosections as there are jumps in the flag \eqref{flagE}. 
Furthermore, if
$M=S^{[n]}$ is the Hilbert scheme of $n$ points of $S$, 
we will show
in Theorem~\ref{thm:wallcrossing}
by wallcrossing that tautological integrals over the 1-nested Quot scheme are equal to the corresponding tautological integrals
over the perverse 1-nested Quot scheme.

In Section~\ref{dtptred},
the Quot schemes $\Quot_{\alpha}(\CE)$ and $\Quot^{\sharp}(\CE)$ for $M=S^{[n]}$  will appear naturally (with their obstruction theories) as $\BC^*$-fixed loci of the Hilbert schemes of curves and of the moduli spaces of stable pairs on the threefold $S \times \p^1$. 
The existence of sufficiently many cosections shows that only the $1$-jump flags can contribute to the reduced DT/PT invariants in the localization formula.
As a consequence, we will derive a multiple cover formula for the PT invariants in Section~\ref{dtptred}.
The equality of tautological integrals immediately implies a DT/PT correspondence.

\subsection{Nested Quot schemes}
Let $S$ be a nonsingular projective surface. Let $M$ be a proper, nonsingular, moduli space of stable sheaves{\footnote{Stability is defined with respect to a stability condition which is omitted in the notation. We assume there are no strictly semistable sheaves.}} on $S$ of positive rank and fixed determinant with Chern character $\alpha_0 \in H^{\ast}(S,\BQ)$. We assume the existence of a universal sheaf
$\CE$ on $M \times S$ and that
the tangent bundle of $M$ is given by{\footnote{The second equality implies that the moduli space $M$ is unobstructed.}} 
\begin{equation} \label{tangent bundle M} T_M = \ext^1_{\pi}(\CE,\CE)_0 = R \hom_{\pi}(\CE,\CE)_0[1]\, , \end{equation}
where
$\pi : M \times S \to M$ is the projection and $( \ - \ )_0$ stands for taking the traceless part, 
\begin{align*}
\ext^1_{\pi}(\CE,\CE)_0 & =\ker(\mathrm{tr}: \ext^1_{\pi}(\CE,\CE) \to \Ext^1(\CO_S, \CO_S) \otimes \CO_M)\, , \\
R \hom_{\pi}(\CE,\CE)_0 & = \Cone( R \hom_{\pi}(\CE,\CE) \xrightarrow{\mathrm{tr}} R \Hom(\CO_S, \CO_S) \otimes \CO_M)[-1]\, .
\end{align*}
The fundamental example for us will
be the Hilbert scheme of points $M=S^{[n]}$ viewed as a moduli space of ideal sheaves with the universal ideal sheaf $\CE=\CI_Z$.

Let $\alpha=(\alpha_0,\ldots,\alpha_{\ell})$ 
where $\alpha_0,\ldots, \alpha_{\ell} \in H^{\ast}(S,\BQ)$.
The nested Quot scheme $\Quot_{\alpha}(\CE)$
parameterizes flags of surjections in $\Coh(S)$ of sheaves in $M$ with numerical data specified by $\alpha_i$,
\[ 
\Quot_{\alpha}(\CE)
=
\Big\{ E_{\bullet} = [E_0 \overset{j_0}{\twoheadrightarrow} E_1 \overset{j_1}{\twoheadrightarrow} \ldots \twoheadrightarrow E_{\ell-1} \overset{j_{\ell}}{\twoheadrightarrow} E_{\ell}]\, \Big|\, E_i \in \Coh(S),\, E_{0} \in M,\, \ch(E_i) = \alpha_i \Big\}. \]

By definition, the objects of $\Quot_{\alpha}(\CE)$ over a finite type scheme $T$ are flags of surjections of coherent sheaves $$\CE_{\bullet} = [\CE_0 \twoheadrightarrow \CE_1 \twoheadrightarrow \ldots \twoheadrightarrow \CE_{\ell}]$$ on $T \times S$, with $\CE_i$ flat over $T$, $\CE_0 \in M(T)$,
and with $\ch((\CE_i)_t)=\alpha_i$ for all closed points $t \in T$.
Moreover, two such flags $\CE_{\bullet}$ and $\CE'_{\bullet}$ are isomorphic if there are isomorphisms $\varphi_i : \CE_i \to \CE'_i$ such that $$\varphi_i \circ j_{i-1} = j'_{i-1} \circ \varphi_{i-1}$$ for all $i$.\footnote{The isomorphism
$\varphi_0$ must respect the determinant trivialization. More precisely,
the moduli space $M$ with fixed determinant $\CL\in \mathsf{Pic}(S)$ parameterizes, over a scheme $T$, the data $(\CE_0,\gamma)$ where $\CE$ is a coherent sheaf on $T \times S$ flat over $T$, fiberwise stable, and an isomorphism $\gamma : \det(\CE) \to \pi_S^{\ast}(\CL)$. Two pairs $(\CE,\gamma)$ and $(\CE', \gamma')$ are isomorphic if there exists an isomorphism $\varphi : \CE \to \CE'$ such that $\gamma' \circ \varphi = \gamma$.}

\begin{rmk}
For a flag $E_{\bullet} \in \Quot_{\alpha}(\CE)$, 
we may set $F_i = \mathrm{Ker}(E_0 \to E_i)$ and obtain a flag of inclusion
\begin{equation} \label{inclusions}
0 = F_0 \hookrightarrow F_1 \hookrightarrow \ldots \hookrightarrow F_{\ell} \subset E_0\, .
\end{equation}
As such, $\Quot_{\alpha}(\CE)$ may also be viewed as a scheme parameterizing filtrations \eqref{inclusions} with $\ch(F_i) = \alpha_0 - \alpha_i$.
In families, $\Quot_{\alpha}(\CE)$ parameterizes flags $\CF_0 \hookrightarrow \ldots \hookrightarrow \CF_{\ell} \hookrightarrow \CE_0$ of inclusions with $\CE_0/\CF_i$ flat.
The point of view of inclusions will be our connection to the fixed locus of DT/PT moduli spaces.
However, from a functorial point of view, surjections behave better, so we choose to work with surjections here.
\end{rmk}

The nested Quot scheme can be constructed by induction on $\ell$
by realizing $\Quot_{(\alpha_0,\ldots,\alpha_{\ell})}(\CE)$
as the Grothendieck Quot scheme of the universal sheaf $\CE_{\ell-1}$
over $\Quot_{(\alpha_0,\ldots,\alpha_{\ell-1})}(\CE)$.
Therefore, $\Quot_{\alpha}(\CE)$ is a proper scheme.
An obstruction theory 
for $\Quot_{\alpha}(\CE)$
can be described as follows.

\begin{thm}[\cite{MPR}] \label{thm:pot on nested quot scheme} 
The Quot scheme
$\Quot_{\alpha}(\CE)$ has a natural obstruction theory $\BE^{\bullet} \to \BL_{\Quot_{\alpha}(\CE)}$ with virtual tangent bundle
\begin{equation} \label{Ebullet} (\BE^{\bullet})^{\vee} \cong \Cone\left(
R\hom_{\pi}(\CE_{0},\CE_{0})_{0} \oplus \bigoplus_{i=1}^{\ell} R\hom_{\pi}(\CE_i, \CE_i) \xrightarrow{\psi} \bigoplus_{i=1}^{\ell} R\hom_{\pi}(\CE_{i-1}, \CE_{i}) \right)\, , \end{equation}
where $\CE_0 \twoheadrightarrow \cdots \twoheadrightarrow \CE_{\ell}$ is the universal flag on $S\times \Quot_{\alpha}(\CE)$,
the map $$\pi : S\times \Quot_{\alpha}(\CE)  \to \Quot_{\alpha}(\CE)$$ is the projection, and $\psi$ is the composition of the natural morphism
\[
R\hom_{\pi}(\CE_{0},\CE_{0})_{0} \oplus \bigoplus_{i=1}^{\ell} R\hom_{\pi}(\CE_i, \CE_i)
\to 
\bigoplus_{i=0}^{\ell} R\hom_{\pi}(\CE_i, \CE_i)
\]
with the morphism $\delta(\varphi_0,\ldots,\varphi_{\ell}) = (j_0 \circ \varphi_0 - \varphi_1 \circ j_0, \ldots, j_{\ell-1} \circ \varphi_{\ell-1} - \varphi_\ell \circ j_{\ell-1})$.
\end{thm}

Theorem \ref{thm:pot on nested quot scheme} is proven in \cite{MPR}. For the convenience of the
reader (and to set notation), we briefly sketch the main ideas of the proof.

\begin{proof} (Sketch)
Let $\Perf(S)$ be the derived stack of perfect complexes on $S$.
Let $\Perf^{[\ell]}(S)$ be the stack parameterizing length $\ell+1$ flags $$\CE_{\bullet}=[ 
\CE_0 \to \CE_1 \to \ldots \to \CE_{\ell}]$$ of perfect complexes (with the maps not necessarily injective or surjective).
The stack $\Perf(S)$ has a derived tangent complex $R \hom_{\pi}(\CE,\CE)[1]$, where $\CE$ is the universal perfect complex. The morphism
\[ f: \Perf^{[\ell]}(S) \to \Perf(S)^{\ell+1}, \quad \CE_{\bullet}
\mapsto (\CE_0,\ldots, \CE_{\ell}) \]
has  a relative tangent complex $\BT_f=\bigoplus_{i=1}^{\ell} R \hom_{\pi}( \CE_{i-1},\CE_{i})$. 
Using the exact triangle $$\BT_{f} \to \BT_{\Perf^{[\ell]}(S)} \to f^{\ast} \BT_{\Perf(S)^{\ell+1}}\,,$$ we find that $\Perf^{[\ell]}(S)$ has the tangent complex
\[
\BT_{\Perf^{[\ell]}(S)} \cong \Cone\left( \bigoplus_{i=0}^{\ell} R \hom_{\pi}(\CE_i,\CE_i) \xrightarrow{\delta} \bigoplus_{i=1}^{\ell} R \hom_{\pi}(\CE_{i-1}, \CE_{i}) \right).
\]
The map $\delta$ is worked out in \cite{MPR} to precisely be as stated in Theorem \ref{thm:pot on nested quot scheme}.

Next, we base change along the natural map
$M \to \Perf(S)$.
Since stability is open, there is an open substack $\BR M \subset \Perf(S)$ parameterizing stable sheaves of class $\alpha_0$.
The moduli space $M$ is a fiber of the {{derived determinant}} map 
\cite{Schuerg}
$$\BR \det : \BR M \to \BR \Pic(S)\,.$$
Hence, the relative tangent complex of the morphism $M \to \BR M \subset \Perf(S)$ is $R \hom_{\pi}(\CO_S,\CO_S)$.
Consider the fiber diagram
\[
\begin{tikzcd}
\Perf^{[\ell]}(S)_{/M} \ar{r} \ar{d} & \Perf^{[\ell]}(S) \ar{d} \\ 
M \times \Perf(S)^{\ell} \ar{r} & \Perf(S)^{\ell+1}.
\end{tikzcd}
\]
We find that $\Perf^{[\ell]}(S)_{/M}$ has tangent complex
\begin{equation} \label{abc44}
\BT_{\Perf^{[\ell]}(S)_{/M}}
\cong \Cone\left( R \hom_{\pi}(\CE_0,\CE_0)_0 \oplus \bigoplus_{i=1}^{\ell} R \hom_{\pi}(\CE_i,\CE_i) \xrightarrow{\psi} \bigoplus_{i=1}^{\ell} R \hom_{\pi}(\CE_{i-1}, \CE_{i}) \right).
\end{equation}

Finally, surjection is an open condition by \cite[Lemma 2.4 (1-ii)]{MPR},
so $\BR \Quot_{\alpha}(\CE)$ embeds as an open substack into $\Perf^{[\ell]}(S)_{/M}$.
Therefore, isomorphism \eqref{abc44} also describes the tangent complex of $\BR \Quot_{\alpha}(\CE)$.
After
restricting to the classical truncation $\Quot_{\alpha}(\CE)$ and using the functoriality of the tangent complexes, we obtain the desired obstruction theory.
\end{proof}

\begin{thm} If $S$ is symplectic, then $\BE^{\bullet} \to \BL_{\Quot_{\alpha}(\CE)}$ defines a perfect obstruction theory.
\end{thm}

\begin{proof}
We must show that $\BE^{\bullet}$ as given in \eqref{Ebullet} is $2$-term in case $S$ is symplectic. 
The condition can be checked pointwise. 
Let $E_{\bullet} \in \Quot_{\alpha}(\CE)$. Since $j_i$ are surjective, we see that
\[ \oplus_{i=1}^{\ell} \Hom(E_i,E_i) \to \oplus_{i=1}^{\ell} \Hom(E_{i-1},E_{i}) \]
is injective (using the definition of $\delta$). Hence, $(\BE^{\bullet})^{\vee}$ is supported in degrees $\geq 0$.

To show that
$(\BE^{\bullet})^{\vee}$ is
supported in
degrees $\leq 1$, we must show that
\[
\bigoplus_{i=1}^{\ell} \Ext^2(E_i,E_i) \to \bigoplus_{i=1}^{\ell} \Ext^2(E_{i-1},E_{i})
\]
is surjective. 
Since $\omega_S \cong \CO_S$, the above map is Serre dual to the map
\[
\tilde{\delta} : \bigoplus_{i=1}^{\ell} \Hom(E_{i},E_{i-1}) \to \bigoplus_{i=1}^{\ell} \Hom(E_i,E_i)
\]
given by
\[
\tilde{\delta} : (\varphi_i : E_i \to E_{i-1})_{i=1}^{\ell} \mapsto 
(j_0 \circ \varphi_1 - \varphi_2 \circ j_1, \, \ldots \, , j_{\ell-2} \circ \varphi_{\ell-1} - \varphi_{\ell} \circ j_{\ell-1}, j_{\ell-1} \circ \varphi_{\ell})\, .
\]
The precise form of $\tilde{\delta}$ is determined by functoriality of Serre duality.
We must show $\tilde{\delta}$ is injective.

Assume $\tilde{\delta}(\varphi_1,\ldots,\varphi_{\ell})=0$. Since $E_{0}$ is simple, we have $$\varphi_1 \circ j_0 = \lambda \id_{E_0}$$ for some $\lambda \in \BC$.
If $\lambda=0$, the surjectivity of $j_0$ implies that $\varphi_1=0$.
Then, by looking at the first component of $\tilde{\delta}(\varphi_1,\ldots,\varphi_{\ell})$, we obtain $\varphi_2 \circ j_1=0$, and so $\varphi_2=0$. By repeating the argument, we find $\varphi_i=0$ for all $i$.
If $\lambda \neq 0$, then $j_0$ is injective, so an isomorphism, 
and the claim follows by induction using Remark~\ref{rmk:simplifying iso} below.
\end{proof}

\begin{rmk} \label{rmk:simplifying iso}
If $\alpha_{i}=\alpha_{i+1}$ for some $i$, then we have an isomorphism
\[ \Quot_{(\alpha_0,\ldots,\alpha_{\ell})}(\CE) \xrightarrow{\cong} \Quot_{(\alpha_0,\ldots,\alpha_i, \alpha_{i+2}, \ldots,\alpha_{\ell})} \]
given by sending a flag $[\CE_0 \twoheadrightarrow \ldots \twoheadrightarrow \CE_{\ell}]$ to 
\[
[\CE_0 \twoheadrightarrow \ldots \twoheadrightarrow \CE_{i} \xrightarrow{j_{i+1}\circ j_{i}} \CE_{i+2} \twoheadrightarrow \ldots \twoheadrightarrow \CE_{\ell}]\, .
\]
Since $j_{i}$ is an isomorphism,  an inverse is readily constructed.
The isomorphism is seen to respect the perfect obstruction theory \eqref{Ebullet}.
\end{rmk}

A finer analysis of the obstruction space reveals many non-trivial cosections. These will play a crucial role
in our study of the PT theory of $(S\times\mathbb{P}^1, S_0 \cup S_\infty)$.

\begin{thm} \label{thm:cosections on nested quot}
Let $S$ be a symplectic surface.
Let $s = |\{ i \in \{ 0, \ldots, \ell-1 \} \,|\, \alpha_{i+1} \neq \alpha_{i} \} |$. The
perfect obstruction theory $E^{\bullet} \to \BL_{\Quot_{\alpha}(\CE)}$ admits an everywhere surjective cosection of rank $s$:
\[ h^1((E^{\bullet})^{\vee}) \twoheadrightarrow \CO_{\Quot_{\alpha}(\CE)}^{\oplus s}\, . \]
\end{thm}

Here, $s$ is the number of indices for which
the quotient map $j_i : E_i \to E_{i+1}$ is not an isomorphism.
We call such an instance a {\em jump} in the flag. Theorem \ref{thm:cosections on nested quot} says that we obtain
a linearly independent cosection for every jump.

\begin{proof}
By Remark~\ref{rmk:simplifying iso}, we may assume that
$\alpha_{i+1} \neq \alpha_i$ for all $i$, so that $s= \ell$.
We claim that the composition
\[
\mathrm{Obs} = h^1((E^{\bullet})^\vee) \to \bigoplus_{i=1}^{\ell} \Ext^2_{\pi}(\CE_i,\CE_i) \xrightarrow{T= (\tr, \tr,\ldots,\tr)} \CO^{\oplus \ell}
\]
is everywhere surjective, where the first map is the connecting homomorphism of the cone \eqref{Ebullet}.

We can check surjectivity at points $E_{\bullet} \in \Quot_{\alpha}(\CE)$. The dual map $T^{\vee}$ fits into the diagram
\[
\begin{tikzcd}
& & \BC^{\ell} \ar{d}{T^{\vee}|_{E_{\bullet}}} \\
0 \ar{r} & \bigoplus_{i=1}^{\ell} \Hom(E_{i},E_{i-1}) \ar{r}{\tilde{\delta}} & \bigoplus_{i=1}^{\ell} \Hom(E_i,E_i)
\ar{r} & (\mathrm{Obs}|_{E_{\bullet}})^{\vee}
\end{tikzcd}
\]
and is given by $T^{\vee}(a_1,\ldots,a_{\ell}) = (a_1 \id_{E_1},\ldots,a_{\ell} \id_{E_{\ell}})$.
We must show that
\[ T^{\vee}(\BC^{\ell}) \cap \mathrm{Im}(\tilde{\delta}) = 0\, . \]
Let $(\varphi_1,\ldots,\varphi_\ell) \in \oplus_{i=1}^{\ell} \Hom(E_{i},E_{i-1})$ 
satisfy
\begin{equation} \label{3sdf} \tilde{\delta}(\varphi_1,\ldots,\varphi_\ell) = (a_1 \id_{E_1},\ldots,a_{\ell} \id_{E_{\ell}})\, . \end{equation}
Since $E_0$ is simple and $j_0$ is not an isomorphism, we have $\varphi_1 \circ j_0=0$, so $\varphi_1=0$.
The first component of \eqref{3sdf} yields $-\varphi_2 \circ j_1 = a_1 \id_{E_1}$, but since $j_1$ is also not an isomorphism, we find $\varphi_2=0$.
Continuing, we get $\varphi_i=0$ for all $i$, so $\tilde{\delta}(\varphi_1,\ldots,\varphi_\ell)=0$.
\end{proof}

\subsection{Perverse $t$-structure} \label{subsec:perverse t structure}
Consider the torsion pair $(\CT, \CT^\perp)$ in $\Coh(S)$:
\begin{align*}
	\CT & = \{ A \in \Coh(S) | \dim(A) = 0 \}\, , \\
	\CT^{\perp} & = \{ A' \in \Coh(S) | \Hom(A,A') = 0 \text{ for all } A \in \CT \}\, .
\end{align*}
Tilting $\Coh(S)$ along $(\CT, \CT^\perp)$ yields the perverse abelian heart
\[ \Coh^{\sharp}(S) = \langle \CT^{\perp}, \CT[-1] \rangle\, . \]
An object $E \in D^b(S)$ lies in $\Coh^{\sharp}(S)$ if and only if
\[ h^0(E) \in \CT^{\perp}, \quad h^1(E) \in \CT, \quad \forall i \notin \{ 0,1 \}: \ h^i(E) = 0\, . \]
For objects $A,B \in \Coh^{\sharp}(S)$, a morphism $\varphi : A \to B$ is {\em injective} if $\Cone(\varphi) \in \Coh^{\sharp}(S)$
and is {\em surjective}  if $\Cone(\varphi)[-1] \in \Coh^{\sharp}(S)$. Hence,
\begin{itemize}
	\item $\varphi$ is injective iff $h^{-1}(\Cone(\varphi)) = 0$ and $h^0(\Cone(\varphi)) \in \CT^{\perp}$,
	\item $\varphi$ is surjective iff $h^{0}(\Cone(\varphi)) \in \CT$ and $h^1(\Cone(\varphi)) = 0$ .
\end{itemize}

\begin{rmk}
\label{rmk:cone explicit in coh sharp}
Let $A=[A_0 \xrightarrow{d_A} A_1]$ and $B=[B_0 \xrightarrow{d_B} B_1]$ be in $\Coh^{\sharp}(S)$. Suppose
$\varphi_i : A_i \to B_i$ define a morphism of complexes $\varphi : A \to B$.
The cone of $\varphi$ is given by
\[ A_0 \xrightarrow{(\varphi_0,-d_A)} B_0 \oplus A_1 \xrightarrow{(d_B, \varphi_1)} B_1 \]
with $A_0$ in degree $-1$. In particular, if $\varphi_0 = \id$, then
$\Cone(\varphi) = \mathrm{Cone}(A_1 \to B_1)$ and $\Cone(\varphi) \in \Coh^{\sharp}(S)$ if and only if
the zero-dimensional torsion of $\Ker(\varphi_1)$ vanishes.
\end{rmk}

\begin{lemma} \label{coh sharp is noetherian}
The category $\Coh^{\sharp}(S)$ is Noetherian.
\end{lemma}
\begin{proof}
We must show, for every object $E$ in $\Coh^{\sharp}(S)$,  every 
filtration $$E_1 \subset E_2 \subset \ldots \subset E$$ stabilizes.
The property is preserved under extension, so we can assume $E \in \CT[-1]$ or $E \in \CT^{\perp}$.
If $E \in \CT[-1]$, then a subobject of $E$ is just an object of $\CT[-1]$,
so the claim follows since $\Coh(S)$ is Noetherian.
If $E \in \CT^{\perp}$, then the injection $E_i \subset E$ yields a long exact sequence in $\Coh(S)$:
\[ 0 \to h^0(E_i) \to E \to h^0(C) \to h^1(E_i) \to 0 \]
so $h^1(C)=0$ and $C=h^0(C) \in \CT^{\perp}$. For $i \gg 0$, $h^0(E_i)$ stabilizes, so $C$ is the extension
\[ 0 \to F \to C \to T_i \to 0\]
for a fixed sheaf $F \in \CT^{\perp}$ and $T_i = h^1(E_i) \in \CT$.
The double dual $F^{\ast \ast}$ is locally free with $Q=F^{\ast \ast}/F$ zero-dimensional, and $F^{\ast} = C^{\ast}$, so also $F^{\ast \ast} = C^{\ast \ast}$. We obtain a diagram
\[
\begin{tikzcd}
0 \ar{r} & F \ar{r} & C \ar{r} \ar{d} & T_i \ar{r} & 0 \\
0 \ar{r} & F \ar{r} & F^{\ast \ast} \ar{r} & Q \ar{r} & 0
\end{tikzcd}
\]
which yields an embedding $T_i \hookrightarrow Q$. Since $Q\in \Coh(S)$, 
$T_i$ stabilizes for $i \gg 0$ and, therefore, so does $C$ and $E_i$.
\end{proof}

A perfect object $F \in D_{\mathrm{perf}}(S \times B)$ over a base scheme $B$ is {\em $B$-flat with respect to $\Coh^{\sharp}(S)$} if,
for all closed points $t \in B$, we have
\[ F_t = L \iota_{t}^{\ast} F \in \Coh^{\sharp}(S)_t\, , \]
where $\iota_{t} : S \times \{ t \} \to S \times B$ is the inclusion.

Let $M$ be a proper algebraic space over $\mathbb{C}$, and let $\CE$ be a $M$-flat object on $M \times S$
with respect to 
$\Coh^{\sharp}(S)$,
and let $v \in H^{\ast}(S,\BQ)$.
Consider the functor
\[ \Quot^{\sharp}(\CE,v) : (\mathrm{Schemes\,  of\, finite\, type}/M)^{\mathrm{op}} \to \mathrm{Sets} \]
sending a $M$-scheme $f : T \to M$ to the set of morphisms $\CE_T \to Q$ such that
\begin{itemize}
\item $Q \in D_{\mathrm{perf}}(T \times S)$ is $T$-flat with respect to $\Coh^{\sharp}(S)$,
\item for all closed points $t \in T$, the restriction $\CE_{T,t} \to Q_{t}$ is surjective in $\Coh^{\sharp}(S)$
and $\ch[Q_t] = v$.
\end{itemize}
For any morphism $T' \to T$, the map $$\Quot^{\sharp}(\CE,v)(T) \to \Quot^{\sharp}(\CE,v)(T')$$ sends
$\CE_T \to Q$ to the pullback $\CE_{T'} \to Q_{T'}$.

\begin{thm}
The functor $\Quot^{\sharp}(\CE,v)$ is represented by a proper algebraic space of finite type over $M$.
\end{thm}
\begin{proof}
The $t$-structure on $D^b(\Coh(S))$ defined by the heart $\Coh^{\sharp}(S)$
defines a constant family of $t$-structures on $D(S \times B)$ for any scheme $B$ by base change \cite{AP,StabFam}.
Since $\Coh^{\sharp}(S)$ is Noetherian, \cite[Proposition 3.3.2]{AP} shows that the constant family of $t$-structures satisfies {\em openness in flatness} over finite type schemes: {\em for every object $E \in D^b(S \times B)$, with $B$ of finite type over $\mathbb{C}$, which satisfies $E_b \in \Coh^{\sharp}(S)_b$ for some point $b \in B$, there exists a Zariski open neighborhood 
$$b\in U\subset B$$ 
such that $E|_{b'} \in \Coh^{\sharp}(S)_{b'}$ for all closed $b' \in U$.}
Then, by \cite[Lemma 10.6]{StabFam}, the constant family of $t$-structures 
{\em universally satisfies openness of flatness} in the sense of \cite[Definition 10.4]{StabFam}, see also \cite[Remark 2.14]{Rota}.
Therefore, by \cite[Proposition 11.6]{StabFam}, $\Quot^{\sharp}(\CE,v)$ is represented by an algebraic space locally of finite type over $M$ (since a locally of finite presentation over finite type is locally of finite type).

Moreover, by \cite[Proposition 11.11]{StabFam},  $\Quot^{\sharp}(\CE,v)$ satisfies the existence and uniqueness part of the valuative criterion for properness (for all valuation rings).
The hypotheses of  
\cite[Proposition 11.11]{StabFam}
are satisfied because,
for a valuation ring $R$, the condition of having a $\Spec(R)$-torsion structure on the base-changed $t$-structure is satisfied 
since $\Coh^{\sharp}(S)$ is Noetherian,
see \cite[Section 6.6 and  Remark 6.16]{StabFam} and the comments after \cite[Theorem 5.7]{StabFam}.

The last step is to show that $\Quot^{\sharp}(\CE,v)$ is bounded:   all objects $Q$ arising as quotients $\CE \to Q$ can be parameterized by a finite type scheme.
To see this, consider a quotient in $\Coh^{\sharp}(S)$
\[ 0 \to K \to E \to Q \to 0 \]
where $E = \CE_{m}$ for some $m \in M$. Take the associated long exact sequence in cohomology{\footnote{Here, $K,E,Q\in D^b(S)$ and $h^i$ denotes
the cohomology of the complex.}}
\begin{align*}
0 \to & h^0(K) \to h^0(E) \to h^0(Q) \to \\
& h^1(K) \to h^1(E) \to h^1(Q) \to 0\, .
\end{align*}

\vspace{4pt}
\noindent $\bullet$ Since $E$ moves in a bounded family, the $0$-dimensional sheaves $h^1(E)$ are also bounded (and have a finite fixed list of possible values of $\ch_2$).
Since the Quot scheme in $\Coh(S)$ is bounded, the families of both  $h^1(Q)$ 
and the kernels 
$$H = \mathrm{Ker}(h^1(E) \to h^1(Q))$$
are also bounded.

\vspace{8pt}
\noindent $\bullet$
The quotients $0\to h^0(K) \to h^0(E) \to E'\to 0$ fit into the short exact sequence
\[ 0 \to E' \to h^0(Q) \to T \to 0 \, , \]
where $T = \mathrm{Ker}(h^1(K) \to H)$ is 0-dimensional.
Since $$\ch(h^0(Q)) = (v_0,v_1,v_2 - \ch_2(h^1(Q)))\,, $$
$\ch_i(E')=v_i$ for $i=0,1$, and 
\[ \ch_2(E') = \ch_2(h^0(Q)) - \ch_2(T) = v_2 - \ch_2(H^1(Q)) - \ch_2(T)\, . \]
Since the possible values of $\ch_2(h^1(Q))$ are bounded and $\ch_2(T) \geq 0$,  we conclude that $\ch_2(E')$ is bounded from above.
However, $E'$ is also the quotient of the bounded $h^0(E)$, so $\ch_2(E')$ is also bounded from below universally in terms of $\ch_0(E')$ and $\ch_1(E')$.
Hence, $\ch(E')$ is constrained to lie in a finite set.
By the boundedness of the Quot scheme for $\Coh(S)$, the possible choices of $E'$ are bounded.
By the double-dual argument (used in the proof of Lemma~\ref{coh sharp is noetherian}), the possible extensions
\[ 0 \to E' \to h^0(Q) \to T \to 0 \]
are bounded.


\vspace{8pt}
The families of both $h^0(Q)$ and $h^1(Q)$ are bounded. Since the Ext sheaves are finite dimensional, we conclude the objects $Q$ are bounded, which 
completes the proof of boundedness for $\Quot^{\sharp}(\CE,v)$. Boundedness for
$\Quot^{\sharp}(\CE,v)$ implies
properness.
\end{proof}

Let $M$ be a proper, nonsingular, moduli space of stable{\footnote{We again omit the stability condition in the notation.}} objects 
in $\Coh^{\sharp}(S)$  of fixed determinant and positive rank
with Chern character $\alpha_0 \in H^{\ast}(S)$.
We assume the existence of a 
universal object $\CE$ on
$M \times S$
and that the tangent bundle of $M$ is given by \eqref{tangent bundle M}.
Since the Quot scheme exists as a proper algebraic space for $\Coh^{\sharp}(S)$,
we find, by induction, that  the nested Quot scheme
\[
\Quot^{\sharp}_{\alpha}(\CE) = \{
E_{\bullet}=[E_0 \overset{j_0}{\twoheadrightarrow} E_1 \overset{j_1}{\twoheadrightarrow} \ldots \twoheadrightarrow E_{r-1} \overset{j_r}{\twoheadrightarrow} E_r | E_i \in \Coh^{\sharp}(S), E_0 \in M, \ch(E_i) = \alpha_i \}
\]
exists as a proper algebraic space.
Moreover,
the statements of Theorem~\ref{thm:pot on nested quot scheme} and
Theorem~\ref{thm:cosections on nested quot} also hold.

\subsection{Nested and perverse-nested Hilbert schemes} \label{subsec:nested hilbert schemes}
We now specialize Quot schemes to the particular setting relevant to our study of the multiple cover formula. 

Let $S$ be a symplectic surface.
Let $M=S^{[n]}$ be the Hilbert scheme of $n$ points on $S$, let $\alpha_0=(1,0,-n)$ be the Chern character
of the associated ideal sheaves, and let $\CE=I_{\CZ}$ be the universal ideal sheaf. Let 
\[ \alpha_1=(0,\beta,m) \in H^2(S,\BZ) \oplus H^4(S,\BQ) \subset H^{\ast}(S,\BQ)\,,  \quad \beta \neq 0 \]
be another Chern character.
Consider the $1$-nested Quot schemes with $\alpha=(\alpha_0,\alpha_1)$:
\begin{align*}
\Quot^{\flat}_\alpha
&= \Quot_{\alpha}(\CI_{\CZ})\, , \\
\Quot^{\sharp}_\alpha 
&= \Quot^{\sharp}_{\alpha}(\CI_{\CZ})\, .
\end{align*}
A concrete description is:
\begin{align*}
\Quot^{\flat}_\alpha & = \left\{ \varphi : I_z \to G \middle| \begin{array}{c}
	z \in S^{[n]},\,  \ch(G)=\alpha_1 \\ 	G \text{ is a 1-dimensional coherent sheaf} \\ \mathrm{coker}(\varphi) = 0
\end{array} \right\} 
\\
\Quot^{\sharp}_\alpha & = 
 \left\{ \varphi : I_z \to G \middle| \begin{array}{c} z \in S^{[n]},\,  \ch(G)=\alpha_1 \\ 
G \text{ is a pure 1-dimensional coherent sheaf} \\ \mathrm{coker}(\varphi) \text{ 0-dimensional sheaf}
\end{array} \right\}\,.
\end{align*}
Theorems~\ref{thm:pot on nested quot scheme}
and \ref{thm:cosections on nested quot} specialize to yield the following result.
\begin{cor}
    \label{lemma:pot reduction on Quot}
For $\star \in \{ \flat, \sharp \}$, there is a perfect obstruction theory $E^{\bullet} \to L_{\Quot^{\star}_{\alpha}}$ with virtual tangent bundle
\[
(E^{\bullet})^{\vee} \cong \mathrm{Cone}
\left( 
T_{S^{[n]}}[-1] \oplus
R\hom_{\pi}(\BI^{\bullet},\BI^{\bullet}) \to R \hom_{\pi}( \BI^{\bullet}, \CI_z ) \right)
\]
where 
 $\BI^{\bullet} = [\CI_Z\xrightarrow{\Phi} \BG]$ is the universal quotient.
The perfect obstruction theory admits an everywhere surjective cosection $h^1(E^{\bullet}) \to \CO$. The reduced virtual dimension is $m+2n+\beta^2+1$
\end{cor}
By Kiem-Li cosection localization \cite{KiemLiCosectionLoc}, there exists an associated reduced virtual class
\[ [ \Quot^{\star}_{\alpha} ]^{\mathrm{red}} \in \mathsf{CH}_{\beta^2+m+2n+1}(\Quot^{\star}_{\alpha})\, . \]
Let $\rho$ and $\pi$ denote the projections
$$\rho: S\times \Quot^{\star}_{\alpha} \to S\, , \ \ \  \
\pi: S\times \Quot^{\star}_{\alpha} \to \Quot^\star_\alpha\, .
$$
Two types of {\em descendent classes} 
will play a role in our study:
\begin{enumerate}
\item[$\bullet$]
$\ch^{\BG}_a(\gamma)
= \pi_{\ast}( \ch_a(\BG) \cdot\rho^{\ast}(\gamma))
$ on
$\Quot^{\star}_{\alpha}$
with respect to the sheaf
$$ \mathbb{G} \rightarrow  S \times \Quot^{\star}_{\alpha}\, , $$
\item[$\bullet$] 
$\ch^{\mathcal{O}_{\mathcal{Z}}}_a(\gamma)=
\pi_{\ast}( \ch_a(\CO_{\CZ}) \cdot \rho^{\ast}(\gamma))$ 
on $S^{[n]}$
with respect to the sheaf
$$ \mathcal{O}_{\mathcal{Z}} \rightarrow S\times S^{[n]}\, .$$
\end{enumerate}

The descendents $\ch^{\mathcal{O}_{\mathcal{Z}}}_a(\gamma)$ can be canonically
pulled-back to 
$\Quot^{\star}_{\alpha}$
via the structure map
$$
\Quot^{\star}_{\alpha}
\rightarrow S^{[n]}\, .$$
The equality of descendent integrals over $\Quot^{\flat}_{\alpha}$ and $\Quot^{\sharp}_{\alpha}$ will play a crucial role in the DT/PT correspondence.

\begin{thm}\label{thm:wallcrossing}
Let $a_{1}, \ldots, a_{\ell} \geq 0$ be integers, and let $b : \{ 1, \ldots, \ell \} \to \{ \BG, \CO_{\CZ} \}$ be a superscript assignment.
Let $\gamma_1,\ldots,\gamma_{\ell} \in H^{\ast}(S,\BQ)$ be cohomology classes.
Let $\alpha=((1,0,-n),(0,\beta,m))$ for an effective curve class $\beta\in H^2(S,\BZ)$.
Then, we have
\[ 
\int_{ \left[ \Quot^{\flat}_{\alpha} \right]^{\red} }
\prod_{i=1}^{\ell} \ch_{a_i}^{b(i)}(\gamma_i)
=
\,
\int_{ \left[ \Quot^{\sharp}_{\alpha} \right]^{\red} }
\prod_{i=1}^{\ell} \ch_{a_i}^{b(i)}(\gamma_i)
\]
\end{thm}

Every monomial in the descendent classes is of the form $\prod_{i=1}^{\ell} \ch_{a_i}^{b(i)}(\gamma_i)$.
Hence, Theorem \ref{thm:wallcrossing} covers arbitrary descendent integrals over the nested Quot schemes.
The proof of Theorem \ref{thm:wallcrossing} follows  a well-known wallcrossing approach \cite{Joyce2,Mochizuki}. Since we could not find an appropriate reference, a proof is provided in Appendix A.

\subsection{Universality}
For the proof of the PT multiple cover formula in Section \ref{dtptred}, a universality result for
descendent integrals on $\Quot_{\alpha}^\star$ will be required.
The universality here is based on the results of Gholampour-Thomas \cite{GT2} which calculate the pushforward of the virtual class of the nested Quot scheme to a product of Hilbert schemes of points of $S$ (to which the
universality of Ellingsrud-G\"ottsche-Lehn \cite{EGL} can be applied).

\begin{thm} \label{thm:universality}
Let $\star \in \{ \flat, \sharp \}$.
Let $a_{1}, \ldots, a_{\ell} \geq 0$ be integers, and let $$b : \{ 1, \ldots, \ell \} \to \{ \BG, \CO_{\CZ} \}$$ be a superscript assignment.
Let $\gamma_1,\ldots,\gamma_{\ell} \in H^{\ast}(S,\BQ)$ be cohomology classes.
Let $\alpha=((1,0,-n),(0,\beta,m))$ for an effective curve class $\beta\in H^2(S,\BZ)$.
Then, for every symplectic surface $S'$ and degree-preserving $\mathbb{Q}$-algebra isomorphism $$\varphi : H^{\ast}(S, \BQ) \to H^{\ast}(S',\BQ)$$ for which $\varphi(\pt)=\pt$ and $\varphi(\beta)$ is effective, we have
\[ 
\int_{ [ \Quot^{\star}_{\alpha} ]^{\red} }
\prod_{i=1}^{\ell} \ch_{a_i}^{b(i)}(\gamma_i)
=
\int_{ [ \Quot^{\star}_{\varphi(\alpha)} ]^{\red} }
\prod_{i=1}^{\ell} \ch_{a_i}^{b(i)}(\varphi(\gamma_i))\, .
\]
\end{thm}

\begin{proof}[Proof of Theorem~\ref{thm:universality}]
The claim was proven in \cite[Theorem 3.2]{QuasiK3} for $K3$ surfaces. The argument for abelian surfaces requires some modifications, so we provide a full argument.
The proof is heavily inspired by the ideas of \cite{GT1,GT2,KT1,KT2} where similar results were proven.

By Theorem~\ref{thm:wallcrossing}, we can work with integrals over $\Quot^{\flat}_{\alpha}$.
For an element $\varphi : I_z \to G$ in $\Quot_{\alpha}^{\flat}$, the kernel of $\varphi$ is of the form $I_{z'}(-D)$ where $D \subset S$ is a Cartier divisor with homology class $\beta$ and $I_{z'}$ is the ideal sheaf of a 0-dimensional subscheme of length $$n'=n+m + \beta^2/2\,,$$ as can be seen by taking double-duals of the inclusion $\mathrm{Ker}(\varphi) \hookrightarrow I_z$.
Thus $\Quot_{\alpha}^{\flat}$ is isomorphic to the nested Hilbert scheme\footnote{The moduli space $H_{\beta}^{[n',n]}$ is denoted by $S_{\beta}^{[n',n]}$ in \cite{GT2}.} $H_{\beta}^{[n',n]}$
as defined in \cite{GT2} parameterizing inclusions of coherent sheaves
\begin{equation} I_{z'}(-D) \subset I_z. \label{e0-dsf} \end{equation}
By a direct check, the $K$-theory classes of the perfect obstruction theories of
$\Quot_{\alpha}^{\flat}$ (as constructed above) and 
of $H_{\beta}^{[n',n]}$ (as constructed in \cite[Theorem 4.16]{GT2}) are the same.
In \cite{GT2}, a reduced virtual class on $H_{\beta}^{[n',n]}$ is constructed by removing a trivial summand from the obstruction space. It follows that the reduced virtual classes agree.

Let $H_{\beta}$ be the Hilbert scheme of curves on $S$ in class $\beta$.
The Hilbert scheme admits a reduced perfect obstruction theory with reduced virtual class $[H_{\beta}]^{\red}$. Let $\CD_{\beta}$ be the universal Cartier divisor on $S \times H_{\beta}$. We define the descendent classes
\[ \ch_k^{\CO(\CD_{\beta})}(\gamma) = \pi_{\ast}( \ch_k(\CO(\CD_{\beta})) \cdot \rho^{\ast}(\gamma) )\, . \]

By \cite[Corollary 4.22]{GT2}, the natural embedding
\[ j : \Quot_{\alpha}^{\flat} \to S^{[n']} \times S^{[n]} \times H_{\beta} \]
given by sending the element \eqref{e0-dsf} to the triple $(z',z,D)$
satisfies
\[
j_{\ast} [ \Quot_{\alpha}^{\flat} ]^{\red}
=
c_{n+n'}\Big( R \pi_{\ast} \CO(\CD_{\beta}) - R \hom_{\pi}(\CI_{\CZ'}, \CI_{\CZ}(\CD_{\beta}) ) \Big) \cap [S^{[n']} \times S^{[n]}] \times [H_{\beta}]^{\red}
\]
where
\begin{itemize}
\item $\pi : S \times S^{[n']} \times S^{[n]} \times H_{\beta} \to S^{[n']} \times S^{[n]} \times H_{\beta}$ is the projection,
\item the universal families $\CI_{\CZ'},\, \CI_{\CZ},\, \CO(\CD_{\beta})$ of the moduli spaces $S^{[n']},\, S^{[n]},\, H_{\beta}$ are viewed here on $S \times S^{[n']} \times S^{[n]} \times H_{\beta}$ by pullback.
\end{itemize}

By the Grothendieck-Riemann-Roch theorem, the factor
\[ c_{n+n'}\Big( R \pi_{\ast} \CO(\CD_{\beta}) - R \hom_{\pi}(\CI_{\CZ'}, \CI_{\CZ}(\CD_{\beta}) ) \Big) \]
can be written as a universal polynomial in descendents
\begin{equation} \ch_k^{\CO_{\CZ}}(\gamma)\, ,\   \ch_{k}^{\CO_{\CZ'}}(\gamma)\,, \ \ch_k^{\CO(\CD_{\beta})}(\gamma)\, . \label{descendent classes} \end{equation}
Clearly, we have $j^{\ast} \ch_k^{\CO_{\CZ}}(\gamma) = \ch_k^{\CO_{\CZ}}(\gamma)$.
Since $\BG = \CI_{\CZ} - \CI_{\CZ'}(-\CD_{\beta})$, we also have
\[ \ch_k^{\BG}(\gamma) = j^{\ast}\left( 
-\ch_k^{\CO_{\CZ}}(\gamma) + (-1)^{k-1} \ch_k^{\CO(\CD_{\beta})}(\gamma)
+ \sum_{r+s=k} (-1)^r \ch_{r}^{\CO_{\CZ'}}(\gamma \delta_i) \ch_{s}^{\CO(\CD_{\beta})}(\delta_i^{\vee}) \right)\, , \]
where $\sum_i \delta_i \boxtimes \delta_i^{\vee}$ is a K\"unneth decomposition of the diagonal in $S^2$. 
It follows that any descendent integral
over $[\Quot_{\alpha}^{\flat} ]^{\red}$ can be written as
a descendent integral over $[S^{[n']} \times S^{[n]}] \times [H_{\beta}]^{\red}$
in a universal way.
The claim then follows immediately 
from the universality of descendent integrals over the Hilbert scheme \cite{EGL}
and the universality of descendent integrals over $[S_{\beta}]^{\red}$ of  Theorem~\ref{thm:universality Hilbert scheme symplectic} proven in Appendix B.

To describe the argument more
precisely, let us view the pair $(a,b)$ as defining a linear map
\[ \ch_{(a,b)} : H^{\ast}(S,\BQ)^{\otimes \ell} \to H^{\ast}(\Quot_{\alpha}^{\flat},\BQ) \]
given by
\[ \ch_{(a,b)}(\gamma_1 \boxtimes \ldots \boxtimes \gamma_{\ell})
=
\prod_{i=1}^{\ell} \ch_{a(i)}^{b(i)}(\gamma_i)\, , \]
where we have written $a(i)$ for $a_i$ indicating that $a$ can be viewed as a map $\{ 1, \ldots, \ell \} \to \BZ_{\geq 0}$.
Then, there exists a list of 
triples $(a_{\kappa},b_{\kappa},\eta_{\kappa})$ depending only on $(a,b)$ and $n,n'$, 
\begin{itemize}
\item $a_{\kappa} : \{ 1, \ldots, \ell_{\kappa} \} \to \BZ_{\geq 0}$
and $b_{\kappa} : \{ 1, \ldots, \ell_{\kappa} \} \to \{ \CO_{\CZ}, \CO_{\CZ'}, \CO(\CD_{\beta}) \}$ are functions, which we view as defining a morphism
\[ \ch_{(a_{\kappa},b_{\kappa})} : H^{\ast}(S^{\ell_{\kappa}},\BQ) \to H^{\ast}(S^{[n']} \times S^{[n]} \times H_{\beta},\BQ) \]
using
\[
\ch_{(a_{\kappa},b_{\kappa})}(\gamma_1 \boxtimes \ldots \boxtimes \gamma_{\ell_{\kappa}}) = 
\prod_{i=1}^{\ell_{\kappa}} \ch_{a_{\kappa}(i)}^{b_{\kappa}(i)}(\gamma_i)\, , \]
\item operators $\eta_{\kappa} : H^{\ast}(S,\BQ)^{\otimes \ell} \to H^{\ast}(S,\BQ)^{\otimes \ell_{\kappa}}$ 
which act by pulling back a class $\alpha \in H^{\ast}(S^{\ell})$ along a projection map $S^{r} \to S^{\ell}$ for some $r$, multiplying with a polynomial in diagonals and pullbacks of $c_2(S)$ from factors, and pushforwards along a projection map $S^{r} \to S^{\ell_{\kappa}}$,
\end{itemize}
such that, for every $\Gamma \in H^{\ast}(S^{\ell}\,\BQ)$,
we have
\[
\int_{ \left[ \Quot^{\flat}_{\alpha} \right]^{\red} }
\ch_{(a,b)}(\Gamma)
=
\sum_{\kappa}
\int_{[S^{[n']} \times S^{[n]}] \times [H_{\beta}]^{\red}}
\ch_{(a_{\kappa},b_{\kappa})}( \eta_{\kappa}(\Gamma))\, . \]

By the results of \cite{EGL}, 
see also \cite[Section 4]{KT} where the case of integrating over an arbitrary base is discussed,
the factors $S^{[n]}$ and $S^{[n']}$ can be integrated out in an universal way, which depends only on the intersection pairings of the cohomology insertions of 
$\ch_k^{\CO_{\CZ}}$ and $\ch_{k}^{\CO_{\CZ'}}$. 
More precisely for each $\kappa$ there exists a further list of pairs $(a_{\kappa \lambda}, \xi_{\kappa \lambda})$ depending only on $(a_{\kappa},b_{\kappa})$, $n$, $n'$ where
\begin{itemize}
\item $a_{\kappa \lambda} : \{ 1 ,\ldots, \ell_{\kappa \lambda} \} \to \BZ_{\geq 0}$ 
is a function viewed as defining a morphism
\[ \ch_{(a_{\kappa \lambda})} : H^{\ast}(S^{\ell_{\kappa \lambda}},\BQ) \to H^{\ast}(H_{\beta},\BQ)\, , \quad \ch_{(a_{\kappa \lambda})}(\gamma_1 \boxtimes \ldots \boxtimes \gamma_{\ell_{\kappa \lambda}})
=
\prod_{i=1}^{\ell_{\kappa \lambda}} \ch_{a_{\kappa \lambda}(i)}(\gamma_i)\, ,
\]
\item operators $\xi_{\kappa \lambda} : H^{\ast}(S^{\ell_{\kappa}}) \to H^{\ast}(S^{\ell_{\kappa \lambda}})$ 
defined by the same operations as $\eta_{\kappa}$ above,
\end{itemize}
such that, for every $\Gamma \in H^{\ast}(S^{\ell_{\kappa}})$, we have
\[
\int_{[S^{[n']} \times S^{[n]}] \times [H_{\beta}]^{\red}}
\ch_{(a_{\kappa},b_{\kappa})}( \Gamma) 
=
\sum_{\lambda} \int_{[H_{\beta}]^{\red}}
\ch_{(a_{\kappa \lambda})}( \xi_{\kappa \lambda}(\Gamma))\, .
\]
By Theorem~\ref{thm:universality Hilbert scheme symplectic} of Appendix B, the last integral can be replaced by its $\varphi$-image.
Moreover, diagonal pushforward, multiplication by $c_2(S)$, and pushforward all commute with $\varphi$. So we have $\varphi \circ \eta_{\kappa} = \eta_{\kappa} \circ \varphi$ and $\varphi \circ \xi_{\kappa \lambda} = \xi_{\kappa \lambda} \circ \varphi$. Hence, we obtain, for every $\Gamma \in H^{\ast}(S^{\ell},\BQ)$,
\begin{align*}
\int_{ \left[ \Quot^{\flat}_{\alpha} \right]^{\red} }
\ch_{(a,b)}(\Gamma)
& = 
\sum_{\kappa, \lambda}
\int_{[H_{\beta}]^{\red}}
\ch_{(a_{\kappa \lambda})}( \xi_{\kappa \lambda}(\eta_{\kappa}(\Gamma))) \\
& = 
\sum_{\kappa,\lambda}
\int_{[H_{\varphi(\beta)}]^{\red}}
\ch_{(a_{\kappa \lambda})}( \varphi\xi_{\kappa \lambda}\eta_{\kappa}(\Gamma)) \\
& = 
\sum_{\kappa, \lambda}
\int_{[H_{\varphi(\beta)}]^{\red}}
\ch_{(a_{\kappa \lambda})}( \xi_{\kappa \lambda}\eta_{\kappa}\varphi(\Gamma)) \\
& = 
\int_{ \left[ \Quot^{\flat}_{\varphi(\alpha)} \right]^{\red} }
\ch_{(a,b)}(\varphi(\Gamma))\, ,
\end{align*}
which completes the argument.
\end{proof}

\section{Multiple cover formulas for PT invariants}
\label{dtptred}
\subsection{Definition}
Let $S$ be a symplectic surface,  let $\beta \in H_2(S,\BZ)$ be an effective curve class, and let $z=(z_1,\ldots, z_{m})$ be a tuple of distinct points of $\p^1$. Consider the relative geometry
\[ (X, S_{z}), \quad X = S \times \p^1\, , \quad S_z = S_{z_1,\ldots,z_{m}} = \sqcup_j S \times \{ z_j \}\, . \]
Let $P_{\ch_3,(\beta,n)}( X, S_{z})$ be the moduli space of relative stable pairs on $(X,S_z)$ which parameterizes stable pairs $(F,s)$ on $l$-step expansions $X[l]$ of $X$ satisfying
\[ \rho_{\ast} \ch_2(F) = (\beta,n), \quad \ch_3(F) = \ch_3\, , \]
where $\rho : X[l] \to X$ is the map contracting the bubbles.
The moduli space of relative stable pairs compactifies the open locus in the usual moduli space of stable pairs $(F,s)$ on $X$ where $F$ meets the boundary divisors $S_z$ transversally and the cokernel of $s$ is disjoint from $S_z$, see \cite{LiWu}. 
There are evaluation maps
at the fibers
\[ \ev_{z_i} : P_{\ch_3,(\beta,n)}( X, S_{z}) \to S^{[n]}. \]
The moduli space $P_{\ch_3,(\beta,n)}( X, S_{z})$ carries a perfect obstruction theory of dimension $2n$,
which admits a surjective cosection, and thus gives rise to a reduced virtual class of dimension $2n+1$.

The universal target over the moduli space,
$$\pi : \CX \to P_{\ch_3,(\beta,n)}( X, S_{z})\, .$$
carries a universal stable pair  $\CO_{\CX} \to \BF$.
For $\zeta \in H^{\ast}(S \times \p^1)$, descendent classes are defined by
\[ \ch_a(\zeta) = \pi_{\ast}( \ch_a(\BF) \cdot \rho^{\ast}(\zeta) ) \in H^{\ast}(P_{\ch_3,(\beta,n)}( X, S_{z})) \]
where
$\rho : \CX \to X$ the map forgetting the subscheme and contracting the bubbles of the expansions.
If $a<2$, then $\ch_a(\zeta)=0$, so we will always assume the PT descendent index satisfies $a\geq 2$.

Given $\lambda^1,\ldots,\lambda^{m} \in H^{\ast}(S^{[n]})$, we define the relative PT invariants by
\[
\big\langle \, \lambda^1,\ldots,\lambda^{m} \, \big| \, \ch_{\ellt_1}(\zeta_1) \cdots \ch_{\ellt_r}(\zeta_r) \big\rangle^{(S \times \p^1,S_z), \PT}_{\ch_3,\beta}
=
\int_{[  P_{\ch_3,(\beta,n)}( S \times \p^1, S_{z_1,\ldots,z_{m}}) ]^{\red}}
\prod_{i=1}^{m} \ev_{z_i}^{\ast}(\lambda^i) \cdot \prod_{i=1}^{r} \ch_{\ellt_i}(\zeta_i)\, .
\]

\subsection{Multiple cover formulas}
A {\em cohomology-weighted} partition $\lambda = ((\lambda_1,\delta_1), \ldots, (\lambda_\ell,\delta_\ell))$
of a non-negative integer $n$ satisfies the following
conditions:
$$\lambda_1 \geq \lambda_2 \geq \ldots\geq \lambda_\ell \geq 1\, , \ \ \
\sum_{i=1}^\ell \lambda_i =n\, , \ \ \
\delta_i \in H^*(S)\, .$$ 
To every cohomology weighted partition $\lambda=(\lambda_i,\delta_i)$ there is an associated class
in the cohomology of the Hilbert scheme $S^{[n]}$,
\begin{equation} \label{nakclass}
\lambda = \frac{1}{\prod_{i=1}^{m} \lambda_i} \prod_{i=1}^{\ell} \Fq_{\lambda_i}(\delta_i) \vacuum \ \in H^{\ast}(S^{[n]})\, , 
\end{equation}
where $\Fq_i(\gamma)$ are the Nakajima operators following the conventions of \cite[Example 2.8]{Marked}.
The classes \eqref{nakclass}, where $\lambda$ runs over all $H^{\ast}(S)$-weighted partitions of $n$ (with weights in a basis of
the cohomology of $S$),
generate the cohomology of the Hilbert scheme $S^{[n]}$. 
The degree of $\lambda$ is
$$\deg_{\mathbb{C}}(\lambda) = n-\ell + \sum_{i=1}^\ell \deg_{\BC}(\delta_i)\, .$$

If $\varphi : H^{\ast}(S, \BQ) \to H^{\ast}(S',\BQ)$ is a degree-preserving $\mathbb{Q}$-algebra isomorphism preserving the point class, 
we define $$\varphi(\lambda) = (\lambda_i, \varphi(\delta_i))\, .$$
We then obtain a well-defined $\mathbb{Q}$-linear isomorphism 
\begin{equation}\label{fvr}
\varphi : H^{\ast}(S^{[n]},\BQ) \to H^{\ast}((S')^{[n]},\BQ)\, .
\end{equation}
In fact, the map \eqref{fvr} is also
a $\mathbb{Q}$-algebra isomorphism preserving
the point class,
but 
we will not need the $\mathbb{Q}$-algebra claim.

For our study of the multiple cover formulas, we will only require  cohomology insertions of the 
form{\footnote{For notational convenience, we omit the pullback by the projections.}}
\[ \zeta = \omega \cup \gamma\, , \]
where $\omega \in H^2(\p^1,\BQ)$ is the point class and $\gamma \in H^{\ast}(S,\BQ)$.
The main result of  Section \ref{dtptred}
is a multiple cover formula for PT invariants.

\begin{thm}[Multiple Cover Formula]  \label{thm:PT MCF} We have
\begin{multline*}
\big\langle \, \lambda^1,\ldots,\lambda^{m} \, \big| \, \prod_{i=1}^{r} \ch_{\ellt_i}(\omega \gamma_i) \big\rangle^{(S \times \p^1,S_z), \PT}_{\ch_3,(\beta,n)}
=\\
\sum_{k|(\beta,\ch_3)} (-1)^{(k-1) \frac{\ch_3}{k}}
\big\langle \, \varphi_k(\lambda^1),\ldots,\varphi_k(\lambda^{m}) \, \big| \, \prod_{i=1}^{r} \ch_{\ellt_i}(\omega \varphi_k(\gamma_i)) \big\rangle^{(S_k \times \p^1,(S_k)_z), \PT}_{\frac{\ch_3}{k},(\varphi_k(\beta/k),n)}.
\end{multline*}
\end{thm}

We will also prove a multiple cover formula for the $\mathbb{C}^*$-equivariant theory of the quasiprojective Calabi-Yau threefold $S \times \BC$, where $\BC^{\ast}$ acts by scaling on the second factor with tangent weight $t$ at $0\in \mathbb{C}$.
Let $$\iota : S \to S \times \BC$$ be the inclusion of the zero section, and let
$[\mathbf{0}] \in H^2_{\BC^*}(\mathbb{C})$ be the equivariant cycle class of the point $0 \in \mathbb{C}$. As before, let
 $\gamma_1,\ldots, \gamma_r \in H^{\ast}(S)$.
The PT invariants of $S\times \BC$ are defined by $\mathbb{C}^*$-equivariant localization 
of the reduced equivariant virtual class\footnote{$P_{\ch_3,\iota_{\ast}\beta}(Y)$ is the moduli space of stable pairs with $\ch_2(F)=\iota_{\ast}\beta$ and $\ch_3(F)=\ch_3$,
where $\ch_3(F) \in \BZ$ is defined since $F$ has proper support. 
We have $\ch_3(F) = \chi(F) \in \BZ$ since $S \times \BC$ is Calabi-Yau.}:
\[ \left\langle \, \prod_{i=1}^{r} \ch_{\ellt_i}([\mathbf{0}] \gamma_i) \right\rangle^{S\times \BC, \PT}_{\ch_3,\beta}
=
\int_{[ P_{\ch_3,\iota_{\ast}\beta}(Y) ]^{\red}}
\prod_{i=1}^{r}
\ch_{\ellt_i}([\mathbf{0}]\gamma_i)\, .
\]

\begin{thm}[$\BC^*$-equivariant Multiple Cover Formula] 
\label{thm: PT SxC MCF} 
We have
\[
\left\langle \, \prod_{i=1}^{r} \ch_{\ellt_i}([\mathbf{0}] \gamma_i) \right\rangle^{S\times \BC, \PT}_{\ch_3,\beta}
=
\sum_{k|(\ch_3,\beta)}
(-1)^{(k-1) \frac{\ch_3}{k}}
k^{-1+\sum_i(\deg_{\BC}(\gamma_i) + \ellt_i - 2)}
\left\langle \, \prod_{i=1}^{r} \ch_{\ellt_i}([\mathbf{0}] \varphi_k(\gamma_i)) \right\rangle^{S_k\times \BC,\PT}_{\frac{\ch_3}{k},\varphi_k(\beta/k)}\, .
\]
\end{thm}

For both Theorems \ref{thm:PT MCF} and \ref{thm: PT SxC MCF}, we follow
the conventions of Conjectures \ref{conj:MCF} and \ref{conj:Taut} for the $\mathbb{Q}$-algebra isomorphism
$$\varphi_k: H^*(S,\mathbb{Q}) \rightarrow H^*(S_k,\mathbb{Q})\, .$$

\subsection{DT/PT correspondence}
The multiple cover formula for the DT theory of $(S \times \p^1,S_{z})$ was proven in \cite{QuasiK3}.
Theorem~\ref{thm:PT MCF} would therefore follow immediately from a DT/PT correspondence. We will instead prove the multiple cover formula on the PT side 
directly. As a byproduct, we will obtain the DT/PT correspondence.

Let $\Hilb_{\ch_3,(\beta,n)}(X,S_{z})$ 
be the relative Hilbert scheme of curves $C$ on expansions $X[l]$ satisfying
\[
\rho_{\ast} \ch_2(\CO_Z) = (\beta,n) \in H_2(X,\BZ)\, , \quad \ch_3(\CO_Z) = \ch_3\, .
\]
As in the case of stable pairs, the moduli space has evaluation maps, a reduced virtual class of dimension $2n+1$, and descendent classes
\[ \ch_a(\zeta) = \pi_{\ast}( \ch_a(\CO_{\CZ}) \cdot \rho^{\ast}(\zeta) ) \in H^{\ast}(\Hilb_{\ch_3,(\beta,n)}( X, S_{z}))\, , \]
where $\CO_{\CX} \to \CO_{\CZ}$ is the universal subscheme over the universal target
\[ \pi : \CX \to \Hilb_{\ch_3,(\beta,n)}( S \times \p^1, S_z)\, . \]
If $a<2$, then $\ch_a(\zeta)=0$, so we will always assume the DT descendent index satisfies $a\geq 2$.

Given $\lambda^1,\ldots,\lambda^{m} \in H^{\ast}(S^{[n]})$, we define the relative DT invariants by
\[
\big\langle \, \lambda^1,\ldots,\lambda^{m} \, \big| \, \ch_{\ellt_1}(\zeta_1) \cdots \ch_{\ellt_r}(\zeta_r) \big\rangle^{(S \times \p^1,S_z), \DT}_{\ch_3,\beta}
=
\int_{[  \Hilb_{\ch_3,(\beta,n)}( S \times \p^1, S_{z_1,\ldots,z_{m}}) ]^{\red}}
\prod_{i=1}^{m} \ev_{z_i}^{\ast}(\lambda^i) \cdot \prod_{i=1}^{r} \ch_{\ellt_i}(\zeta_i)\, .
\]

\begin{thm}[DT/PT correspondence] \label{thm:dtpt} For every effective curve class $\beta \in H_2(S,\BZ)$,
\[
\big\langle \, \lambda^1,\ldots,\lambda^{m} \, \big| \, \ch_{\ellt_1}(\omega \gamma_1) \cdots \ch_{\ellt_r}(\omega \gamma_r) \big\rangle^{(S \times \p^1,S_z), \DT}_{\ch_3,\beta}
=
\big\langle \, \lambda^1,\ldots,\lambda^{m} \, \big| \, \ch_{\ellt_1}(\omega \gamma_1) \cdots \ch_{\ellt_r}(\omega \gamma_r) \big\rangle^{(S \times \p^1,S_z), \PT}_{\ch_3,\beta}\,
 .
\]
\end{thm}
\begin{rmk}
Theorem~\ref{thm:dtpt} has been proven by \cite{NesterovK3} in case
there are at least 3 relative fibers and there are {\em no} descendent insertions.
The cases with fewer than 3 relative fiber require the DT/PT correspondence for a vertex term, which was not proven in \cite{NesterovK3}. Proving the DT/PT correspondence for the vertex term is essentially what we do below.
\end{rmk}

\subsection{Cap geometry} \label{capg}
The study of the {\em cap geometry},  
\[ (X, S_{\infty}), \quad X = S \times \p^1\, , \quad \infty \in \p^1\, , \]
will play a crucial role in our proof of Theorem \ref{thm:PT MCF}.

Let $\mathbb{C}^*$ act on $\p^1$ with fixed points $0,\infty\in \p^1$ and tangent weight $t$ at $0 \in \p^1$, where $t$ is the Chern class
of the standard representation $\mathfrak{t}$ of $\mathbb{C}^*$.
The $\mathbb{C}^*$-action lifts to the moduli space of relative stable pairs $P_{\Gamma}(X,S_{\infty})$, where $\Gamma=(\ch_3,(\beta,n))$.
The reduced virtual class naturally lifts to a $\mathbb{C}^*$-equivariant class: the perfect obstruction theory admits a canonical $\mathbb{C}^*$-equivariant lift and the cosection is $\mathbb{C}^*$-invariant so also lifts.

The virtual localization formula \cite{GP} yields
\begin{equation}
\label{locform}
[ P_{\Gamma}(X,S_{\infty}) ]^{\red} = \sum_{Z} \frac{[ Z ]^{\red}}{e(N_Z^{\vir})} \, ,\end{equation}
where $Z$ are the fixed loci, $[Z]^{\red}$ is the reduced virtual class of the fixed perfect obstruction theory,
and $N_{Z}^{\vir}$ is the virtual normal bundle. 

For the proof of the multiple cover formula for the cap geometry, we are
specifically interested in the following consequence of the 
virtual localization formula: 
\begin{equation}
\label{locform invariant}
\big\langle \, \lambda \, \big| \, \ch_{\ellt_1}([0] \gamma_1) \cdots \ch_{\ellt_r}([0] \gamma_r) \big\rangle^{(S \times \p^1,S_\infty), \PT}_{\ch_3,\beta}
 = \sum_{Z} \int_{[ Z ]^{\red}} \frac{\ev_{\infty}^{\ast}(\lambda)\cdot  \ch_{\ellt_1}([0] \gamma_1) \cdots \ch_{\ellt_r}([0] \gamma_r)}{e(N_Z^{\vir})}\, .\end{equation}
Since the reduced virtual dimension of
$P_{\Gamma}(X,S_{\infty})$ is $2n+1$, 
we will consider \eqref{locform} only in cases where the following
dimension constraint holds: 
\begin{equation} \label{dim constraint inequality} \deg_\BC(\lambda) + \sum_i \deg_\BC(\ch_{\ellt_i}([0] \gamma_i)) 
\geq 2n+1\, . \end{equation}
In all other cases, the integral on the left vanishes (since $P_{\Gamma}(X,S_{\infty})$ is proper).
We can expand the dimension constraint \eqref{dim constraint inequality} as
$$\deg_\BC(\lambda) + \sum_i \deg_\BC(\ch_{\ellt_i}([0] \gamma_i)) = 
n-\ell(\lambda) +\sum_{j=1}^{\ell(\lambda)} \deg_{\BC}(\delta_j) + \sum_{i=1}^r(\ellt_i-2+ \deg_{\mathbb{C}}(\gamma_i)) \geq 2n+1\,,$$
where $\ell(\lambda)$ is the length of the partition $\lambda$.

The $\BC^*$-fixed loci on the right decompose into three types of (not necessarily connected) components,
\[
P_{\Gamma}(X,S_{\infty})^{\BC^*} = P_{\Gamma}^{\mathrm{main}} \sqcup P_{\Gamma}^{\mathrm{mixed}} \sqcup P_{\Gamma}^{\mathrm{rubber}}\, , \]
parameterizing the following types of $\BC^*$-fixed stable pairs:
\begin{itemize}
	\item[$\bullet$] {\sf{Main type}:} $(F,s) \in P^{\mathrm{main}}_{\Gamma}$ are $\BC^*$-fixed relative stable pairs on $X$ without expansion.
	\item[$\bullet$] {\sf{Mixed type}:} $(F,s) \in P^{\mathrm{mixed}}_{\Gamma}$ are stable pairs on expansions $X[l]$ with $l>0$  where $$\ch_2(F|_{X})=(\beta',n)\,\,  \text{with}\, \,  \beta' \neq 0\, \,  \text{or}\, \,  \ch_3(F|_{X}) > 0\, .$$
	\item[$\bullet$] {\sf{Rubber type}:} $(F,s) \in P^{\mathrm{rubber}}_{\Gamma}$ are stable pairs on expansions $X[l]$ with $l>0$ where
    $$\ch_2(F|_{X}) = (0,n)\,\,  \text{and}\, \,  \ch_3(F|_{X}) = 0\, .$$
\end{itemize}
We will separately analyze the three types of components (in reverse order) in Sections \ref{rubbertype}--\ref{maintype}.

\subsection{Rubber type}
\label{rubbertype}
The rubber type component is isomorphic to the moduli space of stable pairs 
$$\Prubber_\Gamma  = P_{\Gamma}(X,S_{0,\infty})^{\sim}$$
on the rubber space $(X,S_{0,\infty})^\sim$,
which is of reduced virtual dimension $2n$.
The virtual normal bundle of $$
P_{\Gamma}(X,S_{0,\infty})^{\sim}\subset
P_\Gamma(X,S_{\infty})$$
is $N = \CL_0 \otimes T_{\p^1,\infty}$, where $\CL_0$ is the relative cotangent line bundle over the rubber space for the boundary divisor $S_0 \subset X$. Let $\Psi_0 = c_1(\CL_0)$.
The contribution 
of
$P_{\Gamma}(X,S_{0,\infty})^{\sim}$
to the virtual localization formula \eqref{locform} for the reduced
virtual class of $P_\Gamma(X,S_\infty)$
is
\[
\frac{1}{-t - \Psi_0} \cap [ P_{\Gamma}(X,S_{0,\infty})^{\sim} ]^{\red}\, .
\]
The descendent classes in \eqref{locform invariant} restrict to the rubber component as follows:

\begin{lemma} \label{lemma:restriction1} We have
$\ch_{a}([\mathbf{0}]\delta)|_{\Prubber_{\Gamma}} = \ev_{S_0}^{\ast}(\ch^{\mathcal{O}_{\mathcal{Z}}}_{a}(\delta))$.
\end{lemma}
\begin{proof}
Let $\CX \to \Prubber_{\Gamma}$ be the restriction of the universal target to the rubber component.
There is a natural decomposition $$\CX = (X \times \Prubber) \cup_{S \times \Prubber}\CX' $$ 
where $\CX' \to P_{\Gamma}(X,S_{0,\infty})^{\sim}$ is the universal target over the rubber.
After taking the associated normalization sequence and tensoring with the universal sheaf $\BF$, we obtain
an equality of $K$-theory classes
\[ \BF = \BF' + \BF|_{X} - \BF|_{S_{\infty}} \]
Since we are on the rubber component, we have $\BF|_{X} = (\ev_{S_0} \times p_S)^{\ast}({\mathcal{O}_{\mathcal{Z}}})$ where ${\mathcal{O}_{\mathcal{Z}}}$ is the structure sheaf of the universal subscheme $\mathcal{Z} \subset S^{[n]} \times S$.
Let $$\iota_{0} : S \times \Prubber \to X \times \Prubber \to \CX$$ be the inclusion of the fiber over $0$. Then,
\[ \ch(\BF)\cdot p_X^{\ast}(\delta [\mathbf{0}]) 
= \iota_{0 \ast}( \iota_{0}^{\ast}(\ch(\BF))\cdot p_S^{\ast}(\delta))
=
\iota_{0 \ast}( (\ev_0 \times \id_S)^{\ast}(\ch({\mathcal{O}_{\mathcal{Z}}})\cdot  p_S^{\ast}(\delta)) )\, .
\]
After pushing forward by $\pi : \CX \to \Prubber_{\Gamma}$, we obtain
\[
\ch_{a}([\mathbf{0}] \delta) = \rho_{\ast}((\ev_0 \times \id_S)^{\ast}(\ch({\mathcal{O}_{\mathcal{Z}}})\cdot p_S^{\ast}(\delta)) )
= \ev_0^{\ast}q_{\ast}( \ch_a({\mathcal{O}_{\mathcal{Z}}})\cdot p_S^{\ast}(\delta)) = \ch_{a}^{\mathcal{O}_{\mathcal{Z}}}(\delta)\, ,
\]
where $\rho : S\times \Prubber  \to \Prubber$ and $q:S \times S^{[n]} \to S^{[n]}$ are the projections.
\end{proof}

\begin{lemma}\label{rubvan}
The contribution of the rubber type
components to \eqref{locform invariant} vanishes.
\end{lemma}

\begin{proof}
By Lemma~\ref{lemma:restriction1}, the contribution from the rubber integral is
\begin{equation} \label{rubber contr}
\int_{[ P_{\Gamma}(X, S_{0,\infty} )^{\sim}]^{\red}}
\frac{ \ev_{S_{\infty}}^{\ast}(\lambda)\cdot \ev_{S_0}^{\ast}\left( \prod_{i=1}^{r} \ch_{\ellt_i}^{S^{[n]}}(\gamma_i)   \right)}{-t - \Psi_0}
\end{equation}
The rubber space $P_{\Gamma}(X,S_{\infty})^\sim$ has reduced virtual dimension $2n$,
but the class in the numerator of the integrand of \eqref{rubber contr}
imposes at least
$$\deg_{\BC}(\lambda) + \sum_{i=1}^r(\ellt_i-2+ \deg_{\mathbb{C}}(\gamma_i)) \geq 2n+1$$
conditions by \eqref{dim constraint inequality}.
\end{proof}

\subsection{Mixed type}
\label{mixedtype}
By resolving along the divisor $S_{\infty}$ on the main component $X \subset X[l]$, 
the mixed type component is seen to be of the form
\[
P_{\Gamma}^{\mathrm{mixed}} = \bigsqcup_{\substack{\beta = \beta_1 + \beta_2 \\ m=m_1+m_2}}
P_{\Gamma_1}^{\mathrm{rubber}} \times_{S^{[n]}} P_{\Gamma_2}^{\mathrm{main}}.
\]
where $\Gamma_i=(m_i,(\beta_i,n))$ as before. The virtual  class splits accordingly,
\[
[ P_{\Gamma}^{\mathrm{mixed}} ]^{\vir}=
\sum \Delta^{!}\left( [ P_{\Gamma_1}^{\mathrm{rubber}} ]^{\vir} \times [ P_{\Gamma_2}^{\mathrm{main}}]^{\vir} \right).
\]
By a cosection argument, the induced fixed perfect obstruction theory of $P_{\Gamma}^{\mathrm{mixed}}$ has two non-trivial cosections, while for the reduced class we have removed only one of them.
The fixed reduced virtual class therefore vanishes:
\[
[ P_{\Gamma}^{\mathrm{mixed}} ]^{\red} = 0\, .
\]

\subsection{Main type}\label{maintype}
\subsubsection{Appearance of the perverse-nested Hilbert scheme}
Let $(F,s) \in P_{\Gamma}^{\mathrm{main}}$ be a $\BC^*$-fixed stable pair on $X$ relative to $S_{\infty}$. 
Consider the decomposition
\[ \p^1 = \BA^1 \cup \BA^1_{\infty} \, , \]
where $\BA^1, \BA^1_{\infty}$ are the standard affine  charts around $0,\infty\in  \p^1$ respectively.

By the relative condition along  $S_{\infty}$, we have
\[ F|_{\BA^1_{\infty}} = p_S^{\ast}(\CO_z)\, , \quad s|_{\BA^1_\infty}=p_S^{\ast}(s_z) \, ,\]
for some $z \in S^{[n]}$ where $s_z : \CO_S \to \CO_z$ is the canonical surjection.
Here and below, $$p_S : Y \times S \to S$$
denotes the projection to $S$ from the product.

We write $\BA^1=\Spec \BC[x]$ and identify $\BC^*$-equivariant sheaves $\CF$ on 
$S \times \BA^1$ with the graded $\CO_S[x]$-modules given by $p_{S \ast} \CF$.
The restriction of the stable pair to $S \times \BA^1$ can then be written as
\[ F|_{\BA^1} = \bigoplus_{i \in \BZ} F_i \Ft^{-i}, \quad s|_{\BA^1} = \bigoplus_{i\in \BZ} s_i \, ,\]
with the following properties satisfied:
\begin{enumerate}
\item[$\bullet$] the $F_i$ are at most $1$-dimensional sheaves on $S$ of $\BC^*$-weight $0$,
\item[$\bullet$] the $F_i$ vanish (by coherence of $F$) for $i \ll 0$,
\item[$\bullet$] the sections $s_i : \CO_S \to F_i$ vanish (by equivariance) for $i<0$,
\item[$\bullet$] the $F_i$ are 0-dimensional for $i<0$ (by the cokernel condition of a stable pair).
\end{enumerate}
Multiplication by $x$ on $F$ is given by a morphism $${\bf{x}} : F_i \Ft^{-i} \to F_{i+1} \Ft^{-(i+1)}$$ 
of $\BC^*$-weight $-1$.
We arrange the data in a diagram:
\[
\begin{tikzcd}
\ldots \ar{r} & 0 \ar{d}{s_{-2}} \ar{r} & 0 \ar{d}{s_{-1}} \ar{r} & \CO_S \Ft^{0} \ar{d}{s_0} \ar{r}{ \id \Ft^{-1}} & \CO_S \Ft^{-1} \ar{d}{s_{1}} \ar{r}{\id \Ft^{-1}} & \cdots \\
\ldots \ar{r} & F_{-2} \Ft^{2} \ar{r}{{\bf{x}}} & F_{-1} \Ft^1 \ar{r}{{\bf{x}}}& F_0 \Ft^0 \ar{r}{{\bf{x}}} & F_1 \Ft^{-1} \ar{r}{{\bf{x}}} & \ldots 
\end{tikzcd}
\]
and obtain the following further properties:
\begin{enumerate}
\item[$\bullet$] 
$F_i = \CO_z$ for $i \gg 0$ (since $F|_{S \times (\BA^1 \setminus 0)} = \CO_z[x,x^{-1}]$),
\item[$\bullet$] the morphisms $F_{i} \stackrel{{\bf{x}}}{\longrightarrow} F_{i+1}$ are injective on 0-dimensional torsion for all $i$
(since $F$ is pure).
\end{enumerate}
The above properties imply that
$F_i \subset \CO_z$ for $i < 0$.
We use the notation 
$$I^{\bullet} = [\CO_X \to F]$$ and denote the $-i$-th graded piece by $I_i$, so 
\begin{equation} \label{defn Ii} I_i =
\begin{cases}
[\CO_S \to F_i] & \text{ if } i \geq 0, \\
F_i[-1] & \text{ if } i<0.
\end{cases}
\end{equation}

%

\begin{lemma} \label{Iprop} The following two properties hold:
\begin{itemize}
	\item[(a)] $I_i \in \Coh^{\sharp}(S)$ for all $i$,
	\item[(b)] the natural map $I_{i} \to I_{i+1}$ is injective as a morphism of $\Coh^{\sharp}(S)$ for all $i$.
\end{itemize}
\end{lemma}
\begin{proof}
Property (a) is clear since $\Coker(s_i)$ is 0-dimensional for all $i$.
Property (b) is an immediate direct check using Remark~\ref{rmk:cone explicit in coh sharp}.
From a more conceptual perspective,
both properties follow from the interpretation of
stable pairs as flat families in $\Coh^{\sharp}(S)$ given in \cite[Section 6.3]{Nesterov}.
\end{proof}

By Lemma \ref{Iprop},  $I^{\bullet}$ gives a flag of inclusions in $\Coh^{\sharp}(S)$,
\begin{equation} \label{flag}
\left[ 0 = I_{i_0} \hookrightarrow I_{i_0+1} \hookrightarrow \ldots \hookrightarrow I_{i_0+\ell} = I_z \right]
\end{equation}
where $i_0 \in \BZ$,   
$I_i = 0$ for all $i \leq i_0$, $\ell > 0$, and
$$I_{i_0+\ell} = [\CO_S \to \CO_z] = I_z\, .$$
By setting $E_i = \mathrm{Cone}(I_i \to I_z)$, we obtain a flag of surjections in $\Coh^{\sharp}(S)$:
\begin{equation} \label{gv44}
I_{z} = E_{i_0} \to E_{i_1} \to \ldots \to E_{i_0+\ell}\, . 
\end{equation}
Moreover, we have
\begin{equation}\label{chdata}
\ch(E_j) =
\begin{cases}
(1,0,-n + m_j) & \text{ if } j < 0\, , \\
(0,\beta_j, -n+m_j) & \text{ if } j \geq 0\, ,
\end{cases}
\end{equation}
where $m_i = \ch_2(F_i)$ and $\beta_i = \ch_1(F_i)$,
and 
\[ \sum_i \beta_i = \beta, \quad \ch_3 = \ch(F) = -(i_0 + \ell) n + \sum_{i=i_0}^{i_0+\ell-1} m_i\,  \]
see \cite{QuasiK3} for parallel constructions.
Conversely, every flag of surjections \eqref{gv44} with Chern data \eqref{chdata} 
gives rise to a $\BC^*$-fixed stable pair in $P^{\mathrm{main}}_{\Gamma}$.

Assume that $i_0 \ll 0$ and $\ell \gg 0$ are chosen such that, for every $(F,s) \in P^{\mathrm{main}}_{\Gamma}$, 
the restriction $(F,s)|_{\BA^1}$ is given by a flag of the form
\eqref{flag}. Such a choice is possible since we can choose $(i_0,\ell)$ uniformly on each connected component (of which there
are only finitely many), and we can extend a given flag by equalities on the left and right.
We conclude that
\begin{equation} \label{fixed locus quot iso}
P_{\Gamma}^{\mathrm{main}} = \bigsqcup_{\alpha=(\alpha_0,\ldots,\alpha_{\ell})} \Quot^{\sharp}_{\alpha}(\CI_z)\, ,
\end{equation}
where $\alpha_0=(1,0,-n)$ is fixed and $\alpha_1,\ldots,\alpha_{\ell}$ runs over suitable Chern characters (the precise form is not important here).

\begin{example}
Contrary to the case of ideal sheaves of curves (DT theory) studied in \cite{QuasiK3},
the complexes $I_i$ do not have to be simple (or even stable).

Let $C \subset S_0$ be a nonsingular curve, let $p \in C$ be a point, let $D = \{ p \} \times \p^1$, let
	\[ F = \CO_C \oplus \CO_{D}\, , \]
	and let $s=(s_C,s_D)$ where $s_C : \CO_X \to \CO_C$ and $s_D : \CO_X \to \CO_D$ are the canonical surjections.
	Then, $F_0 = \CO_C \oplus \BC_p$ and $F_{i} = \BC_p$ for $i > 0$.
	Let $I = I_0 = [\CO_S \to F_0]$.
	By considering the exact triangle $$h^0(I) = I_C \to I \to h^1(I)[-1] \to I_C[1]\, ,$$ and using $h^1(I) = \BC_p$, we obtain
	\[
	0 \to \Hom(\BC_p[-1], I) \to \Hom(I,I) \to \Hom(I_C,I) = \Hom(I_C,I_C) = \BC
	\]
	where the last map is surjective with natural splitting map $\lambda \mapsto \lambda \cdot \id_I$.
	However, applying $\Hom(\BC_p, - )$ to $I \to \CO_S \to F_0 \to I[1]$,  we obtain
	the injection $$\BC = \Hom(\BC_p, F_0) \to \Ext^1(\BC_p, I)\,,$$ so $\Hom(\BC_p[-1],I) \neq 0$. Therefore, $I$ is non-simple.
\end{example}

\subsubsection{Perfect-obstruction theory}
Let $P_{\Gamma}(X,S_{\infty})^{\circ} \subset P_{\Gamma}(X,S_{\infty})$ be the open locus parameterizing stable pairs on $X$ (and not on any expansion of $X$). The virtual tangent bundle of the standard (non-reduced) perfect obstruction theory on
$P_{\Gamma}(X,S_{\infty})^{\circ}$ is given by
\[
T^{\vir} = 
R\hom_{\pi}(\BI^{\bullet}, \BI^{\bullet})_0[1]\, ,
\]
where
$\pi : X \times P_{\Gamma}(X,S_{\infty})^{\circ} \to P_{\Gamma}(X,S_{\infty})^{\circ}$ is the projection and $\BI^{\bullet} = [\CO \to \BF]$ is the universal stable pair.

The restriction of $\BI$ to $\Pmain_{\Gamma}$ has a canonical $\BC^{\ast}$-equivariant structure. Consider the filtration
\[
\BI|_{\Pmain_{\Gamma} \times \BA^1} = \bigoplus_{i \in \BZ} \BI_{i} \Ft^{-i}
\]
into weight spaces: $\BI_i$ restricted to a point $(F,s)$ is just $I_i$, as defined in \eqref{defn Ii}.

\begin{prop} \label{prop:fixed pot}
The fixed and moving parts of the restriction of $T^{\vir}$ to $\Pmain_{\Gamma}$ have $K$-theory class
\begin{align*}
\left( T^{\text{vir}}|_{\Pmain_\Gamma} \right)^{\textup{fixed}} 
& =
-\sum_{i=i_0}^{i_0+\ell} R \hom_{\rho}(\BI_i, \BI_i) + \sum_{i=i_0}^{i_0+\ell-1} R \hom_{\rho}(\BI_i, \BI_{i+1})
+ R \Gamma(S,\CO_S)\, , \\
\left( T^{\text{vir}}|_{\Pmain_{\Gamma}} \right)^{\textup{mov}} 
	& = 
	\sum_{i \geq i_0} \sum_{k \geq 1} 
	\left( \begin{array}{c}
		- \Ft^{k} \otimes R\Hom_{\rho}( \BI_{i+k} - \BI_{i+k-1}, \BI_i) \\
		+ \Ft^{-k} \otimes R\Hom_{\rho}( \BI_{i+k+1} - \BI_{i+k}, \BI_i)^{\vee}
	\end{array} \right)\, ,
\end{align*}
where $\rho : S  \times \Pmain_{\Gamma}  \to \Pmain_{\Gamma}$ is the projection.
	%
\end{prop}
\begin{proof}
The calculation is exactly as in \cite[Proposition 4.1]{QuasiK3}.
\end{proof}

By a direct check (by rewriting $\BI_i$ in terms of the universal quotients)
 the isomorphism \eqref{fixed locus quot iso} the $K$-theory class
of the fixed obstruction theory of $(T^{\vir}|_{P_{\Gamma}^{\mathrm{main}}})^{\mathrm{fixed}}$
agrees with the $K$-theory class of the virtual tangent bundle of the nested Quot scheme $\Quot_{\alpha}^{\sharp}(S)$. The latter was computed{\footnote{Theorem \ref{thm:pot on nested quot scheme} also holds for the perverse nested Quot scheme, see Section~\ref{subsec:perverse t structure}).}}
in Theorem~\ref{thm:pot on nested quot scheme}. 
By Siebert's formula the (reduced) virtual class of a perfect obstruction theory only depends on the $K$-theory class of the virtual tangent bundle,
see \cite[Appendix C1]{PTKKV}.
We therefore conclude that  isomorphism \eqref{fixed locus quot iso} gives an equality of (1-)reduced virtual classes:
\[ [ P_{\Gamma}^{\mathrm{main}} ]^{\red} = \sum_{\alpha=(\alpha_0,\ldots,\alpha_{\ell})} [ \Quot^{\sharp}_{\alpha}(\CI_z) ]^{\red}\, . \]
By Theorem~\ref{thm:cosections on nested quot},
the Quot scheme $\Quot^{\sharp}_{\alpha}(\CI_Z)$ has as many linearly independent cosections as jumps in the filtration (not counting the jump from the zero sheaf). Since we have reduced precisely $1$-cosection from the obstruction sheaf on both sides,
we find that the corresponding reduced virtual class vanishes unless there is precisely $1$-jump.
The corresponding stable pairs are called {\em uniformly thickened}.

\begin{defn} \label{defn:unif thickened}
A $\BC^*$-fixed stable pair $I^{\bullet}=[\CO_X \to F]$ in $P_{\Gamma}^{\mathrm{main}}$ is {\em uniformly thickened} if
\[ {\bf{x}} : F_j \to F_{j+1} \]
is an isomorphism for all $j \in \BZ$ except for two indices $j_0 < j_1$.
\end{defn}

\begin{prop}\label{jumpvan}
Let $P \subset \Pmain_{\Gamma}$ be a connected component which parameterizes $\BC^*$-fixed stable pairs which are not uniformly thickened, then the induced reduced virtual class vanishes: $[P]^{\red}=0$.
\end{prop}

Proposition \ref{jumpvan} is precisely the analogue of the vanishing result of \cite{PTKKV} which was
used in the proof of the Katz-Klemm-Vafa formula in the imprimitive case.

\subsubsection{Uniformly thickened stable pairs}
If $(F,s)$ is uniformly thickened, then the indices $j_0,j_1 \in \BZ$ in Definition~\ref{defn:unif thickened} are characterized by $F_{j_0} = 0$, $F_{j_0+1} \neq 0$, and $F_{j_1} \neq \CO_z$, $F_{j_1+1} = \CO_z$.
Since $\beta$ is non-zero, there is a $k \geq 1$ such that
\[ F|_{S \times \BA^1} = \bigoplus_{i=0}^{k-1} F_0 \Ft^{-i} \oplus \bigoplus_{i=k}^{\infty} \CO_z \Ft^{-i}. \]
The stable pair $(F,s)$ is then called {$k$-times uniformly thickened} (since $F_0$ appears precisely $k$ times).
The corresponding flag \eqref{flag} is determined by the inclusion in $\Coh^{\sharp}(S)$,
\[ I_0 \subset I_z \]
where $\rk(I_0) = 1$. The stable pair therefore corresponds to a quotient $I_z \to G$ in $\Coh^{\sharp}(S)$ where
\[ G = \mathrm{Cone}(I_0 \to I_z) = \mathrm{Ker}(F_0 \twoheadrightarrow \CO_z). \]
By the discussion in Section~\ref{subsec:nested hilbert schemes},
$G$ is a pure $1$-dimensional sheaf and as a morphism in $\Coh(S)$, the map $I_z \to G$ has zero-dimensional cokernel.

\begin{lemma} \label{kthick}
If $(F,s)$ is $k$-times uniformly thickened, then
$\ch(G) = \frac{1}{k}(0,\beta,\ch_3)$.
\end{lemma}
\begin{proof}
Let $j_k : Z_k = V(x^k) \to X$ be the inclusion of the $k$-times thickening of the zero section $S_0 \subset X$, and let $p_k : Z_k \to S$ be the projection.
There are short exact sequences
\begin{equation} \label{ses}
\begin{gathered}
0 \to p_S^{\ast}(\CO_z) \otimes \CO_{\p^1}(-k) \to F \to j_{k \ast}p_k^{\ast}(F_0) \to 0\, , \\
0 \to p_S^{\ast}(\CO_z) \otimes \CO_{\p^1}(-k) \to p_S^{\ast}(\CO_Z) \to j_{k \ast}p_k^{\ast}(\CO_z) \to 0\, .
\end{gathered}
\end{equation}
In $K$-theory, we therefore have
\begin{align*} 
F & = j_{k \ast}p_k^{\ast}(F_0) + p_S^{\ast}(\CO_z) \otimes \CO_{\p^1}(-k)  \\
& = j_{k \ast}p_k^{\ast}(F_0) + p_S^{\ast}(\CO_z) - j_{k \ast}p_k^{\ast}(\CO_z) \\
& = j_{k \ast}p_k^{\ast}(G) + p_S^{\ast}(\CO_z)\\
& = k\cdot j_{1 \ast}(G) + p_S^{\ast}(\CO_z)\, .
\end{align*}
By taking Chern characters on both sides and comparing coefficients, the claim follows.
\end{proof}

Let $\Pmain_{\Gamma,k}$ the union of connected components of $k$-times uniformly thickened stable pairs.
By Lemma \ref{kthick}, for every $k | (\beta,\ch_3)$, isomorphism \eqref{fixed locus quot iso} restricts{\footnote{We use here (and for the remainder of Section \ref{dtptred}) the abbreviated notation 
$$\Quot^{\sharp}_{n,\frac{1}{k}(0,\beta,\ch_3)} =\Quot^{\sharp}_{((1,0,-n),\frac{1}{k}(0,\beta,\ch_3))}\, .$$}} to
\begin{equation}\label{isompq}
    \Pmain_{\Gamma,k} \cong \Quot^{\sharp}_{n,\frac{1}{k}(0,\beta,\ch_3)}\, , \quad (F,s) \mapsto [(I_z \to I_z/I_0=G)]\, . 
    \end{equation}
Moreover, as seen above, the corresponding reduced virtual classes match,
\[ [ \Pmain_{\Gamma,k} ]^{\red} = [ \Quot^{\sharp}_{n,\frac{1}{k}(0,\beta,\ch_3)} ]^{\red}. \]

\subsubsection{Virtual normal bundle}
The $\BC^*$-equivariant Euler class of the virtual normal bundle of
 $\Pmain_{\Gamma,k}$ is required for
the localization formula.

By Proposition~\ref{prop:fixed pot}, the moving part of the virtual tangent bundle of $\Pmain_{\Gamma,k}$ is
\begin{align*}
N^{\vir}_{\Pmain_{\Gamma,k}} = \left( T^{\vir}|_{\Pmain_{\Gamma,k}} \right)^{\mathrm{mov}}
= & - R\hom_{\rho}(\BI_0, \BG)^{\vee} \otimes (\Ft \oplus \Ft^{2} \oplus \cdots \oplus \Ft^{k}) \\
& + R \hom_{\rho}(\BI_0,\BG) \otimes (\Ft^{-1} \oplus \Ft^{-2} \oplus \cdots \oplus \Ft^{-(k-1)} )\, ,
\end{align*}
where $\CI_Z \to \BG$ is the universal quotient over $\Quot^\sharp_{n,\frac{1}{k}(0,\beta,\ch_3)}$, and $\BI_0=[\CI_z \to \BG]$.
We have 
\begin{align*}
\mathrm{rank}( R \hom_{\rho}(\BI_0,\BG) )
& = \chi( R\Hom_S(I_0,G)) \\
& = \int_{S} \ch(I_0)^{\vee} \ch(G) \cdot \mathrm{td}_S \\
& = \int_{S} (1, -\frac{\beta}{k}, -n  - \frac{\ch_3}{k})^{\vee} \cdot (0,\frac{\beta}{k}, \frac{\ch_3}{k}) \cdot (1,0,2) \\
& = \frac{\beta}{k} \cdot \frac{\beta}{k} + \frac{\ch_3}{k} \\
& \equiv \frac{\ch_3}{k} \text{ modulo } 2\, .
\end{align*}
The equivariant Euler class of the virtual normal bundle is hence
\begin{align*}
e_{\BC^{\ast}}(N^{\vir}_{\Pmain_{\Gamma,k}}) & =
\frac{e_{\BC^*}( R\hom_{\rho}(\BI_0,\BG) \otimes (\Ft^{-1} \oplus \cdots \oplus \Ft^{-(k-1)}) )}
{e_{\BC^*}(R\hom_{\rho}(\BI_0,\BG)^{\vee} \otimes (\Ft^{1} \oplus \cdots \oplus \Ft^{(k-1)}) )}
\cdot
e_{\BC^{\ast}}( - R\hom_{\rho}(\BI_0,\BG)^{\vee} \otimes \Ft^k) \\
& =
(-1)^{(k-1) \cdot \frac{\ch_3}{k}} e_{\BC^{\ast}}( - R\hom_{\rho}(\BI_0,\BG)^{\vee} \otimes \Ft^k).
\end{align*}
where we used that for an equivariant $K$-theory class
\[ e(V) = e(V^{\ast}) \cdot (-1)^{\mathrm{rank}(V)}. \]
The contribution 
of the $k$-step locus
$\Pmain_{\Gamma,k}$
to the virtual localization formula \eqref{locform} for the reduced
virtual class of $P_\Gamma(X,S_\infty)$
is
\begin{equation} \label{fsdfds}
\frac{1}{e_{\BC^{\ast}}(N^{\vir}_{\Pmain_{\Gamma,k}})} [ \Pmain_{\Gamma,k} ]^{\red}
=
(-1)^{(k-1) \cdot \frac{\ch_3}{k}} e_{\BC^{\ast}}( R\hom_{\rho}(\BI_0,\BG)^{\vee} \otimes \Ft^k) \cap [ \Quot^\sharp_{n, \frac{1}{k} (0,\beta,\ch_3)} ]^{\red}.
\end{equation}

\subsubsection{Summary}
The fixed locus of the main component $\Pmain_{\Gamma}$
is isomorphic to a union of perverse nested Hilbert schemes 
$\Quot^{\sharp}_{\alpha}(\CI_Z)$.
The isomorphism respects the reduced virtual classes.
By the existence of additional cosections, only the single-nested perverse Hilbert schemes $\Quot^\sharp_{n,\alpha}$,
where $\alpha=\frac{1}{k}(0,\beta,\ch_3)$ and $k|(\beta,\ch_3)$, contribute. By \eqref{fsdfds},
the total contribution of
$\Pmain_\Gamma$
to the localization
formula 
\eqref{locform} for the reduced
virtual class of $P_\Gamma(X,S_\infty)$
is
\[
\sum_{k|(\beta,\ch_3)} (-1)^{(k-1) \cdot \frac{\ch_3}{k}} e_{\BC^{\ast}}( R\hom_{\rho}(\BI_0,\BG)^{\vee} \otimes \Ft^k) \cap [ \Quot^{\sharp}_{n, \frac{1}{k} (0,\beta,\ch_3)} ]^{\red}\, .
\]

\subsection{Localization formula}
\label{lform}
We have now analyzed all three types of fixed loci of the $\BC^*$-action on 
$P_{\ch_3,(\beta,n)}(X,S_\infty)$.
We arrive at the following result.
\begin{thm} \label{thm:PT fixed}
The localization formula for the reduced virtual class in the $\BC^*$-equivariant Chow theory of
$P_{\ch_3,(\beta,n)}(X,S_\infty)$ is
\begin{multline*}
[ P_{\ch_3,(\beta,n)}(X, S_{\infty}) ]^{\red}
= \frac{1}{-t - \Psi_0} [ P_{\ch_3,(\beta,n)}(X, S_{0,\infty} )^{\sim}]^{\red} \\
+ \sum_{k|(\beta,\ch_3)} (-1)^{(k-1) \frac{\ch_3}{k}} e_{T}\left( R \hom_{\rho}( [\CI_z \to \BG], \BG)^{\vee} \otimes \Ft^{k} \right) \cap \left[ \Quot^{\sharp}_{n,\frac{1}{k}(0,\beta,\ch_3)} \right]^{\mathrm{red}}.
\end{multline*}
where $\CI_z \to \BG$ is the universal quotient sequence on $\Quot^{\sharp}_{n,\frac{1}{k}(0,\beta,\ch_3)}$.
\end{thm}

As a consequence of Theorem \ref{thm:PT fixed}, we will obtain a multiple cover formula for the PT invariants
of the cap geometry.

\begin{thm} \label{thm:PT MCF equivariant}
The $\BC^*$-equivariant PT invariants of $(S \times \p^1,S_{\infty})$ satisfy:
\[ \left\langle \, \lambda \, \middle| \, \prod_{i=1}^{r} \ch_{\ellt_i}([\mathbf{0}] \gamma_i) \right\rangle^{(X,S_{\infty}), \PT}_{\ch_3,\beta}
=
\sum_{k|(\ch_3,\beta)}
k^{\nu}(-1)^{(k-1) \frac{\ch_3}{k}}
\left\langle \, \varphi_k(\lambda) \, \middle| \, \prod_{i=1}^{r} \ch_{\ellt_i}([\mathbf{0}] \varphi_k(\gamma_i)) \right\rangle^{(X,S_{\infty}), \PT}_{\frac{\ch_3}{k},\varphi_k(\beta/k)}\, ,
\]
where the exponent $\nu$ is defined by
\begin{eqnarray*}
\nu &=& \deg_{\BC}(\lambda) + \sum_i \deg_{\BC}(\ch_{\ellt_i}([\mathbf{0}] \gamma_i)) - 2n - 1 \\
&=& \deg_{\BC}(\lambda) + \sum_{i} (\ellt_i - 2 + \deg_{\BC}(\gamma_i)) - 2n - 1 \,.
\end{eqnarray*}
\end{thm}
\begin{proof}
We assume the dimension constraint \eqref{dim constraint inequality} holds (otherwise both sides vanish). 
Then, by Lemma \ref{rubvan},
the rubber component does not contribute, so we need only study the contributions of the $k$-step loci
$\Pmain_{\Gamma,k}$ for each $k| (\ch_3,\beta)$.

To study the $k$-step loci, we will use  isomorphism \eqref{isompq},
$$\Pmain_{\Gamma,k} \cong \Quot^{\sharp}_{n,\frac{1}{k}(0,\beta,\ch_3)}\, .$$
We follow the descendent notation of Section \ref{subsec:nested hilbert schemes}:
\begin{enumerate}
\item[$\bullet$]
$\ch^{\BG}_a(\gamma)$ on
$\Quot^{\sharp}_{n,\frac{1}{k}(0,\beta,\ch_3)}$
with respect to the sheaf
$$ \mathbb{G} \rightarrow  S \times \Quot^{\sharp}_{n,\frac{1}{k}(0,\beta,\ch_3)}\, , $$
\item[$\bullet$] 
$\ch^{\mathcal{O}_{\mathcal{Z}}}_a(\gamma)$ 
on $S^{[n]}$
with respect to the sheaf
$$ {\mathcal{O}_{\mathcal{Z}}} \rightarrow S\times S^{[n]}\, .$$
\end{enumerate}
The descendents $\ch^{\mathcal{O}_Z}_a(\gamma)$ are
pulled-back to 
$\Quot^{\sharp}_{n,\frac{1}{k}(0,\beta,\ch_3)}$
via the structure map
$$
\Quot^{\sharp}_{n,\frac{1}{k}(0,\beta,\ch_3)}
\rightarrow S^{[n]}\, .
$$

\begin{lemma} \label{lemma:restriction2}
Let $\gamma \in H^{\ast}(S)$.
Under the isomorphism $\Pmain_{\Gamma,k} \cong \Quot^{\sharp}_{n,\frac{1}{k}(0,\beta,\ch_3)}$, we have
\[ \sum_{a \geq 0} \ch_{a}([\mathbf{0}]\gamma)|_{\Pmain_{\Gamma,k}}
= (1-e^{-k t}) \sum_{a \geq 0} \ch^{\BG}_a(\gamma) + \sum_{a \geq 0} \ch_a^{\mathcal{O}_{\mathcal{Z}}}(\gamma)\, . \]
\end{lemma}
\begin{proof}
By \eqref{ses}, we have
$\BF = j_{k \ast}p_k^{\ast}(\BG) + p_S^{\ast}({\mathcal{O}_{\mathcal{Z}}})$.
From the short exact sequence 
\[ 0 \to \BG \otimes \CO_{\p^1}(-k) \to \BG \to j_{k \ast}p_k^{\ast}(\BG) \to 0\, , \]
we obtain
$\ch(j_{k \ast}p_k^{\ast}(\BG)) |_{\Pmain_{\Gamma,k} \times S_0} = \ch(\BG) (1-e^{-kt})$ , 
which implies the claimed formula for the class
$\ch_a([{\mathbf{0}}]\gamma)|_{\Pmain_{\Gamma,k}}$.
\end{proof}

As an algebraic consequence of Lemma \ref{lemma:restriction2}, there exist universal polynomials,
$$P_{\ell}( c_{a,i}, c'_{a,i} )\in \mathbb{Q}[\{c_{a,i},c'_{a,i}\}]\, ,$$ homogeneous of degree
\[ \sum_{i} (\ellt_i+\deg_\BC(\gamma_i)-2) + \ell = \deg_\BC\left( \prod_{i=1}^{r} \ch_{\ellt_i}([\mathbf{0}]\gamma_i) \right) + \ell \]
with $c_{a,i}, c'_{a,i}$ of degree $a+\deg(\gamma_i)-2$, such that
\[ \left( \prod_{i=1}^{r} \ch_{\ellt_i}([\mathbf{0}]\gamma_i) \right)\Big|_{\Pmain_{\Gamma,k}} = 
\sum_{\ell \leq 0} (kt)^{-\ell} \cdot P_{\ell}( \ch_a^{\BG}(\gamma_i), \ch_a^{\mathcal{O}_{\mathcal{Z}}}(\gamma_i)) \, .
\]

By Theorem \ref{lform} and Lemma \ref{lemma:restriction2},
the contribution of $\Pmain_{\Gamma,k}$ to the localization formula \eqref{locform}  is 
\begin{equation} \label{Dfsd0-ir3rw}
\sum_{j,\ell \in \BZ}
(kt)^{-j-\ell}
(-1)^{(k-1) \frac{\ch_3}{k}}
\int_{[ \Quot^{\sharp}_{n,\frac{1}{k}(0,\beta,\ch_3)} ]^{\mathrm{red}}}
\ev_{S_{\infty}}^{\ast}(\lambda)\cdot c_{\rk(\CV)+j}(\CV) \cdot P_{\ell}( \ch_a^{\BG}(\gamma_i), \ch_a^{\mathcal{O}_{\mathcal{Z}}}(\gamma_i))\, ,
\end{equation}
where $\CV = R \hom_{\rho}( [\CI_z \to \BG], \BG)^{\vee}$ and{\footnote{For any $K$-theory class we have 
$e_{\BC^*}(V \otimes \Ft^k) = \sum_{j \in \BZ} c_1(\Ft^k)^{-j} c_{\rk(V)+j}(V)$ by the splitting principle.}}
$$
\rk(\CV) = (\beta/k)^2 + \ch_3/k\,, \ \ \ \ \
e_{\BC^*}(\CV \otimes \Ft^k) = \sum_{j \in \BZ} c_{\rk(\CV) + j}(V) (kt)^{-j}\, .$$
The reduced virtual dimension of $\Quot^{\sharp}_{n,\frac{1}{k}(0,\beta,\ch_3)}$ is
\[ (\beta/k)^2 + \ch_3/k + 2n+1 \, ,\]
and the degree of the integrand is
\begin{multline*}
\deg_{\BC}\left( \ev_{S_{\infty}}^{\ast}(\lambda)\cdot c_{\rk(V)+j}(\CV)\cdot P_{\ell}( \ch_a^{\BG}(\gamma_i), \ch_a^{\mathcal{O}_{\mathcal{Z}}}(\gamma_i)) \right)
= \\
\deg_{\BC}(\lambda) + \Big((\beta/k)^2 + \ch_3/k + j\Big) + 
\deg_{\BC}\left( \prod_{i=1}^{r} \ch_{\ellt_i}([\mathbf{0}]\gamma_i) \right) + \ell\, .
\end{multline*}
To obtain a non-zero contribution, we must have 
\[ j+\ell = 2n+1 - \deg_{\BC}(\lambda) - \sum_i \deg_{\BC}(\ch_{\ellt_i}([0]\gamma_i)) = - \nu\, . \]

By Theorem \ref{thm:universality}, the integrals in the summation of \eqref{Dfsd0-ir3rw} 
can be replaced by their $\varphi_k$-images.
As a result,
the sum \eqref{Dfsd0-ir3rw} equals
\begin{multline*}
(kt)^{\nu}
(-1)^{(k-1) \frac{\ch_3}{k}}
\sum_{\substack{j,\ell \in \BZ \\ j + \ell = -\nu}} 
\int_{[ \Quot^{\sharp}_{n,(0,\varphi_k(\beta/k),\ch_3/k)} ]^{\mathrm{red}}}
\ev_{S_{\infty}}^{\ast}(\varphi_k(\lambda))\cdot c_{\rk(\CV)+j}(\CV) \cdot P_{\ell}( \ch_a^{\BG}(\varphi_k(\gamma_i)), \ch_a^{\mathcal{O}_Z}(\varphi_k(\gamma_i)))\, .
\end{multline*}
Working backwards, the sum is just
\[
k^{\nu} (-1)^{(k-1) \frac{\ch_3}{k}}
\left\langle \, \varphi_k(\lambda) \, \middle| \, \prod_{i=1}^{r} \ch_{\ellt_i}([\mathbf{0}] \varphi_k(\gamma_i)) \right\rangle^{(X,S_{\infty}), \PT}_{\ch_3,\varphi_k(\beta/k)}.
\]

Finally, by summing up over all $k$-step contributions, for each $k|(\ch_3,\beta)$, we obtain
\begin{align*}
& \left\langle \, \lambda \, \middle| \, \prod_{i=1}^{r} \ch_{\ellt_i}([\mathbf{0}] \gamma_i) \right\rangle^{(X,S_{\infty}), \PT}_{\ch_3,\beta} 
& =
\sum_{k|(\ch_3,\beta)}  k^{\nu} (-1)^{(k-1) \frac{\ch_3}{k}}
\left\langle \, \varphi_k(\lambda) \, \middle| \, \prod_{i=1}^{r} \ch_{\ellt_i}([\mathbf{0}] \varphi_k(\gamma_i)) \right\rangle^{(X,S_{\infty}), \PT}_{\ch_3,\varphi_k(\beta/k)}\, ,
\end{align*}
which is the claimed multiple cover formula.
\end{proof}

\subsection{Proof of Theorem~\ref{thm: PT SxC MCF}}
There is an open equivariant embedding $$P_{\ch_3,\beta}(Y) \subset P_{\ch_3,(\beta,0)}(X,S_{\infty})\,. $$
The fixed locus of $P_{\ch_3,\beta}(Y)$ is precisely the main component of the fixed locus of $P_{\ch_3,(\beta,0)}(X,S_{\infty})$.
The claim follows from
the analysis of the main component as in the proof of Theorem~\ref{thm:PT MCF equivariant}.
\qed

\begin{rmk}
If we have at least one insertion, Theorem~\ref{thm: PT SxC MCF} is also a direct consequence of
Theorem~\ref{thm:PT MCF equivariant} in the special case where $n=0$ and $\lambda= (\emptyset)$.
Because of the insertion, the
rubber terms do not contribute to the integral of Theorem~\ref{thm:PT MCF equivariant}.
\end{rmk}

\subsection{Proof of Theorem~\ref{thm:PT MCF}}
We will deduce Theorem~\ref{thm:PT MCF}
from the case of the cap geometry by the degeneration formula.
The PT invariants of $(S \times \p^1,S_z)$ in curve classes $(0,n)$ defined through the standard (non-reduced) virtual class will be needed here. These invariants are denoted by
\begin{equation} \big\langle \, \lambda^1,\ldots,\lambda^{m} \, \big| \, \ch_{\ellt_1}(\omega \gamma_1) \cdots \ch_{\ellt_r}(\omega \gamma_r) \big\rangle^{(S \times \p^1,S_z), \PT}_{\ch_3,0}. 
\label{nonreduced invariants} \end{equation}
By a cosection argument, the invariants \eqref{nonreduced invariants} vanish if $\ch_3>0$, see \cite[4.3]{PPJap}. For vanishing $\ch_3$, we have the evaluation:
\[
\big\langle \, \lambda^1,\ldots,\lambda^{m} \, \big| \, \ch_{\ellt_1}(\omega \gamma_1) \cdots \ch_{\ellt_r}(\omega \gamma_r) \big\rangle^{(S \times \p^1,S_z), \PT}_{\ch_3=0,0}
=
\int_{S^{[n]}} \prod_{i} \lambda^i \cdot \prod_{i=1}^{r} \ch_{\ellt_i}^{S^{[n]}}(\gamma_i)\, .
\]
The above evaluation holds since $P_{0,(0,n)}(X,S_z) = S^{[n]}$, and the virtual class is just the fundamental class of $S^{[n]}$.


Assume $m > 1$, and let $\{ 1, \ldots, r \} = A_1 \sqcup A_2$ be any decomposition.
The degeneration formula for reduced invariants for the deformation to the normal cone of $S \times \p^1$ along the divisor $S_{z_{m}}$ is:
\begin{multline} \label{deg formula}
\big\langle \, \lambda^1,\ldots,\lambda^{m-1} \, \big| \, {\textstyle \prod_i} \ch_{\ellt_i}(\omega \gamma_i) \big\rangle^{(S \times \p^1,S_{z_1,\ldots,z_{m-1}}), \PT}_{\ch_3,\beta} \\
=
\sum_{j}
\big\langle \, \lambda^1,\ldots,\lambda^{m-1}, \nu_j \, \big| \, 
{\textstyle \prod_{i \in A_1}} \ch_{\ellt_i}(\omega \gamma_i)
 \big\rangle^{(S \times \p^1,S_{z_1,\ldots,z_{m}}), \PT}_{\ch_3,\beta}
\cdot 
\big\langle \, \nu_j^{\vee} \, \big| \, {\textstyle \prod_{i \in A_2}} \ch_{\ellt_i}(\omega \gamma_i)\big\rangle^{(S \times \p^1,S_{\infty}), \PT}_{0,0} \\
+ \sum_{j}
\big\langle \, \lambda^1,\ldots,\lambda^{m-1}, \nu_j \, \big| \, 
{\textstyle \prod_{i \in A_1}} \ch_{\ellt_i}(\omega \gamma_i)
 \big\rangle^{(S \times \p^1,S_{z_1,\ldots,z_{m}}), \PT}_{0,0}
\cdot 
\big\langle \, \nu_j^{\vee} \, \big| \, {\textstyle \prod_{i \in A_2}} \ch_{\ellt_i}(\omega \gamma_i)\big\rangle^{(S \times \p^1,S_{\infty}), \PT}_{\ch_3,\beta}
\end{multline}
where $\sum_i \nu_j \otimes \nu_j^{\vee} \in H^{\ast}(S^{[n]}) \otimes H^{\ast}(S^{[n]})$ is a K\"unneth decomposition of the diagonal of $S^{[n]}$.

The key point shown in \cite[Proposition 6]{PPJap} or \cite[Lemma 5.2]{QuasiK3} is that the above relation can be inverted in order to express arbitrary invariants 
\[ \big\langle \, \lambda^1,\ldots,\lambda^{m} \, \big| \, \ch_{\ellt_1}(\omega \gamma_1) \cdots \ch_{\ellt_r}(\omega \gamma_r) \big\rangle^{(S \times \p^1,S_z), \PT}_{\ch_3,\beta} \]
as $\BQ$-linear combinations of the same type of invariants (with same $(\ch_3,\beta)$) of $(X,S_{z_1,\ldots,z_{m-1}})$ and $(X,S_{\infty})$.

If Theorem~\ref{thm:PT MCF} holds for  invariants with $m-1$ and $1$ relative fibers, then 
Theorem~\ref{thm:PT MCF}
holds for invariants  with $m$ relative fibers.
The case of Theorem~\ref{thm:PT MCF} with $m=1$ fibers holds
by Theorem~\ref{thm:PT MCF equivariant} by passing to the non-equivariant limit (the non-equivariant limit is non-zero only if the dimension constraint $\nu=0$ holds).
\qed

\subsection{DT/PT correspondence}
The degeneration formula \eqref{deg formula} takes exactly the same form for DT invariants. Moreover, the non-reduced DT invariants for curve classes $(0,n)$ agree with the PT invariants:
\[
\big\langle \, \lambda^1,\ldots,\lambda^{m} \, \big| \, \ch_{\ellt_1}(\omega \gamma_1) \cdots \ch_{\ellt_r}(\omega \gamma_r) \big\rangle^{(S \times \p^1,S_z), \PT}_{\ch_3,0}
=
\big\langle \, \lambda^1,\ldots,\lambda^{m} \, \big| \, \ch_{\ellt_1}(\omega \gamma_1) \cdots \ch_{\ellt_r}(\omega \gamma_r) \big\rangle^{(S \times \p^1,S_z), \DT}_{\ch_3,0}\, .
\]
Indeed, the DT invariants vanish for $\ch_3>0$ by the same cosection argument,
and for $\ch_3=0$ the moduli spaces are the same.

By the same invertibility of the degeneration formula, it suffices to prove 
Theorem~\ref{thm:dtpt} for the cap geometry $(S \times \p^1,S_{\infty})$.
The non-equivariant case follows here from the following equivariant counterpart:

\begin{thm}[$\BC^*$-equivariant DT/PT for $(S \times \p^1,S_{\infty})$]
We have
\[ \left\langle \, \lambda \, \middle| \, \prod_{i=1}^{r} \ch_{\ellt_i}([\mathbf{0}] \gamma_i) \right\rangle^{(X,S_{\infty}), \PT}_{\ch_3,\beta}
=
\left\langle \, \lambda \, \middle| \, \prod_{i=1}^{r} \ch_{\ellt_i}([\mathbf{0}] \gamma_i) \right\rangle^{(X,S_{\infty}), \DT}_{\ch_3,\beta}\, .
\]
\end{thm}
\begin{proof}
By \cite[Sec. 4]{QuasiK3}
we have the exact analogue of Theorem~\ref{thm:PT fixed}
on the DT side: the $\BC^*$-equivariant reduced virtual classes of $\Hilb_{\ch_3,(\beta,n)}(X,S_{\infty})$ satisfies
\begin{multline*}
[ \Hilb_{\ch_3,(\beta,n)}(X, S_{\infty}) ]_T^{\red}
= \frac{1}{-t - \Psi_0} [ \Hilb_{\ch_3,(\beta,n)}(X, S_{0,\infty} )^{\sim}]^{\red} \\
+ \sum_{k|(\beta,\ch_3)} (-1)^{(k-1) \frac{\ch_3}{k}} e_{T}\left( R \hom_{\rho}( [\CI_z \to \BG], \BG)^{\vee} \otimes \Ft^{k} \right) \cap \left[ \Quot^{\flat}_{n,\frac{1}{k}(0,\beta,\ch_3)} \right]^{\mathrm{red}}\, ,
\end{multline*}
where $\CI_z \to \BG$ is the universal quotient sequence on $\Quot^{\flat}_{n,\frac{1}{k}(0,\beta,n)}$.

Following the discussion parallel to the proof of Theorem~\ref{thm:PT MCF equivariant} on the DT side, the $\BC^*$-equivariant DT invariants
can be expressed as integrals over $\Quot^{\flat}_{n,\frac{1}{k}(0,\beta,\ch_3)}$ 
with exactly the same integrand as for PT invariants.
The claimed DT/PT correspondence hence follows from
the equality of tautological integrals on the nested and the perverse nested Quot schemes established in Theorem~\ref{thm:wallcrossing}.
\end{proof}

\section{Proof of Theorem~\ref{mainn}}
\label{gwt}
\subsection{GW/PT correspondence}
Let $S$ be a symplectic surface, and let $\beta \in H_2(S,\BZ)$ be an effective (hence {nonzero}) curve class. 
Let $\lambda=(\lambda_i,\gamma_i)$ and $\mu=(\mu_j,\tilde{\gamma}_j)$ be cohomology-weighted partitions of a positive integer{\footnote{In Section \ref{dtptred}, $n=0$ was considered, but we  now require $n>0$.}} $n$,
$$ \lambda_i,\mu_j\in \mathbb{Z}_{>0}\, , \ \ \ 
n=\sum_{i=1}^{\ell(\lambda)} \lambda_i = \sum_{j=1}^{\ell(\mu)} \mu_j\, , \ \ \ \gamma_i,\tilde{\gamma}_j \in H^*(S,\BQ)\, .$$
Let $\underline{\lambda}$ and $\underline{\mu}$ denote the associated ordered partitions
$$\underline{\lambda}=(\lambda_1,\ldots, \lambda_{\ell(\lambda)})\, , \ \ \  
\underline{\mu}=(\mu_1,\ldots, \mu_{\ell(\mu)})\, .$$
Let $\zeta_1,\ldots,\zeta_s \in H^{\ast}(S\times \p^1)$ be cohomology classes.

The {\em relative Gromov-Witten invariant} with primary insertions is defined by
\[
\langle \lambda, \mu \, |\,  \tau_{0}(\zeta_1) \cdots \tau_0(\zeta_s) \rangle^{\bullet}_{g,\beta} =
\int_{[ \Mbar^{\bullet}_{g,s}(S \times \p^1,S_{0,\infty})_{{\lambda},{\mu}} ]^{\red}}  \prod_{i=1}^{\ell(\lambda)} \ev_{S_0,i}^{\ast}(\gamma_i) \cdot \prod_{j=1}^{\ell(\mu)} \ev_{S_{\infty},j}^{\ast}(\tilde{\gamma}_j) \cdot \prod_{k=1}^{s} \ev_k^{\ast}(\zeta_k)\, ,
\]
where 
$\Mbar_{g,s}^{\bullet}(S \times \p^1,S_{0,\infty})_{{\lambda},{\mu}}$
is the moduli space of $s$-marked relative stable maps  to $(S \times \p^1,S_{0,\infty})$ 
with possibly disconnected domains{\footnote{No connected component
of the domain is permitted to map to a point.}} 
with ramification profiles
$$      \underline{\lambda}\, ,\,  \underline{\mu} \ \  \text{over}  \ \ 0\, ,\infty \in\p^1\, .$$
The cohomology weights of $\lambda$ and $\mu$ do not play a role in the 
definition of $\Mbar_{g,s}^{\bullet}(S \times \p^1,S_{0,\infty})_{{\lambda},{\mu}}$
.

The associated partition function in relative Gromov-Witten theory is:
\begin{multline*} \label{defn:ZGW}
Z^{(S \times \p^1, S_{0,\infty})}_{\GW, (\beta,n)}\left( \lambda, \mu\,  \middle|\,  \tau_{0}(\zeta_1) \cdots \tau_{0}(\zeta_s) \right) 
= \\
(-1)^{-n + \ell(\lambda) + \ell(\mu)}
z^{\ell(\lambda) + \ell(\mu)} 
\sum_{g \in \BZ} (-1)^{g-1} z^{2g-2}
\left\langle \, \lambda,\mu \, \middle| \, \tau_{0}(\zeta_1) \cdots \tau_{0}(\zeta_s) \right\rangle^{\bullet}_{g, (\beta,n)}\, .
\end{multline*}
The corresponding partition function in PT theory is:
\[ 
Z^{(S \times \p^1,S_{0,\infty})}_{\PT, (\beta,n)}\left( \lambda,\mu \, \middle| \, \ch_{2}(\zeta_1) \cdots \ch_{2}(\zeta_s) \right)
=
\sum_{\ch_3 \in \BZ} (-p)^{\ch_3}
\big\langle \, \lambda,\mu \, \big| \, \ch_{2}(\zeta_1) \cdots \ch_{2}(\zeta_s) \big\rangle^{(S \times \p^1,S_{0,\infty}), \PT}_{\ch_3,\beta}.
\]
Our partition functions here take the nonstandard form used in \cite{Marked} (differing in variable conventions and prefactors
from \cite{MNOP1,MNOP2,PT1}).

The GW/PT correspondence is formulated for families of relative threefolds \cite{MNOP1,MNOP2,PT1,PP,P-Beijing}.
The GW/PT correspondence for primary insertions for
twistor families of symplectic surfaces $S$ yields the following statement.

\begin{conj} \label{conj:GWPT}
Let $S$ be a symplectic surface. 
The series $$Z^{(S \times \p^1,S_{0,\infty})}_{\PT, (\beta,n)}\left( \lambda, \mu \, \middle|\, \ch_{2}(\zeta_1) \cdots \ch_{2}(\zeta_s) \right)$$
is the Laurent expansion of a rational function in $p$.
After the variable change $p=e^{z}$, we have
\begin{equation*}
Z^{(S \times \p^1, S_{0,\infty})}_{\PT, (\beta,n)}\left( \lambda, \mu \, \middle|\, \ch_{2}(\zeta_1) \cdots \ch_{2}(\zeta_s) \right)
=
Z^{(S \times \p^1, S_{0,\infty})}_{\GW, (\beta,n)}\left( \lambda, \mu \, \middle|\, \tau_{0}(\zeta_1) \cdots \tau_{0}(\zeta_s) \right) 
\end{equation*}
\end{conj}

Since $S\times \mathbb{P}^1/ S_0 \cup S_\infty$
is log nef and we are only considering primary insertions, Conjecture \ref{conj:GWPT} is not far from the results of Pardon \cite{Pardon}, see \cite{P-Beijing} for a discussion. At the moment, however,
the statement remains a conjecture.

\subsection{MCF: from PT to GW rubber} \label{pttorub}
We start with the PT multiple cover formula of Theorem~\ref{thm:PT MCF}.
Let $D \in H^2(S)$ be a divisor class satisfying $D \cdot \beta \neq 0$.
As a direct consequence of Theorem~\ref{thm:PT MCF}, we obtain the following identity of partition functions:
\[
Z^{(S \times \p^1,S_{z})}_{\PT, (\beta,n)}\left( \lambda,\mu \, \middle|\,  \ch_2(\omega D) \right)
=
\sum_{k|\beta}
Z^{(S_k \times \p^1,(S_k)_{z})}_{\PT, (\varphi_k(\beta/k),n)}\left( \varphi_k(\lambda), \varphi_k(\mu) \, \middle|\,  \ch_{2}(\omega \varphi_k(D)) \middle) \right|_{p \mapsto p^k}\, .
\]
If Conjecture~\ref{conj:GWPT} holds, then the above equation implies
\begin{equation} \label{fds03rfs}
\left\langle \, \lambda,\mu \, \middle| \, \tau_0(\omega D) \right\rangle^{\GW, (S \times \p^1, S_{0,\infty}), \bullet}_{g, (\beta,n)}
=
\sum_{k | \beta} k^{2g-2+\ell(\lambda) + \ell(\mu)}
\left\langle \, \varphi_k(\lambda),\varphi_k(\mu) \, \middle| \, \tau_0(\omega \varphi_k(D)) \right\rangle^{\GW, (S \times \p^1, S_{0,\infty}), \bullet}_{g, (\varphi_k(\beta/k),n)} \, .
\end{equation}

Define the {\em Gromov-Witten rubber invariants} for the rubber target $(S\times \p^1,S_{0,\infty})^\sim$ by
\[
\langle \lambda, \mu \rangle^{\bullet,\sim}_{g,\beta} =
\int_{[ \Mbar_{g}(S \times \p^1,S_{0,\infty})^{\bullet,\sim}_{{\lambda},{\mu}} ]^{\red}} \prod_{i=1}^{\ell(\lambda)} \ev_{S_0,i}^{\ast}(\gamma_i) \cdot \prod_{j=1}^{\ell(\mu)} \ev_{S_{\infty},j}^{\ast}(\tilde{\gamma}_j)\, ,
\]
where the integral is over the moduli space of stable maps to rubber 
with possibly disconnected domains.\footnote{
As before,
no connected component of the domain is permitted to map to a point.}
By derigidification (in the rubber calculus) and the divisor equation, 
the rubber and relative invariants are related by
\[ 
\langle \lambda, \mu \, |\,  \tau_0(\omega D) \rangle^{(S \times \p^1,S_{0,\infty}),\bullet}_{g,\beta}
=
\langle \lambda, \mu \, |\,  \tau_0(D) \rangle^{\bullet,\sim}_{g,\beta}
=
(D \cdot \beta)\,  \langle \lambda, \mu \rangle^{\bullet,\sim}_{g,\beta}\, .
\]
Equation \eqref{fds03rfs} therefore implies the following formula for rubber invariants:
\begin{equation} \label{abcmcf}
\left\langle \, \lambda,\mu \right\rangle^{\bullet, \sim}_{g,\beta}
=
\sum_{k | \beta} k^{2g-3+\ell(\lambda) + \ell(\mu)}
\left\langle \, \varphi_k(\lambda),\varphi_k(\mu) \right\rangle^{\bullet,\sim}_{g, \varphi_k(\beta/k)}\, .
\end{equation}

The left side, or equivalently the right side, of \eqref{abcmcf} can only be non-zero if 
the dimension constraint
\[ 
2n = |\lambda| - \ell(\lambda) + |\mu| - \ell(\mu) + \sum_{i=1}^{\ell(\lambda)} \deg_{\mathbb{C}}(\gamma_i) + \sum_{j=1}^{\ell(\mu)} \deg_{\mathbb{C}}(\tilde{\gamma}_j) \]
is satisfied, which is equivalent to
 \begin{equation} \ell(\lambda) + \ell(\mu) =\sum_{i=1}^{\ell(\lambda)} \deg_{\mathbb{C}}(\gamma_i) + \sum_{j=1}^{\ell(\mu)} \deg_{\mathbb{C}}(\tilde{\gamma}_j)\, .
\label{dimcon}
\end{equation}
After imposing the dimension constraint \eqref{dimcon}, equality \eqref{abcmcf} becomes:
\begin{equation}\label{rubmcf}
\left\langle \, \lambda,\mu \right\rangle^{\bullet,\sim}_{g,\beta}
=
\sum_{k | \beta} k^{2g-3+\sum_i \deg_{\mathbb{C}}(\gamma_i) + \sum_j \deg_{\mathbb{C}}(\tilde{\gamma}_j)}
\left\langle \, \varphi_k(\lambda),\varphi_k(\mu) \right\rangle^{\bullet,\sim}_{g, \varphi_k(\beta/k)}\, .
\end{equation}
If equality \eqref{rubmcf} is satisfied, we say that {\em the MCF holds for rubber invariants}.
We have proven:

\vspace{6pt}
\noindent{\bf{Step A.}}
{\em Conjecture~\ref{conj:GWPT} 
implies the MCF for rubber invariants.}

\subsection{Connected/disconnect rubber invariants}

Define the {\em connected Gromov-Witten rubber invariants} for the rubber target $(S\times \p^1,S_{0,\infty})^\sim$ by
\[
\langle \lambda, \mu \rangle^{\sim}_{g,\beta} =
\int_{[ \Mbar_{g}(S \times \p^1,S_{0,\infty})^{\sim}_{{\lambda},{\mu}} ]^{\red}} \prod_{i=1}^{\ell(\lambda)} \ev_{S_0,i}^{\ast}(\gamma_i) \cdot \prod_{j=1}^{\ell(\mu)} \ev_{S_{\infty},j}^{\ast}(\tilde{\gamma}_j)\, ,
\]
where the integral is over the moduli space of stable maps to rubber 
with connected domains.

We say that {\em the MCF holds for connected rubber invariants} if
\begin{equation}\label{connrub}
\left\langle \, \lambda,\mu \right\rangle^{\sim}_{g,\beta}
=
\sum_{k | \beta} k^{2g-3+\sum_i \deg_{\mathbb{C}}(\gamma_i) + \sum_j \deg_{\mathbb{C}}(\tilde{\gamma}_j)}
\left\langle \, \varphi_k(\lambda),\varphi_k(\mu) \right\rangle^{\sim}_{g, \varphi_k(\beta/k)}\, .
\end{equation}

There is a simple relationship between
possibly disconnected 
and connected rubber
invariants based on two principles.
First, because of the reduction, {\em only
one connected component can carry a nonzero curve class on $S$}. Second,
by a derigidification
argument, all connected components which have
curve class 0 on $S$ must be genus 0 tubes
fully ramified over $0,\infty \in \p^1$ by the following result.

\begin{lemma} The connected invariant $\left\langle \lambda,\mu \right\rangle^{(S \times \p^1,S_{0,\infty})}_{g,0}$ vanishes 
unless $g=0$, $\lambda=(n,\gamma)$, $\mu=(n,\tilde{\gamma})$. If these conditions 
hold, then
\begin{equation}
\label{w3e}
\left\langle (n,\gamma), (n,\tilde{\gamma}) \right\rangle^{(S \times \p^1,S_{0,\infty})}_{0,0}= \frac{1}{n} \int_S \gamma \cdot \tilde{\gamma}\, . 
\end{equation}
\end{lemma}

\begin{proof} The virtual dimension of $\Mbar_{g}(S \times \p^1,S_{0,\infty})_{\lambda,\mu}$ is
$\ell(\lambda)+\ell(\mu)$. Since the domains are connected and the curve class on $S$ is 0, at most
two dimensions can be imposed by the cohomology weights. Therefore, unless $\ell(\lambda)=\ell(\mu)=1$ and
the corresponding cohomology weights satisfy
\begin{equation}\label{kee4}
\deg_{\BC}(\gamma)+ \deg_{\BC}(\tilde{\gamma})=2\, , 
\end{equation}
the connected invariant $\left\langle \lambda,\mu \right\rangle_{g,0}$ vanishes.

We assume now  $\lambda=(n,\gamma)$, $\mu=(n,\tilde{\gamma})$, and that the  condition \eqref{kee4} is satisfied.
Since we have connected domains and curve class 0 on $S$, there is an isomorphism of moduli spaces 
$$\Mbar_{g}(S \times \p^1,S_{0,\infty})_{(n),(n)} = S \times \Mbar_{g}(\p^1,{0,\infty})_{(n),(n)}$$
under which the virtual classes are related by
$$[\Mbar_{g}(S \times \p^1,S_{0,\infty})_{(n),(n)}]^{\vir}=
e(\BE^\vee \otimes T_S) \cap \left([S] \times [\Mbar_{g}(\p^1,S_{0,\infty})_{(n),(n)}]^{\vir}\right)\, ,$$
where $\BE \to \Mbar_{g}(\p^1,S_{0,\infty})_{(n),(n)}$ is the Hodge bundle with fiber $H^0(C,\omega_C)$ over a moduli point 
$$[f:C \to \p^1[\ell]]\in \Mbar_{g}(\p^1,S_{0,\infty})_{(n),(n)} \,.$$
If $g=0$, then $\BE^\vee=0$, and we obtain the claimed evaluation \eqref{w3e}.
If $g=1$, then $e(\BE^\vee \otimes T_S) = c_2(T_S)$ which forces the vanishing of
$\left\langle \lambda,\mu \right\rangle_{1,0}$ because of the additional constraints
placed by the cohomology weights.
If $g\geq 2$, the vanishing 
$$e(\BE^\vee \otimes T_S) = 0$$
is a consequence of \eqref{Dfcsdfre3r}.
\end{proof}

Straightforward accounting (a disconnected tube raises the genus of the remainder of the domain by 1 and
removes cohomology weights with degree sum 2) then yields:

\vspace{6pt}
\noindent{\bf{Step B.}}
{\em The MCF for rubber invariants implies
the MCF for connected rubber invariants.}
\vspace{6pt}

\subsection{MCF: from  rubber invariants to  vertex terms}
For a cohomology-weighted partition $$\lambda = ((\lambda_1,\gamma_1), \ldots, (\lambda_r,\gamma_r))$$ 
of a positive integer $n$, we define the vertex term
\[
I_{g,\beta}(\lambda) = \int_{[\Mbar_{g,r}(S,\beta) ]^{\red}} 
c(\BE^{\vee})
\prod_{i=1}^{r} \frac{\ev_i^{\ast}(\gamma_i)}{1-\lambda_i \psi_i}\, ,
\]
where $\BE \to \Mbar_{g,r}(S,\beta)$ is the Hodge bundle with fiber $H^0(C,\omega_C)$ over a moduli point 
$$[f:C \to S, p_1,\ldots,p_r]\in \Mbar_{g,r}(S,\beta) \,.$$
The vertex term is an integral over the moduli space of stable maps $\Mbar_{g,r}(S,\beta)$ with {\em connected domains}.

We say that {\em the MCF holds for vertex terms} if
\begin{equation} \label{frfr3}
I_{g,\beta}(\lambda)
=
\sum_{k|\beta} k^{2g-3+\sum_{i} \deg_{\BC}(\gamma_i)}
I_{g,\varphi_k(\beta/k)}(\varphi_k(\lambda))\, .
\end{equation}

\vspace{4pt}
\noindent{\bf{Step C.}}
{\em If the MCF holds for connected rubber invariants, then the MCF holds for vertex terms.}
\vspace{2pt}
\begin{proof}
Assume that the MCF holds for rubber invariants.
Let $\Mbar_{g}(S \times \p^1,S_{\infty})_{{\lambda}}$ be the moduli space of relative stable maps to the cap $(S \times \p^1,S_{\infty})$ with ramification profile $\underline{\lambda}$ over ${\infty}\in \p^1$.
Let $\BC^{\ast}$ act on $\p^1$ with tangent weight $t$ at $0 \in \p^1$.
The reduced virtual class of
$\Mbar_{g}(S \times \p^1,S_{\infty})_{{\lambda}}$
admits a natural equivariant lift. 

We will study the vanishing 
$\BC^*$-equivariant integral
\begin{equation} \label{vanishing integral}
t^{n+r+1-\sum_i\deg_{\BC}(\gamma_i)} \int_{[ \Mbar_{g}(S \times \p^1,S_{\infty})_{{\lambda}} ]^{\red}} \prod_{i=1}^{\ell(\lambda)} \ev_{S_{\infty},i}^{\ast}(\gamma_i) = 0\, .
\end{equation}
The vanishing follows for dimension reasons. Indeed,
$$\dim_{\mathbb{C}}[\Mbar_{g}(S \times \p^1,S_{\infty})_{{\lambda}}]^{\red} =
2n+1 - \sum_i (\lambda_i-1) = n+1+r\, ,$$ and the degree of the integrand is $$\sum_i \deg_{\BC}(\gamma_i) \leq 2 \ell(\lambda) \leq n+r < n+r+1\, .$$
Since the moduli space is proper and the degree of the integrand is less than the reduced virtual dimension,
the $\BC^*$-equivariant integral \eqref{vanishing integral} vanishes.
The exponent of the $t$-factor in
\eqref{vanishing integral} 
is chosen so that the integral is of cohomological weight $0$.
In particular, after applying the localization formula below, we can set $t=1$ without losing information.

We will apply the virtual localization formula for relative Gromov-Witten invariants \cite{GP,GV}  to the  vanishing integral \eqref{vanishing integral}.
The formula expresses the integral on the left side of \eqref{vanishing integral} in terms of integrals of the following three types:
\begin{enumerate}
\item[(i)] $I_{g,\beta}(\lambda)$,
\item[(ii)] $I_{g',\beta}((\lambda'_1,\gamma'_1), \ldots, (\lambda'_{r'},\gamma'_{r'}))$ where either
\begin{enumerate}
\item[(a)] $g'<g$, or
\item[(b)] $g'=g$ and $\sum_i \lambda'_i < \sum_i \lambda_i=n$, or
\item[(c)] $g'=g$ and $\sum_i \lambda'_i = n$ and $r'<r$,
\end{enumerate}
\item[(iii)] connected rubber invariants $\langle \mu,\nu \rangle^{\sim}_{g',\beta}$ for some $g',\mu,\nu$.
\end{enumerate}

The plan of the proof is as follows.
Each summand of the virtual localization formula for \eqref{vanishing integral}
will be proven
to satisfy the MCF with the same $k$ exponent.
By assumption, the MCF holds for the rubber invariants (iii).
By induction on $(g,\lambda)$ with ordering defined by (ii), we will show that
the MCF holds for all terms except for (i). The MCF will then also hold for (i) since
the integral \eqref{vanishing integral} vanishes.

The virtual localization formula \cite[Theorem 3.6]{GV} for \eqref{vanishing integral}
has a contribution from the $\BC^*$-fixed main component, where the generic element is a map to $X=S \times \p^1$ {\em without expansion}. 
By \cite[Section 3.7]{GV}, the main component
contribution is 
\begin{align} \label{withoutx}
\int_{[\Mbar_{g}(X,S_{\infty})^{\mathrm{main}}_{{\lambda}}]^{\red}} \frac{\prod_{i=1}^{r} \ev_i^{\ast}(\gamma_i)}{e(N^{\vir})}\, \Big|_{t=1}
& =
\prod_{i=1}^{r} \frac{\lambda_i^{\lambda_i}}{(\lambda_i-1)!} 
\int_{[\Mbar_{g,r}(S,\beta)]^{\red}} c(\BE^{\vee}) \prod_{i=1}^{r} \frac{\ev_i^{\ast}(\gamma_i)}{1-\lambda_i \psi_i} \\
& = \prod_{i=1}^{r} \frac{\lambda_i^{\lambda_i}}{(\lambda_i-1)!} I_{g,\beta}(\lambda)\, . \nonumber
\end{align}
Up to a multiplicative factor, the contribution is the integral (i).

The contribution of all of the $\BC^*$-fixed components {\em with expansions} is
\begin{equation} \label{mixed component} \sum^{\circ}_{\substack{\mu,g_0,g_1\\ g=g_0+g_1 + \ell(\mu)-1 \\ (\beta_0,\beta_1) \in \{ (\beta,0), (0,\beta)\} }}
\frac{\prod_i \mu_i}{|\Aut(\mu)|}\,
\Big\langle \mu^{\vee} \Big\rangle^{(X,S_{\infty}),\bullet,\mathrm{main}}_{g_0,\beta_0}\,
\left\langle \mu,\lambda\, \middle| \frac{1}{-1-\Psi_{\infty}} \right\rangle_{g_1,\beta_1}^{\bullet, \sim}\, ,
\end{equation}
with the following conventions:
\begin{itemize}
\item $\mu$ runs over all weighted partitions $\mu = ((\mu_1,\delta_{a_1}), \ldots, (\mu_{\ell(\mu)},\delta_{a_{\ell(\mu)}}))$, where $(\delta_a)_a$ is a fixed ordered basis of the cohomology $H^{\ast}(S)$, and $\mu^{\vee} = (\mu_i,\delta_{a_i}^{\vee})$ is the dual partition,
\item $\Psi_{\infty}$ is the relative psi class on the rubber moduli space at the boundary divisor $S_{\infty} \subset S \times \p^1$, inserted here as an integrand in the rubber invariant,
\item  the contribution from the part of the stable map to $X$  (not rubber) is
\[ \Big\langle \mu^{\vee} \Big\rangle^{(X,S_{\infty}),\bullet,\mathrm{main}}_{g_0,\beta_0} = 
\int_{[ \Mbar_{g_0}(X,S_{\infty})^{\bullet,\mathrm{main}}_{{\mu}}]^{\red}}
\frac{1}{e(N^{\vir})}
\prod_{i=1}^{\ell(\mu)} \ev_i^{\ast}(\delta_{a_i}^{\vee})\, \Big|_{t=1} \, ,
\]
\item $\Aut(\mu)$ is the automorphism group of the partition $\mu$,
\item $\stackrel{\circ}{\sum}$ stands for summing only over those terms for which the glued stable map is connected.
\end{itemize}

Consider the terms in the sum \eqref{mixed component} with $\beta_0=\beta$. 
Because of the reduction, terms of \eqref{mixed component} where  $\beta_0$ splits
into at least two nonzero summands of $\beta$ vanish, so 
\begin{equation} \big\langle \mu^{\vee} \big\rangle^{(X,S_{\infty}),\bullet,\mathrm{main}}_{g_0,\beta_0} 
\label{tttt}
\end{equation}
is a linear combination of integrals
$I_{g',\beta}(\mu')$.  
We will show that the $(g',\mu')$ which appear are lower than $(g,\lambda)$ in the ordering defined by (ii).


Consider a component $N$ {\em with expansion} in the $\BC^*$-fixed locus, 
$$N\subset \Mbar_{g}(X,S_{\infty})^{\BC^{\ast}}\, ,$$ 
for which the contribution $I_{g',\beta}(\mu')$ appears in the virtual localization formula for \eqref{vanishing integral}.
Let $$[f : C \to X[m]] \in N$$
be the moduli point of a stable relative map, where $X[m]$ is an expansion of $X$.
Let $\widetilde{C}$ be the partial normalization of $C$ along the preimage of the first singular divisor 
$$S \subset X \subset X[m]\, .$$ 
In other words,
we resolve only the nodes mapping to the singular divisor where the main component meets  the bubble.
Let $\Gamma$ be the decorated dual graph formed by the connected components of $\widetilde{C}$: 
\begin{enumerate}
\item[$\bullet$] The vertices of $\Gamma$ are in correspondence with the connected components of $\widetilde{C}$.
\item[$\bullet$] The edges correspond to the nodes which were resolved (connecting two adjacent connected components of $\widetilde{C}$).
\item[$\bullet$] The vertices of $\Gamma$ are decorated with the genus of the corresponding connected component, the curve class in $H^2(S,\mathbb{Z})$ of the restriction of $f$, and the legs of the relative markings of $C$. 
\end{enumerate}
We say that a vertex is a {\em rubber vertex}, if the corresponding component maps to the rubber bubble, otherwise we say that the vertex is a {\em main vertex}.
Each rubber vertex is connected to a main vertex{\footnote{$C$ is connected and $n>0$.}}, but no rubber or main vertices are connected to other vertices of the same type.
Each rubber vertex carries at least one leg.
Since $C$ is connected, $\Gamma$ is a connected graph.
There is a single main vertex of genus $g'$ and curve class $\beta$,
but there may be other main vertices of curve class $0$.

We now show that $(g',\mu')$ is lower in ordering (ii) than $(g,\lambda)$.
First, we have $g' \leq g$ since $C$ is connected.
If $g'=g$, then $\Gamma$ is a tree, and all connected components of $\widetilde{C}$, except for the main vertex carrying curve class $\beta$, have genus $0$. If moreover,
$|\mu'| = |\lambda|$,
then there is a unique main vertex, the rest are rubber vertices.
Since the graph is connected, each rubber vertex must be connected to the unique main vertex.
Hence, $\ell(\mu')$ is the number of rubber vertices, so $\ell(\mu') \leq \ell(\lambda)$. If moreover, 
$\ell(\mu') = \ell(\lambda)$,
there must be $\ell(\lambda)=r$ rubber vertices, so each must have precisely one leg.
But then each rubber vertex corresponds to a genus $0$ relative stable map to $(S \times \p^1,S_{0,\infty})$, of curve class $0$ on $S$, with total ramification over $0$ and $\infty$. In other words, every rubber vertex corresponds to a genus $0$ tube. By stability of the rubber, not every bubble vertex can correspond to a genus $0$ tube,
so such a component can not exist.
We conclude that either 
$$g'<g \ \ \text{or}\ \  |\mu'| < |\lambda| \ \ \text{or} \ \ \ell(\mu') < \ell(\lambda)\, ,$$
which is precisely the ordering defined by (ii).

Next,  we consider the terms in the sum \eqref{mixed component} with $\beta_1=\beta$ and rubber invariant
\begin{equation} \left\langle \mu,\lambda\, \middle| \frac{1}{1-\Psi_{\infty}} \right\rangle_{g_1,\beta}^{\bullet, \sim} \label{rubber with psi} \, .\end{equation}
By the rubber calculus \cite{MPrubber}, the relative psi class $\Psi_{\infty}$ can be 
removed by adding a bubble as follows.
Let $D\in H^2(S)$ be a divisor
class satisfying $D\cdot \beta\neq 0$. Then, we have
\begin{eqnarray*}
\left\langle \mu, \lambda\, \middle|\,  \Psi_{\infty} \right\rangle_{g_1,\beta}^{\bullet, \sim}
& = &
\frac{1}{D\cdot \beta} \left\langle \mu, \lambda\, \middle|\, \tau_0(D)\Psi_{\infty} \right\rangle_{g_1,\beta}^{\bullet, \sim}
\\
&= &
\frac{1}{D\cdot \beta} 
\sum_{h_1,\beta_1,h_2,\beta_2,\nu}
\frac{\prod_i \nu_i}{|\Aut(\nu)|} \left\langle\, \mu \, 
\middle| \, \tau_0(\omega D)\, \middle|\,  \nu \right\rangle_{h_1,\beta_1}^{\bullet, \sim}
\left\langle \nu^{\vee}, \lambda \right\rangle_{h_2,\beta_2}^{\bullet, \sim} \\
& = &
\sum_{h_1,h_2,\nu}
\frac{\prod_i \nu_i}{|\Aut(\nu)|} \left\langle \mu, \nu \right\rangle_{h_1,\beta}^{\bullet, \sim}
\left\langle \nu^{\vee}, \lambda \right\rangle_{h_2,0}^{\bullet, \sim}\, 
+
\frac{1}{D\cdot \beta}
\left\langle D\mu, \lambda \right\rangle_{g_1,\beta}^{\bullet, \sim}
\, ,
\end{eqnarray*}
where $D\mu$ denotes the boundary conditions obtained by summing the action of $D$ on the
cohomology-weights. For example, if $\mu=\{(\mu_1,\tilde{\gamma}_1),(\mu_2,\tilde{\gamma}_2)\}$, then
$$D\mu=\{(\mu_1,D\tilde{\gamma}_1),(\mu_2,\tilde{\gamma}_2)\} +
\{(\mu_1,\tilde{\gamma}_1),(\mu_2,D\tilde{\gamma}_2)\}\, .
$$ 
Parallel expansions hold for $\left\langle \mu, \lambda\, \middle|\,  \Psi^r_{\infty} \right\rangle_{h,\beta}^{\bullet, \sim}$.
The rubber invariant \eqref{rubber with psi} can therefore be expressed as a linear combination of standard rubber invariants
$\langle \lambda', \mu' \rangle^{\bullet,\sim}_{g',\beta}$
times invariants with curve class $0$ over $S$.
These are again linear combinations of the connected rubber invariants.
Thus, we arrive at terms of the form (iii).

We
will show that the MCF holds with $k$ exponent
\begin{equation}\label{kexp}
2g-3+\sum_{i=1}^{r} \deg_\BC(\gamma_i)
\end{equation}
for all of the contributions to \eqref{vanishing integral}.
The induction hypothesis will be used for all the summands of \eqref{mixed component}.
The vanishing of \eqref{vanishing integral} will then imply that the MCF holds for contribution \eqref{withoutx}
with $k$ exponent
\eqref{kexp}.

Consider a summand in \eqref{mixed component} with expansion:
\[ 
\Big\langle \mu^{\vee} \Big\rangle^{(X,S_{\infty}),\bullet,\mathrm{main}}_{g_0,\beta_0} \, 
\left\langle \mu, \lambda \, \middle| \frac{1}{1-\Psi_{\infty}} \right\rangle_{g_1,\beta_1}^{\bullet, \sim}\, \, ,
\]
where the glued map is required to be connected.
The underlying stable maps  in the two factors are obtained from the stable maps of $\Mbar_{g}(X,S_{\infty})^{\BC^*}$
by cutting $\ell(\mu)$ edges, thereby reducing the genus by $\ell(\mu)$, and inserting $\ell(\mu)$ diagonal insertions $\Delta_S$ along the connecting boundary.
The resulting stable map may be viewed  as part of a disconnected invariant
of $S$ 
with total genus $g-\ell(\mu)$ together with an insertion of cohomological degree 
$$\delta=\sum_{i=1}^r \deg_{\BC}(\gamma_i) + 2 \ell(\mu)\, .$$ After writing the results as a product of connected invariants 
we see that the connected invariant carrying the class $\beta$ has multiple cover exponent $$2(g-\ell(\mu)) - 3 + \delta = 2g-3+\sum_{i=1}^r \deg_\BC(\gamma_i)$$ using either the MCF for (ii) inductively or for (iii). The
terms of the form
$$
\frac{1}{D\cdot \beta}
\left\langle D\mu, \lambda \right\rangle_{h,\beta}^{\bullet, \sim}$$
which appear in the rubber calculus behave correctly with respect to the exponent.
\end{proof}

\subsection{MCF: from  vertex terms to  descendent invariants} \label{verttodesc}
The final step in our argument is essentially formal.

\vspace{4pt}
\noindent{\bf{Step D.}}
{\em If the MCF holds for  vertex terms,
then the MCF holds for  descendent invariants
$$\big\langle \tau_{\ellt_1}(\gamma_1) \cdots \tau_{\ellt_r}(\gamma_r) \big\rangle^{S}_{g,\beta}\,.$$}


\begin{proof}
Let $x_1,\ldots, x_r$ be variables.
The integral
\begin{equation}\label{polyyy}
\int_{[\Mbar_{g,r}(S,\beta) ]^{\red}} 
c(\BE^{\vee})
\prod_{i=1}^{r} \frac{\ev_i^{\ast}(\gamma_i)}{1-x_i \psi_i}
-
\sum_{k|\beta} k^{2g-3+\sum_i\deg_{\BC}(\gamma_i)}
\int_{[\Mbar_{g,r}(S,\varphi_k(\beta/k)) ]^{\red}} 
c(\BE^{\vee})
\prod_{i=1}^{r} \frac{\ev_i^{\ast}(\varphi_k(\gamma_i))}{1-x_i \psi_i}
\end{equation}
is a polynomial in the variables 
$x_1,\ldots,x_r$.
If MCF holds for vertex terms \eqref{frfr3}, 
then the polynomial \eqref{polyyy} vanishes on the Zariski dense subset $\BN^r \subset \BC^r$, so the polynomial is identically zero.
By taking the $x_1^{\ellt_1} \cdots x_r^{\ellt_r}$ coefficient of the polynomial \eqref{polyyy}
for $$\sum_{i=1}^r (\ellt_i + \deg_{\BC}(\gamma_i)) = g+r\,,$$ 
we find that the MCF 
holds for
the descendent invariants
$\big\langle \tau_{\ellt_1}(\gamma_1) \cdots \tau_{\ellt_r}(\gamma_r) \big\rangle^{S}_{g,\beta}$.
\end{proof}

\subsection{Theorem~\ref{mainn} and the converse}  \label{thmcon}
The proof of Theorem \ref{mainn} is now complete: Theorem \ref{mainn} is exactly obtained by  concatenating the chain of implications of {\bf{Steps A}}, {\bf{B}}, {\bf{C}} and {\bf{D}} of Sections \ref{pttorub}--\ref{verttodesc}. \qed

\subsection{Converse}
The statement of the converse is:
{\em the MCF for the Gromov-Witten descendent theory of $S$ implies the GW/PT  correspondence
for $(S \times \p^1,S_0\cup S_{\infty})$ with primary insertions for all curve classes on $S$}.
The proof proceeds by reduction to the case of primitive curve classes on $S$. For the GW side, we use the product formula together with the assumed MCF. For the PT side, we use Theorem \ref{thm:PT MCF}. A straightforward verification shows the multiple cover structures on the GW and PT sides are compatible with the GW/PT correspondence.
The case where the curve class on $S$ is primitive has been previously established for $K3$ surfaces in \cite{QuasiK3}, and for abelian surfaces the proof is parallel to \cite{QuasiK3}. \qed

\section{Proof of Theorem \ref{hodgemc}} \label{pr-hodgemc}
\subsection{Overview}
We present here the proof of Theorem~ \ref{hodgemc}. We will use the GW/PT correspondence for the $\BC^*$-equivariant theories of $S \times \BC$ with primary point insertions (which is a consequence of the results of Pardon \cite{Pardon}).

\subsection{GW/PT correspondence}
Let $S$ be a symplectic surface, and let
$\beta \in H_2(S,\BZ)$ be an effective curve class.
Let $\iota : S \to S \times \BC$ be the inclusion of the zero section.
Let $\pt \in H^4(S)$ be the class of a point, which we view as an equivariant cohomology class on $S\times \BC$ via pullback by projection.

Consider the reduced GW and PT invariants
\begin{eqnarray} \label{pu77}
\langle \tau_0(\pt)^m \rangle^{S\times \BC,\GW}_{g,\beta} &=&
\int_{[\Mbar_{g,m}(S \times \BC,\iota_{\ast} \beta)]^{\red}} \ev_1^{\ast}(\pt) \cdots \ev_m^{\ast}(\pt)  \,\in\, \BQ(t)\, , \\ \nonumber
\langle \ch_2(\pt)^m \rangle^{S\times \BC,\PT}_{\ch_3,\beta} &=&
\int_{[P_{\ch_3,\iota_{\ast} \beta}(S \times \BC)]^{\red}} \ch_2(\pt)^m  \,\in\, \BQ(t)\, .
\end{eqnarray}
Since the non-reduced Gromov-Witten invariants of $S\times \BC$ with primary point insertions always vanish, 
the Gromov-Witten invariants \eqref{pu77} of $S\times \BC$ with possibly disconnected and connected domains agree.

Define the partition functions:
\begin{align*}
Z^{S\times \BC}_{\GW, \beta}\left( \tau_0(\pt)^m \right)
& =
\sum_{g} (-1)^{g-1} z^{2g-2} \langle \tau_0(\pt)^m \rangle^{S\times \BC,\GW}_{g,\beta} \, ,\\
Z^{S\times \BC}_{\PT, \beta}\left( \ch_{2}(\pt)^m \right)
& =
\sum_{\ch_3 \in \BZ} (-p)^{\ch_3} \langle \ch_2(\pt)^m \rangle^{S\times \BC,\PT}_{\ch_3,\beta}\, .
\end{align*}

Pardon's proof of the GW/PT correspondence for Calabi-Yau threefolds \cite{Pardon} also holds in the family case, see \cite[Theorem 4]{P-Beijing}.
The correspondence for primary insertions for twistor families of symplectic surfaces $S$ times $\BC$ yields
the following statement.

\begin{thm} [Pardon]\label{thm:GWPT for SxC}
The series $Z^{S\times \BC}_{\PT, \beta}\left( \ch_{2}(\pt)^m \right)$
is the Laurent expansion of a rational function in $p$.
After the variable change $p=e^{z}$, we have
\begin{equation*}
Z^{S\times \BC}_{\PT, \beta}\left( \ch_{2}(\pt)^m \right)
=
Z^{S\times \BC}_{\GW, \beta}\left( \tau_{0}(\pt)^m \right)\, .
\end{equation*}
\end{thm}

\subsection{Proof of Theorem~\ref{hodgemc}}
Our starting point is the multiple cover formula for PT invariants of $S \times \BC$
established in Theorem~\ref{thm: PT SxC MCF}:
\[
\langle \ch_2(\pt)^m \rangle^{S\times \BC,\PT}_{\ch_3,\beta} =
\sum_{k|(\ch_3,\beta)} k^{2m-1} (-1)^{(k-1) \frac{\ch_3}{k}} 
\langle \ch_2(\pt)^m \rangle^{S_k\times \BC,\PT}_{\ch_3/k,\varphi_k(\beta/k)} \]
After rewriting in generating series form,
we find:
\[
Z^{S\times \BC}_{\PT, \beta}\left( \ch_{2}(\pt)^m \right)
=
\sum_{k|\beta} k^{2m-1} Z^{S_k\times \BC}_{\PT, \varphi_k(\beta/k)}\left( \ch_{2}(\pt)^m \right)(p^k)\, .
\]
The GW/PT correspondence of Theorem~\ref{thm:GWPT for SxC} then implies:
\begin{equation} \label{dfs-d0of-3w}
\langle \tau_0(\pt)^m \rangle^{S\times \BC,\GW}_{g,\beta}
=
\sum_{k|\beta} k^{2g-3+2m}
\langle \tau_0(\pt)^m \rangle^{S_k \times \BC,\GW}_{g,\varphi_k(\beta/k)}\, .
\end{equation}
Using the virtual localization formula, we have
\[
\langle \tau_0(\pt)^m \rangle^{S\times \BC,\GW}_{g,\beta}
=
\frac{1}{t} \left\langle \tau_0(\pt)^m e\left( \BE^{\vee} \otimes \Ft \right) \right\rangle^{S}_{g,\beta}
=
\frac{1}{t} \left\langle \tau_0(\pt)^m e\left( t^g - t^{g-1} \lambda_1 + \ldots + (-1)^g \lambda_g \right) \right\rangle^{S}_{g,\beta}\, .
\]
The claim is obtained by
extracting the $t^{m-1}$ coefficient of \eqref{dfs-d0of-3w}. \qed

\section{Future directions and open problems}

\subsection{Explicit evaluations}
Let $S$ be a symplectic surface.
The methods of \cite{BOPY,MPT} give effective algorithms to compute all reduced
Gromov-Witten  descendent invariants
of $S$ in primitive curve
classes. 
Descendent invariants in imprimitive curve classes can
then be calculated in terms of
descendent invariants 
in primitive classes  by the multiple cover formula.
However, the current algorithms for primitive curve classes  are  recursive and difficult in practice. 
The key step involves the 
calculus of tautological classes on the moduli spaces of curves \cite{Pan-Calculus}:  sufficiently high degree monomials of $\psi$-classes must be expressed in terms of boundary classes.
To find more efficient methods or, ideally,
explicit closed formulas for descendent series for primitive curve classes is an interesting direction.

For stationary descendent invariants of $K3$ surfaces, explicit formulas were conjectured for  primitive curve classes in \cite{O_K3},
but these conjectures do not treat descendents of $1\in H^*(S,\BQ)$.

An alternative way to encode all the invariants of symplectic surfaces $S$ is through rubber integrals for the geometry
$(S\times \p^1, S_{0,\infty})^\sim$ and the
double ramification cycle.
Another viewpoint on Section \ref{gwt} is that all descendent invariants of a symplectic surface $S$
can be efficiently and effectively determined from the 
rubber integrals 
\begin{equation}\label{drcint}
\langle \tau_0(\gamma_1)\cdots \tau_0(\gamma_n)\, \DR_g(a_1,\ldots,a_n) \rangle^{S}_{g,\beta}\, , \ \ \ 
\DR_g(a_1,\ldots,a_n)\in H^{2g}(\Mbar_{g,n})\, .
\end{equation}
The relationship 
can be made explicit through a Fock space formalism. 

For $K3$ surfaces, an explicit conjecture for the double ramification cycle integrals
\eqref{drcint} in disconnected form
was made in \cite{vIOP} in terms of quasi-Jacobi forms, based on an earlier formulation in \cite{ObPand} in which the universal functions were left undetermined. The formula essentially originates in the computations of the Gromov-Witten invariants of the Hilbert scheme of points of a $K3$ surface in \cite{MR3720346}.

Several 
functions in the variables
$z \in \BC$ and $q$ are required to
state the conjecture:
\begin{enumerate}
\item[$\bullet$]
The Jacobi theta function $\Theta(z,q)$ and the discriminant modular form $\Delta(q)$ were defined in Section~\ref{subsec:an example}.
\item[$\bullet$]For all~$m \in \BZ$, define the quasi-Jacobi form
\[ \varphi_m(z) = \mathrm{Res}_{x=0} \left(\frac{\Theta(x+z,q)}{\Theta(x,q)}\right)^m \, ,\]
where the residue stands for taking the coefficient of $1/x$.
\item[$\bullet$] 
For all~$m,n \in \BZ$,
define quasi-Jacobi forms~$\varphi_{m,n}(z)$ 
by the differential equation
\begin{equation*} D_{q} \varphi_{m,n}(z) = mn\varphi_m(z)\varphi_n(z) \frac{D_{q}^2\Theta(z,q)}{\Theta(z,q)} + (D_{q}\varphi_m(z))(D_{q}\varphi_n(z))\, , \ \ \ 
D_{q} = q \frac{d}{dq}\, ,
\label{defining_diff_eqn} \end{equation*}
together with the condition that the constant term vanishes: 
\[ \varphi_{m,n} = O(q)\, . \]
\end{enumerate}

\begin{conj}[ {\cite[Conjecture 1.5]{vIOP}} ] \label{conj:K3DR} Let $S$ be a K3 surface, and let $\beta \in H_2(S,\BZ)$ be a primitive effective curve class. For $\gamma_1,\ldots,\gamma_n \in H^{\ast}(S,\BQ)$ and $a_1,\ldots,a_n \in \BZ_{\neq 0}$ with $\sum_i a_i = 0$, we have
\begin{multline*}
\sum_{g = 0}^{\infty}
\left\langle \, (a_1, \gamma_1) \cdots (a_n,\gamma_n) \right\rangle^{\bullet, \sim}_{g,\beta}
(-1)^{g+n} z^{2g-2+n} \\
=
\frac{1}{\prod_i a_i^{\deg_{\BC}(\gamma_i)}}
\mathrm{Coeff}_{q^{\frac{1}{2} \beta^2}}\left( 
\sum_{\mathfrak{P}}
\frac{1}{\Theta^2 \Delta} \cdot \prod_{j} (\gamma_{x_j}, \beta) \varphi_{a_{x_j}} \cdot
\prod_{k} (\gamma_{y_k}, \gamma_{z_k}) \varphi_{a_{y_k} a_{z_k}}   \right),
\end{multline*}
where the sum on the right side is over all partitions $\mathfrak{P}$ of the multiset $\{ ({a}_i, \gamma_i) \}_{i=1}^{n}$
into parts of size~$\leq 2$:
\begin{enumerate}
\item[$\bullet$]
the parts of size~$1$ are labeled by~$\{ (a_{x_j}, \gamma_{x_j})\}$, 
\item[$\bullet$]
the parts of size~$2$ are labeled
$\{ (a_{y_k}, \gamma_{a_{y_k}}), (a_{z_k}, \gamma_{a_{z_k}}) \}$. 
\end{enumerate}
Moreover,
$( - , - )$ is the Mukai pairing on~$H^{\ast}(S,\BQ)$ defined by
\[ \big( (r_1, D_1, n_1) , (r_2, D_2, n_2) \big) = r_1 n_2 + r_2 n_1 - \int_{S} D_1 \cup D_2\, . \]
\end{conj}

The convention for the rubber notation of Conjecture \ref{conj:K3DR} is that the negative parts of 
$$(a_1, \gamma_1), \ldots, (a_n,\gamma_n)$$ determine the cohomology-weighted partition over $0\in \p^1$
and the positive parts determine the cohomology-weighted partition over $\infty\in \p^1$.

\subsection{Moduli of $K3$ surfaces} \label{modk3}
Let $\mathcal{M}^{K3}_{2\ell}$ be the moduli space of quasi-polarized $K3$ surfaces $(S,L)$ with $L^2=2\ell$.
We denote the universal $K3$ surface and the universal quasi-polarization{\footnote{For a discussion of the normalization of $\mathcal{L}$, see \cite{PY}.}} by 
$$ \pi: \mathcal{S} \rightarrow \mathcal{M}^{K3}_{2\ell}\, , \ \ \ \mathcal{L} \rightarrow \mathcal{S}\, .$$
Consider the $\pi$-relative moduli space of stable maps,
$$\epsilon^{K3}: \Mbar_{g,n}(\pi, kL)\rightarrow \mathcal{M}^{K3}_{2\ell}\, ,$$
with fiberwise curve class $kL\in H_2(S,\mathbb{Z})$  a multiple of the quasi-polarization.
The families reduced virtual class
can be used to define descendent classes on the moduli space $\mathcal{M}^{K3}_{2\ell}$:
\begin{equation}\label{redfam}
\big\langle \tau_{\ellt_1}(\gamma_1) \cdots \tau_{\ellt_n}(\gamma_n) \big\rangle^{\pi}_{g,kL}
=
\epsilon^{K3}_*\left( \prod_{i=1}^{n} \ev_i^{\ast}(\gamma_i) \psi_i^{\ellt_i} \cap
 [\Mbar_{g,n}(\pi, kL)]^{\red}
\right) \, \in H^*(\mathcal{M}^{K3}_{2\ell})\, ,
\end{equation}
where $\gamma_i \in H^*(\mathcal{S})$.

We can promote $\pi$ to a family of semipositive threefolds by product with $(\p^1, 0\cup \infty)$,
$$ \widehat{\pi}: (\mathcal{S}\times \p^1, \mathcal{S}_0\cup \mathcal{S}_{\infty}) \rightarrow \mathcal{M}^{K3}_{2\ell}\, .$$
In a sequel \cite{nextpaper}, we will study GW/PT for the family $\widehat{\pi}$ 
and the families PT fixed loci (as in proof
of Theorem \ref{mainn}). The wallcrossing argument in families is technically more difficult. We expect the
methods to yield the following result.

\vspace{6pt}
{\em \noindent If the families GW/PT correspondence holds for
the families reduced theory of
$$ \widehat{\pi}: (\mathcal{S}\times \p^1, \mathcal{S}_0\cup \mathcal{S}_{\infty}) \rightarrow \mathcal{M}^{K3}_{2\ell}\, $$
with primary insertions, then
$\big\langle \tau_{\ellt_1}(\gamma_1) \cdots \tau_{\ellt_n}(\gamma_n) \big\rangle^{\pi}_{g,\beta}\in RH^*(\mathcal{M}^{K3}_{2\ell})$
for $\gamma_i\in RH^*(\mathcal{S})$.}
\vspace{6pt}

Definitions, results, conjectures related to the tautological rings of $\mathcal{M}^{K3}_{2\ell}$ and  $\mathcal{S}$
can be found in \cite {MOP,PY,COP}. In particular, the above statement about $RH^*(\mathcal{M}^{K3}_{2\ell})$  was 
proven{\footnote{The genus 0 and 1 results of \cite{PY} concern
classes in $RH^*(\mathcal{S}^m)$ where
$\pi^m:\mathcal{S}^m \rightarrow \mathcal{M}_{2\ell}^{K3}$
is the fiber product of the universal $K3$ surface, see also \cite[Conjecture 1]{PY}. For simplicity, we have only discussed $\mathcal{M}_{2\ell}$ here, but the
lift to $\mathcal{S}^m$ will be studied in \cite{nextpaper}.}
for certain $g=0$ and $g=1$
cases in \cite{PY} and used there to prove the generation of the tautological ring of $\mathcal{M}^{K3}_{2\ell}$
by Noether-Lefschetz classes.


\subsection{Algebraic twistor families}
While Pardon's approach to GW/PT in families is best suited for our paper (especially for the
study of families of threefolds over $\mathcal{M}^{K3}_{2\ell}$), much less is required for
Theorem \ref{mainn}: {\em only the GW/PT correspondence for particular 1-parameter families of threefolds}.

Using the GW/NL correspondence \cite{GWNL,PTKKV}, the GW/PT correspondence
for the reduced theory of $(S\times \p^1, S_{0,\infty})$ with primary insertions can be obtained by proving
a families descendent GW/PT  correspondence for
$$\widehat{\pi}: \mathcal{S} \times \p^1 \rightarrow C\, ,$$
where $\pi:\mathcal{S} \rightarrow C$ is family of $K3$ surfaces over a curve. These 1-parameter families GW/PT correspondences can be approached
by the degeneration methods used for hypersurfaces in \cite{PP}.

As an example, consider the family of quartic $K3$ surfaces defined by a general nonsingular hypersurface
of type $(4,2)$ in $\p^3\times \p^1$,
$$\p^3\times \p^1 \supset \mathcal{S} \stackrel{\pi}\longrightarrow \p^1\,, $$
with equation $Q$. The family can be deformed by a parameter $s \rightarrow 0$,
\begin{equation} \label{deggg}
 \p^3\times \p^1  \supset \mathcal{S}_s \stackrel{\pi_s} \longrightarrow \p^1\, ,
 \end{equation}
defined by the equation $sQ-F_1F_2$, where $F_1,F_2$ are general hypersurfaces of type $(2,1)$
in $\p^3\times \p^1$ with vanishing loci
$$\p^3\times \p^1 \supset \mathcal{R}_1,\mathcal{R}_2 \, \stackrel{\tiny{\pi_1,\pi_2}}\longrightarrow\,  \p^1\,.$$
Semistable reduction is required to make the degeneration \eqref{deggg} log smooth 
as $s\rightarrow 0$. Once log smoothness is achieved (by modification over $s=0$), the
descendent GW/PT correspondence for the family
\begin{equation} \label{orig}
\widehat{\pi}: \mathcal{S}\times \p^1 \stackrel{\pi}\longrightarrow \p^1
\end{equation}
can be reduced to the descendent GW/PT correspondence for the simpler families 
$$ \widehat{\pi}_1: \mathcal{R}_1\times \p^1 \stackrel{\pi}\longrightarrow \p^1\, , \ \ \
\widehat{\pi}_2: \mathcal{R}_2\times \p^1 \stackrel{\pi}\longrightarrow \p^1\, ,$$
where $\mathcal{R}_1$ and $\mathcal{R}_2$
are the families of rational surfaces defined
by $F_1$ and $F_2$.
By further such degenerations, the descendent GW/PT correspondence for \eqref{orig} can be reduced
to the descendent GW/PT correspondence for degenerations of toric varieties. The latter
is a new direction which should be approachable via the torus equivariant
descendent correspondence for toric threefold together with further degeneration arguments.

The above sketch is {\em not} a proof, but an outline of a degeneration strategy to
prove the GW/PT correspondence for certain explicit 1-parameter families of $K3$ surfaces. For the
GW/PT correspondence of Theorem \ref{mainn}, we will require the more complicated
families discussed
in \cite{PTKKV}. We expect, with the additional log degeneration tools for the
descendent GW/PT correspondence which will be developed in upcoming work by Maulik and Ranganathan, 
the above outline should be possible to complete.

\appendix
\section{Wallcrossing}
\subsection{Overview} We present here the proof of Theorem~\ref{thm:wallcrossing}.
Wallcrossing formulas for moduli spaces of stable sheaves on surfaces
have been obtained in the seminal work of Mochizuki \cite{Mochizuki}, building on earlier ideas of Thaddeus in the curve case \cite{Thaddeus}.
More recently, there have been fundamental developments through  the work of Joyce and collaborators, where a general wallcrossing formula for moduli spaces of sheaves encoded in terms of vertex algebras \cite{Joyce1,Joyce2} is proven.
The latter work applies, in particular, to surfaces and also handles reduced obstruction theories. While
Theorem  \ref{thm:wallcrossing} should follow from an application of \cite{Joyce2},
we will use the method
of Mochizuki for a simpler treatment.

We follow here recent work of Kuhn-Liu-Thimm \cite{KLT2,KLT}, in particular  Sections 3 and 4 of \cite{KLT}. In \cite{KLT}, the authors prove a K-theoretic DT/PT correspondence for quasi-projective (toric) Calabi-Yau threefolds (and endow the moduli space with an equivariant symmetric perfect obstruction theory). We are tempted to just refer to their work, but we are proving a correspondence on the surface $S$ of 
DT/PT type 
(a shadow of the DT/PT correspondence for the non Calabi-Yau geometry $S \times \p^1$), which does not exactly fall into their framework. Our discussion is somewhat simpler than \cite{KLT}, since we do not have to deal with
almost-perfect obstruction theories.
Since our setup is extremely similar to \cite{KLT}, 
we will refer to their work for certain details on properties of moduli spaces.

\subsection{Stack of all $1$-dimensional quotients with $0$-dimensional cokernel}
{Let $S$ be a nonsingular projective surface.}
Consider the stack
parameterizing maps from ideal sheaves of colength $n$ to $1$-dimensional sheaves on $S$ with Chern characters $\alpha=((1,0,-n),(0,\beta,m))$,
\[ \mathfrak{N}_{\alpha} = 
\left\{ s: I_z \to F \middle| \begin{array}{c} z \in S^{[n]}\,,\ F \in \Coh(S)\,,\ \ch(F)=(0,\beta,m)\, , \\ 
 \mathrm{coker}(s)\, \text{is 0-dimensional}
\end{array} \right\}.
\]
Objects of $\FN_{\alpha}$ over a finite type scheme $T$ consist of morphisms of coherent sheaves $$s : \CI \to \CF\ \ \ \text{and}\ \ \ \iota : \CI \hookrightarrow \CO_{S \times T}\,,$$ where the subscheme defined by the ideal $\CI$ is flat and finite of length $n$ over $T$, $\CF$ is flat over $T$ with $\ch(\CF_t) = \alpha$ for all closed points $t$ (in particular, $\CF_{t}$ is $1$-dimensional), and $\mathrm{Coker}(s)$ is finite over $T$.
An isomorphism $$\varphi : (s:\CI \to \CF,\iota) \xrightarrow{\cong} (s':\CI \to \CF',\iota')$$ of two such families is given by isomorphisms $\varphi_0 : \CI_{\CZ} \to \CI_{\CZ}$ and $\varphi_1 : \CF \to \CF'$, such that $\iota' \circ \varphi_0 = \iota$ and $\varphi_1 \circ s = s' \circ \varphi_0$.
In particular, an automorphism $\varphi=(\varphi_0,\varphi_1)$ of $(s : \CI \to F,\iota)$ must satisfy
$\varphi_0=\id$ and $\varphi_1 \circ s=s$.
In what follows we will usually suppress the inclusion $\iota$ in the notation (as $\iota$ is unique up to scalar and only serves to rigidify the stack).

The stack $\mathfrak{Y}$ of $1$-dimensional sheaves on $S$ 
embeds as an open subset in the stack of all glueable objects in $D^b(S)$ constructed by \cite{Lieblich}, so $\mathfrak{Y}$ is an Artin stack locally of finite type.
The stack $\mathfrak{N}_{\alpha}$ therefor embeds  into a natural Hom stack over $S^{[n]} \times \mathfrak{Y}$, so $\mathfrak{N}_{\alpha}$ is an Artin stack locally of finite type.

\begin{lemma} \label{lembound}
 Fix an integer
$n\geq0$ and a class $\beta\in H^2(S,\mathbb{Z})$. 
The stack $\FN_{\alpha}$ for 
$$\alpha = ((1,0,-n),(0,\beta,m))$$
is empty if $m \ll 0$.
Furthermore, for fixed $m$, the set of sheaves $F$ which appear in $\mathfrak{N}_{\alpha}$ is bounded. 
\end{lemma}

\begin{proof}
Consider a map 
$\varphi : I_z \to F$
with $0$-dimensional cokernel
defining an object in $\FN_{\alpha}$.
Let $$F' = \mathrm{Im}(\varphi)\,, \ \ \ \ch(F') = (0,\beta,m')\, .$$
We first prove that $m' \geq m_0(n,\beta)$ for some $m_0(n,\beta) \in \BZ$ only depending
on $n$ and $\beta$. To the quotient $\varphi:I_z \to F'$, we associate the pushout $G$ along $j:I_z \to \CO_S$, which is defined by
$$G = \mathrm{Coker}(I_z \xrightarrow{(-\phi,j)} F \oplus \CO_S)\, .$$
We obtain  a sequence
$\CO_S \to G \to \CO_z$,
where both maps are surjections and $F' = \mathrm{Ker}(G \to \CO_z)$.
It is well-known that the quotients $\CO_S \to G$ can only exist if $\ch_3(G) \geq m_0'(\beta)$ for some $m'_0(\beta)$. Hence, 
$$m' \geq m_0(\beta)-n\, .$$

For fixed $m$,
the length of $Q = \mathrm{Coker}(\varphi)$ is bounded since
$m' \geq m_0(\beta,n)$, 
The family of sheaves of possible cokernels $Q$ is therefore bounded. 
The possible  sheaves  $F'$ are bounded by the boundedness of the Quot scheme. Hence, their extensions $F$ form a bounded family.
\end{proof}

Fix a polarization $\CO_S(1)$ on $S$. Recall that $F$ is {\em $\kappa$-regular} if $H^i(F(\kappa-i))=0$ for all $i>0$.  
By a standard argument, using relative Serre vanishing plus cohomology and base change,
a uniform regularity bound can be found for any bounded family of coherent sheaves. By the boundedness claim of 
Lemma \ref{lembound}, we obtain
the following result.

\begin{lemma} 
There exists a $\kappa$ such that $F(\kappa)$ is $\kappa$-regular for all $F$ which occur in $\FN_{\alpha}$.
\end{lemma}

As in \cite[Lemma 3.2.5]{KLT}, we have a partial properness result for $\FN_{\alpha}$.

\begin{lemma} \label{lemma:existence for FN}
The stack $\FN_{\alpha}$ satisfies the existence part of the valuative criterion of properness.
\end{lemma}
\begin{proof}
Let $R$ be a curve with generic point $\eta$. Let $$s_{\eta} : \CI_{\eta} \to F_{\eta}$$ be an element of $\FN_{\alpha}$ over $\eta$. By the properness of the Hilbert scheme of points, we can find a unique extension of $\CI_{\eta}$ to an ideal sheaf $\CI$ over $R$. 
We can find an extension $s_{\eta}$ to $s : \CI \to F$ 
for {\em some} flat extension $F$ of $F_\eta$ over $R$.
If the cokernel $Q$ of $s$ is not 0-dimensional over the special point $\zeta$, then 
let $Q'_\zeta$ be a positive dimensional quotient of $Q_\zeta$ over the special point $\zeta$ such that the kernel 
$$\widetilde{Q} = \mathrm{Ker}(Q \to Q_\zeta')$$ does not have $1$-dimension support over the special point.
Let
$$\widetilde{F} = \mathrm{Ker}(F \to Q \to Q'_\zeta)\, .$$ Then, $s:\CI \to \widetilde{F}$ is the desired extension.
\end{proof}

\subsection{Framing}

We fix $\kappa$ throughout such that for all $(s : I_z \to F)\in\FN_{\alpha}$, the sheaf $F$ is $\kappa$-regular.
Let $\FQ$ be the Artin stack of 0-dimensional sheaves on $S$,
$$\FQ=\bigsqcup_{l\geq 0} \FQ_l\,,$$
where $l$ is the length.
All sheaves of $\FQ$ are automatically $\kappa$-regular.
Consider the functor
\[ \Fr : \Coh(S) \to \mathrm{Vect}_{\BC}\, ,\ \ 
\quad
E \mapsto \Fr(E) = H^0(S, E(\kappa))\, .
\]
Certainly, $\Fr$ is exact on short exact sequences which involve only sheaves $F$ where $(s:I_z \to F) \in \FN_{\alpha}$ and sheaves of $\FQ$.
If $I = (s:I_z \to F) \in \FN_{\alpha}$, we also write $$\Fr(I) = \Fr(F) = H^0(S,F(\kappa))\, .$$

By regularity, for any object $I = (s : I_z \to F) \in \FN_{\alpha}$, 
the dimension of $H^0(S,F(\kappa))$ is 
given by
\[ d_{\Fr} = \dim_{\BC} H^0(S,F(\kappa))
= \chi( F(\kappa)) = \int_{S} (0,\beta,m)\cdot (1, \kappa H,\frac{\kappa^2 H^2}{2})\cdot \td(S)\, . \]

\begin{defn} \label{defn:flag stack}
For $N \geq 1$, let $\mathbf{d}=(d_1,\ldots,d_N,d_{N+1})$ be a vector of non-negative integers satisfying:
$$d_1\in \{0,1\}\, , \ \ \ 
d_i \leq d_{i+1} \leq d_{i}+1\, , \ \ \ d_{N+1} = d_{\Fr}\, .$$
Let $\FNfl_{\alpha, \mathbf{d}}$ be the moduli stack parameterizing triples $(I,V,\rho)$ where
\begin{itemize}
\item[(i)] $I$ is an object of $\FN_{\alpha}$,
\item[(ii)] $V = (V_i)_{i=1}^{N}$ are $\BC$-vector spaces with $\dim V_i = d_i$,
\item[(iii)] $\rho = (\rho_{i})_{i=1}^{N}$ are injective linear maps $$\rho_i : V_i \to V_{i+1}\, ,$$
with $V_0=0$, $\rho_0=0$, and $V_{N+1} = \Fr(I)$\, .
\end{itemize}
We view $(V,\rho)$ as defining a flag of subspaces
\[ 0=V_0 \subset V_1 \subset \ldots \subset V_N \subset V_{N+1} = \Fr(I). \]
An isomorphism $\varphi : (I,V,\rho) \to (I',V',\rho')$ is an isomorphism $\varphi : I \to I'$ in $\FN_{\alpha}$ such that $$\Fr(\varphi_1) : \Fr(I) \to \Fr(I')$$ sends the subspace $V_i$ to $V_i'$.
\end{defn}


\begin{defn}
For $N \geq 1$, let $\mathbf{d}=(d_1,\ldots,d_N,d_{N+1})$ be a vector of non-negative integers satisfying:
$$d_1\in \{0,1\}\, , \ \ \ 
d_i \leq d_{i+1} \leq d_{i}+1\, , \ \ \ d_{N+1} = n\, .$$
Let $\FQ_{n,\mathbf{d}}^{\mathrm{fl}}$ be the stack parameterizing triples $(Q,V,\rho)$ where:
\begin{enumerate}
\item[(i)] $Q$ is an object of $\FQ_n$,
\item[(ii)] $V = (V_i)_{i=1}^{N}$ are $\BC$-vector spaces with $\dim V_i = d_i$,
\item[(iii)] $\rho = (\rho_{i})_{i=1}^{N}$ are injective linear maps $$\rho_i : V_i \to V_{i+1}\, ,$$
with $V_0=0$, $\rho_0=0$, and $V_{N+1} = \Fr(Q)$\, .
\end{enumerate}
The isomorphisms of $\FQ_{n,\mathbf{d}}^{\mathrm{fl}}$ are defined following the definition of the isomorphisms for
$\FNfl_{\alpha, \mathbf{d}}$.
\end{defn}

The morphisms which forget the flag $(V,\rho)$,
\[ \Pi : \FNfl_{\alpha, \mathbf{d}} \to \FN_{\alpha}, \quad \Pi_{Q} : \FQfl_{n,\mathbf{d}} \to \FQ_n \]
are flag bundles for the vector bundles on $\FN_{\alpha}$ and $\FQ_n$ with fibers $\Fr(I)$ and $\Fr(Q)$ respectively. In particular,
$\FNfl_{\alpha,\mathbf{d}}$ and $\FQfl_{\mathbf{d}}$ are Artin stacks locally of finite type.\footnote{Strictly speaking, Definition~\ref{defn:flag stack} is only for $\BC$-valued points. The flag bundle perspective extends the definition to families. We leave the details to the reader, see \cite[Section 2.4.3]{KLT2}.}
Let $\CV_i$ denote the universal bundles on $\FNfl_{\alpha, \mathbf{d}}$
and $\FQfl_{n,\mathbf{d}}$
with fibers $V_i$ at closed points.

\begin{rmk}
If $d_i = d_{i+1}$ for some $i$, then there is an isomorphism
\[ \FNfl_{\alpha, \mathbf{d}} \to \FNfl_{\alpha, (d_1,\ldots,d_i, d_{i+2},\ldots,d_{N+1})} \]
defined by omitting $V_{i+1}$ from the flag, and similar for $\FQfl_n$. In particular, we have an isomorphism
\begin{equation} \label{30fsdfsdF} \FNfl_{\alpha, \mathbf{d}} \cong \FNfl_{\alpha,(1,2, 3, \ldots,d_{\Fr})}\, , \end{equation}
where the elements of $(1,2,3,\ldots, d_{\Fr})$ are distinct.
We will define open subsets of $\FNfl_{\alpha, (1,2,3,\ldots,d_{\Fr})}$ of stable and semistable objects.
Stable objects for all $\FNfl_{\alpha, \mathbf{d}}$ will then be defined by 
isomorphism \eqref{30fsdfsdF}.
\end{rmk}

\subsection{Stability condition}
We define stability conditions for the objects of $\FNfl_{\alpha,\mathbf{d}}$.

\begin{defn}[$\ell$-stability]
Let $\FN^{\mathrm{fl},\ell-\mathrm{st}}_{\alpha,\mathbf{d}} \subset \FNfl_{\alpha,(1,2,3,\ldots,d_{\Fr})}$ be the open substack of objects 
$$(I=(s:I_z \to F),V,\rho)$$ 
satisfying the conditions:
\begin{enumerate}
\item[(i)] for every injection $Z \subset F$ with $Z$ a zero-dimensional sheaf, 
\[ V_{\ell} \cap \Fr(Z) = \{ 0 \} \subset \Fr(I)\,  \]
\item[(ii)] for every non-zero surjection $F/\mathrm{Im}(s) \to Z$, with $Z$ a $0$-dimensional sheaf,
the composition
\[ V_{\ell} \to \Fr(Z) \]
is non-zero.
\end{enumerate}
If conditions (i) and (ii) hold,  then $(I,V,\rho)$ is {\em $\ell$-stable}.
\end{defn}

\begin{rmk}
Let $(I,V,\rho) \in \FN_{\alpha}$ with $I=(s:I_z \to F)$.
If $(I,V,\rho)$ is $0$-stable, then 
condition (i) is always satisfied, and
$s$ is a surjection by condition (ii). Therefore, $I \in \Quot^{\flat}_{\alpha}$.
If $(I,V,\rho)$ is $d_{\Fr}$-stable, then $F$ is pure by condition (i) and condition (ii) is always satisfied.
Therefore, $I \in \Quot^{\sharp}_{\alpha}$.
Thus, the restrictions of $\Pi$ to these loci,
\[ \FN^{\mathrm{fl},0-\mathrm{st}}_{\alpha,(1,2,3,\ldots,d_{\Fr})} \to \Quot^{\flat}_{\alpha}\, ,
\quad
\FN^{\mathrm{fl},d_{\Fr}-\mathrm{st}}_{\alpha,(1,2,3,\ldots,d_{\Fr})} \to \Quot^{\sharp}_{\alpha}.
\]
are simply relative flag bundles over the proper moduli spaces $\Quot^{\star}_{\alpha}$.
\end{rmk}

\begin{prop}
The stack $\FN^{\mathrm{fl},\mathrm{\ell}-st}_{\alpha,\mathbf{d}}$ is represented by an algebraic space.
\end{prop}
\begin{proof}
We check that if $(I,V,\rho)$ is $\ell$-stable, then the stabilizer group is trivial.
Let $$\varphi : (I,V,\rho) \to (I,V,\rho)$$ be an automorphism where $I=(s : I_z \to F)$.
Then $\varphi_0=\id$. Let $\delta = \varphi_1 - \id_F$.
Since $\delta|_{\mathrm{Im}(s)} = 0$, $\delta$ factors through the cokernel of $s$,
\[ \delta : F \to \mathrm{Coker}(s) \overset{\overline{\delta}}{\twoheadrightarrow} Z = \mathrm{Im}(\delta) \subset F\, , \]
where $Z$ is $0$-dimensional.
If $Z$ is non-trivial, then $\overline{\delta} : F/\mathrm{Im}(s) \to Z$ is non-zero, so by condition (ii)  of the definition of $\ell$-stability 
\begin{equation} \label{latmor}
V_{\ell} \to \Fr(Z) \subset \Fr(I)
\end{equation}
must be non-zero. But, by assumption, $\Fr(\varphi_1)$ preserves the filtration, so the image of 
\eqref{latmor}
is contained in $V_{\ell} \cap \Fr(Z)$ which is zero by condition (i) of $\ell$-stability.
Therefore,  $Z$ is trivial, so $\delta=0$.
\end{proof}


\begin{defn} [Stability for $0$-dimensional sheaves]
Let $\FQ^{\mathrm{fl},\mathrm{st}}_{n,\mathbf{d}} \subset \FQfl_{n,\mathbf{d}}$ be the open substack of objects $(Q,V,\rho)$, such that if $i_0$ is the smallest integer satisfying $d_{i_0} > 0$, then the composition $$V_{i_0} \otimes \CO_S \to \Fr(Q) \otimes \CO_S \to Q(\kappa)$$ is surjective.
The objects $(Q,V,\rho)$ of $\FQ^{\mathrm{fl},\mathrm{st}}_{n,\mathbf{d}}$ are {\em stable}.
\end{defn}

Let $\FQ^{\mathrm{pl},\mathrm{st}}_{n,\mathbf{d}}$ be the $\BG_m$-rigidification{\footnote{The superscript 'pl' stands for projective linear.}} of $\FQ^{\mathrm{fl},\mathrm{st}}_{n,\mathbf{d}}$ obtained by quotienting  the stabilizer group by the natural $\BC^{\ast}$-scaling action.
Alternatively, $\FQ^{\mathrm{pl},\mathrm{st}}_{n,\mathbf{d}}$ is the stack of tuples $(Q,V,\rho,\sigma)$ where $(Q,V,\rho) \in \FQ^{\mathrm{fl},\mathrm{st}}_{n,\mathbf{d}}$ and $\sigma: \BC \xrightarrow{\cong} V_{i_0}$ is an isomorphism (where $i_0$ is again the smallest integer such that $d_{i_0}>0$).

\begin{lemma}
The stack $\FQ^{\mathrm{pl},\mathrm{st}}_{n,\mathbf{d}}$ is a relative flag bundle over $S^{[n]}$ and, therefore,
is a nonsingular projective variety.
\end{lemma} 
\begin{proof}
We define a forgetful morphism
\[ \Pi : \FQ^{\mathrm{pl},\mathrm{st}}_{n,\mathbf{d}} \to S^{[n]} \]
be sending a tuple
$(Q,V,\rho,\sigma)$ to the composition $$\CO_S \xrightarrow{\sigma} V_{i_0} \to \Fr(Q) \otimes \CO_S \twoheadrightarrow Q(\kappa)\, .$$
Therefore, $Q(\kappa) = \CO_z$ where $z \in S^{[n]}$ is the image point, and $\Fr(Q) = H^0(\CO_z)$.
The fiber of $\Pi$ over  $z \in S^{[n]}$ is a flag variety of $H^0(S,\CO_Z)/\BC \id$.
So
$\FQ^{\mathrm{pl},\mathrm{st}}_{n,\mathbf{d}}$ is a flag bundle associated to the vector bundle 
$$p_{\ast}(\CO_{\CZ})/\CO_{S^{[n]}} \rightarrow S^{[n]}\, ,$$
where $p : \CZ \to S^{[n]}$ is the projection from the universal subscheme $\CZ \subset S^{[n]} \times S$.
\end{proof}

\begin{defn}[$(\ell,\ell+1)$-semistability] \label{defn l l+1 stability}
Let $\FN^{\mathrm{fl},(\ell,\ell+1)-\mathrm{ss}}_{\alpha,(1,2,3,\ldots,d_{\Fr})} \subset \FNfl_{\alpha,(1,2,3,\ldots,d_{\Fr})}$ be the open substack of objects
$$(I=(s:I_z \to F),V,\rho)$$ 
satisfying the conditions:
\begin{enumerate}
\item[(i)] for every injection $Z \subset F$ with $Z$ a zero-dimensional sheaf, 
\[ V_{\ell} \cap \Fr(Z) = \{ 0 \} \subset \Fr(I)\, , \]
\item[(ii)] for every non-zero surjection $F/\mathrm{Im}(s) \to Z$, with $Z$ a $0$-dimensional sheaf,
the composition
\[ V_{\ell+1} \to \Fr(Z) \]
is non-zero.
\end{enumerate}
If conditions (i) and (ii) hold, then $(I,V,\rho)$ is {\em $(\ell,\ell+1)$-semistable}.
\end{defn}

\begin{lemma} \label{lemma:splitting of l l+1 stability}
Let $(I=(s:I_z \to F),V,\rho) \in \FN^{\mathrm{fl},(\ell,\ell+1)-\mathrm{ss}}_{\alpha,(1,2,3,\ldots,d_{\Fr})}$
have a non-trivial automorphism. Then,
\begin{itemize}
    \item [$\bullet$]
there exists a canonical decomposition
$F = F' \oplus Z$, where the map $s$ factors as $s: I_z \to F' \subset F$.
\end{itemize}
Let $\ch(F')=(0,\beta,m')$, $\alpha'=((1,0,-n),\ch(F'))$, and $\ch_2(Z)=m''$. Let
 $(V',\rho')$ and $(V'',\rho'')$ be the induced filtrations of $\Fr(F')$ and $\Fr(Z)$ with dimension vectors $\mathbf{d}'$ and $\mathbf{d}'' = \mathbf{d}-\mathbf{d}'$. 
Then, 
\begin{itemize}
\item $((s:I_z \to F'), V',\rho') \in \FN_{\alpha',\mathbf{d}'}^{\mathrm{fl},\ell-\mathrm{ss}}$,
\item $(Z, V'',\rho'') \in \FQ_{m'',\mathbf{d}''}^{\mathrm{fl},\mathrm{st}}$,
and $d''_i=0$ for $i<\ell+1$ and $d''_{\ell+1} = 1$.
\end{itemize}
Moreover, the stabilizer group of $(I,V,\rho)$ is $\BC^{\ast}$ and given by scaling the sheaf $Z$.

Conversely, every pair $((s:I_z \to F'), V',\rho')$ and $(Z, V'',\rho'')$ as above define
a $(\ell,\ell+1)$-semistable object in $\FN^{\mathrm{fl},(\ell,\ell+1)-\mathrm{ss}}_{\alpha,(1,2,3,\ldots,d_{\Fr})}$.
\end{lemma}
\begin{proof}
Let $\varphi : (I,V,\rho) \to (I,V,\rho)$ be a non-trivial automorphism.
Let $\delta = \varphi_1 - \id$.
Then $\delta : F \to F$ factors as before as
$$\overline{\delta} : F/\mathrm{Im}(s) \twoheadrightarrow Z = \mathrm{Im}(\delta) \subset F\, .$$
We have $\Fr(\delta)(V_{\ell}) = 0$ by stability.
Moreover, $\Fr(\delta)(V_{\ell+1}) \to \Fr(Z)$ is non-zero, so injective (since $V_{\ell+1}/V_{\ell}$ is $1$-dimensional).
Let $W = V_{\ell+1} \cap \Fr(Z)$. Since $V_{\ell} \cap \Fr(Z) = 0$ by stability, and 
$$\Fr(\delta)(V_{\ell+1}) \to W \subset \Fr(Z)$$ is non-zero, we see that $W$ is $1$-dimensional.
The map $\Fr(\delta)$ sends $W$ to itself and hence acts by scaling by some $\lambda \in \BC^{\ast}$.
By condition (ii) of Definition~\ref{defn l l+1 stability}, the map $W \otimes \CO_S \to Z$ is surjective (the map $W \otimes \CO_S \to Z'$ must be non-zero for any quotient $Z \to Z'$). 
Hence, we see that $\delta/\lambda$ acts by the identity on $Z$,
so $\delta/\lambda$ is a projector onto the subsheaf $Z \subset F$.
In particular, if $F' = \mathrm{Ker}(\delta)$, then we have $F = F' \oplus Z$.

Since $\Fr(\delta)$ respects the filtration, we obtain a decomposition
$V'_i = \Fr(F') \cap V_i$, $V''_i = \Fr(Z) \cap V_i$ with corresponding morphisms $\rho'_i = \rho_i|_{V_i'}$ and $\rho''_i = \rho_i|_{V_i''}$.
Since $V'_\ell = V'_{\ell+1}$, we see that $(I_z \to F', V',\rho')$ is $\ell$-stable.
By construction $V''_{\ell+1} = W$ and $W \otimes \CO_S \to Z$ is surjective, so $Z$ is stable.
By $\ell$-stability of $(I_z \to F', V',\rho')$ every automorphism restricts to the identity on $F'$.
On $Z$, an automorphism must act faithfully on the $1$-dimensional vector space $W=V''_{\ell+1}$, so must act by scaling.
\end{proof}

\begin{prop} The following two properties hold:
\begin{enumerate}
\item[(i)] $\FN^{\mathrm{fl},\ell-\mathrm{st}}_{\alpha,(1,2,3,\ldots,d_{\Fr})}$ is a proper algebraic space,
\item[(ii)] $\FN^{\mathrm{fl},(\ell,\ell+1)-\mathrm{ss}}_{\alpha,(1,2,3,\ldots,d_{\Fr})}$ satisfies the 
existence part of the 
valuative criterion of properness.
\end{enumerate}
\end{prop}

\begin{proof}[Proof (Sketch)]
We closely follow the arguments of \cite[Sections 5.1.8 and 5.1.9]{KLT}.
First we define  a weak stability condition on pairs 
$({v},\mathbf{d}) \in H^{\ast}(S,\BQ) \times \BZ^{N}$ by
\[ \tau_{\mu}(v,\mathbf{d}) =
\begin{cases}
\mu - i_{\mathrm{min}}(\mathbf{d}) & \text{ if } \rk(v) = 0 \\
0 & \text{ if } \rk(v) = 1
\end{cases}
\]
where $\mu \in \BR$ is a stability parameter and $i_{\mathrm{min}}(\mathbf{d})$ is the smallest integer $i$ such that $d_i > 0$. An object $(E,V,\rho)$ in either $\FNfl$ or $\FQfl$ is $\tau_{\mu}$-semistable (respectively $\tau_{\mu}$-stable) if, for every short exact sequence $$F \to E \to G$$ of objects from these stacks (defined in the natural way), we have $\tau_{\mu}(F) \leq \tau_{\mu}(G)$ (respectively $\tau_{\mu}(F) < \tau_{\mu}(G)$).
We then check:
\begin{enumerate}
\item[$\bullet$]
for  $(Z,V,\rho) \in \FQfl$,  $\tau_{\mu}$-stability $\Leftrightarrow$ stability,
\item[$\bullet$] for $(I,V,\rho)\in\FN^{\mathrm{fl},\ell-\mathrm{st}}_{\alpha,(1,2,3,\ldots,d_{\Fr})}$, $\ell$-stability $\Leftrightarrow$
$\tau_{\mu}$-semistability $\Leftrightarrow$ $\tau_{\mu}$-stability for any $\ell < \mu < \ell+1$,
\item[$\bullet$] for  $(I,V,\rho)\in\FN^{\mathrm{fl},\ell-\mathrm{st}}_{\alpha,(1,2,3,\ldots,d_{\Fr})}$,
$(\ell,\ell+1)$-semistability  $\Leftrightarrow$ $\tau_{\mu}$ semistability for $\mu=\ell+1$.
\end{enumerate}

Consider now the existence part of the valuative criterion of properness for either 
$\FN^{\mathrm{fl},\ell-\mathrm{st}}_{\alpha,\mathbf{d}}$ or $\FN^{\mathrm{fl},(\ell,\ell+1)-\mathrm{st}}_{\alpha,\mathbf{d}}$. Both are open substacks of 
$\FNfl_{\alpha,\mathbf{d}}$.
Given a family of such objects over the generic point over a DVR,
by Lemma~\ref{lemma:existence for FN} the object can be extended to $\FNfl_{\alpha,\mathbf{d}}$ (extend the object of $\FN_{\alpha}$ and then use the properness of the relative flag space to extend the flag).
If the special fiber is not $\tau_{\mu}$-semistable, one can repeatedly apply elementary transformation in the style of Langton to finally end up with a $\tau_{\mu}$-semistable model.
This standard argument is discussed and explained in \cite[Section 4.3]{KLT}.

The separatedness of $\FN^{\mathrm{fl},\ell-\mathrm{st}}_{\alpha,\mathbf{d}}$ follows since for $\mu \in \BR \setminus \BZ$ semistability for objects in $\FNfl$ is equivalent to stability (see \cite[Lemma 4.3.2]{KLT}).
\end{proof}


\subsection{Master space}
The stack $\FN^{\mathrm{fl},(\ell,\ell+1)-\mathrm{ss}}_{\alpha,\mathbf{d}}$ 
interpolates between $\ell$ and $(\ell+1)$-stability, with the advantage of having only $\BC^{\ast}$-automorphisms. We construct a proper algebraic space that sits on top of
$\FN^{\mathrm{fl},(\ell,\ell+1)-\mathrm{ss}}_{\alpha,\mathbf{d}}$ and resolves the non-trivial stabilizers.
\begin{defn}
The master space $\BM_{\ell,\ell+1}$ is defined as the stack parameterizing tuples $$(I,(V_{-1},V), (\sigma,\rho_{-1},\rho))\, ,$$
\begin{itemize}
\item $(I,V,\rho)$ is an element of $\FNfl_{\alpha,(1,2,3,\ldots,d_{\Fr})}$ satisfying $(\ell,\ell+1)$-semistability,
\item $V_{-1}$ is a $1$-dimensional vector space,
\item $\rho_{-1}$ and $\sigma$ are linear maps,
\[ \rho_{-1} : V_{\ell+1}/V_{\ell} \to V_{-1}, \quad \sigma : \BC \to V_{-1}\, , \]
\end{itemize}
satisfying the following conditions:
\begin{enumerate}
\item[(i)] if $\rho_{-1} = 0$, then $(I,V,\rho)$ is $(\ell+1)$-stable,
\item[(ii)] if $\sigma = 0$, then $(I,V,\rho)$ is $\ell$-stable,
\item[(iii)] at least one of $\rho_{-1},\sigma$ is non-zero.
\end{enumerate}
An automorphism of $(I,(V_{-1},V), (\sigma,\rho_{-1},\rho))$ is given by an automorphism $\varphi$ of $I$ together with
an isomorphism $\psi : V_{-1} \to V_{-1}$ such that all maps commute
$$\rho_{-1} \circ \Fr(\varphi_1) = \psi \circ \rho_{-1}\, , \ \ \  \psi \circ \sigma = \sigma\, .$$
\end{defn}

\begin{prop}
$\BM_{\ell,\ell+1}$ is a proper algebraic space.
\end{prop}
\begin{proof} The objects are easily seen to have only trivial automorphisms groups, hence $\BM_{\ell,\ell+1}$ is an algebraic space.
For properness, we again define a weak stability function
and use a Langton style algorithm, see \cite[Sections 5.2.3,5.2.4]{KLT}.
\end{proof}

\subsection{Perfect obstruction theory}
We start with general remarks about the relative tangent bundles of flag bundles which are well-known.

\begin{lemma} \label{lemma:relative tangent flag bundle}
Let $E$ be a locally free sheaf on a scheme/stack/derived stack $X$.
Let $0 \leq \ell \leq N$ be fixed.
\begin{enumerate}
\item[(a)] Let $\Pi : \mathrm{Fl}(E) \to X$ be a flag bundle parameterizing flags of subspaces $$V_1 \hookrightarrow \ldots \hookrightarrow V_{N} \hookrightarrow V_{N+1} = E_x$$ for points $x \in X$. Then, $\Pi$ is smooth with relative tangent bundle 
\[ T_{\Pi} = \mathrm{Cone}\left( \bigoplus_{i=1}^{N} \hom(\CV_i,\CV_i) \to \bigoplus_{i=1}^{N} \hom(\CV_i,\CV_{i+1}) \right)\,, \]
where $\CV_i$ are the universal bundles which fibers $V_i$ (in particular, $\CV_{N+1} = \Pi^{\ast}(E)$, and the complex of which we take the cone is in degree $[-1,0]$).
\item[(b)] Let $\widetilde{\Pi} : \widetilde{\mathrm{Fl}}(E) \to X$ be the bundle parameterizing triples $(V, V_{-1}, \rho_{-1},\sigma)$, where 
\begin{itemize}
\item $V=(V_1 \subset \ldots \subset V_{N} \subset E_x)$ is an object of $\mathrm{Fl}(E)$,
\item $V_{-1}$ is a $1$-dimensional vector space,
\item $\rho_{-1}$ and $\sigma$ are linear maps $\rho_{-1} : V_{\ell+1} / V_{\ell} \to V_{-1}$ and $\sigma : \BC \to V_{-1}$.
\end{itemize}
Then, $\widetilde{\Pi}$ is smooth with relative tangent bundle
\[ T_{\widetilde{\Pi}} = \mathrm{Cone}
\left( \hom(\CV_{-1},\CV_{-1}) \bigoplus_{i=1}^{N} \hom(\CV_i,\CV_i) \to \hom(\CV_{\ell+1}/\CV_{\ell},\CV_{-1}) \oplus \hom(\CO,\CV_{-1}) \bigoplus_{i=1}^{N} \hom(\CV_i,\CV_{i+1}) \right)\, .
\]
\end{enumerate}
\end{lemma}

By the arguments of Theorem~\ref{thm:pot on nested quot scheme}, there is a perfect obstruction theory $E_{\FN_{\alpha}}^{\bullet} \to \BL_{\FN_{\alpha}}$ with virtual tangent bundle
\[
(E^{\bullet}_{\FN_{\alpha}})^{\vee} \cong \mathrm{Cone}
\left( R\hom_{\pi}(\BI^{\bullet},\BI^{\bullet}) \oplus s^{\ast} T_{S^{[n]}}[-1] \to R \hom_{\pi}( \BI^{\bullet}, \CI_z ) \right)\, ,
\]
where $s : \FN_{\alpha} \to S^{[n]}$ is the map sending $[\varphi : I_z \to F]$ to $I_z$,
and $\BI^{\bullet} = [\CI_z \xrightarrow{\Phi} \CF]$ is the universal quotient.
The proof in Theorem~\ref{thm:pot on nested quot scheme}
shows that $(E^{\bullet}_{\FN_{\alpha}})^{\vee}$
is of amplitude $[-1,1]$.

We have the following consequences:

\begin{prop} 
The moduli space $M = \FN^{\mathrm{fl},\ell-\mathrm{st}}_{\alpha,\mathbf{d}}$ has a perfect obstruction theory $E^{\bullet} \to \BL_M$ which fits into the morphism of distinguished triangles
\begin{equation} \label{compatibility diagram}
\begin{tikzcd}
\Pi^{\ast} E^{\bullet}_{\FN_\alpha} \ar{r} \ar{d} & E^{\bullet} \ar{d} \ar{r} & \BL_{\Pi} \ar{d} \ar{r}{{[1]}} & \ldots \\
\Pi^{\ast} \BL_{\FN_{\alpha}} \ar{r} & \BL_{M} \ar{r} & \BL_{\Pi} \ar{r}{{[1]}} & \ldots
\end{tikzcd}
\end{equation}
In particular, the associated virtual classes in case $\ell=0$ and $\ell=N+1$ are:
\begin{equation} \label{virtual class pullback}
[ \FN^{\mathrm{fl},0-\mathrm{st}}_{\alpha,\mathbf{d}} ]^{\vir}
= \Pi^{\ast} [ \Quot^{\flat}_{\alpha} ]^{\vir},
\quad
[ \FN^{\mathrm{fl},(N+1)-\mathrm{st}}_{\alpha,\mathbf{d}} ]^{\vir}
= \Pi^{\ast} [ \Quot^{\sharp}_{\alpha} ]^{\vir}\, .
\end{equation}
\end{prop}
\begin{proof}
The perfect obstruction theory on 
$\FN_{\alpha}$
is induced by a natural derived enhancement of the stack, as constructed in
the proof in Theorem~\ref{thm:pot on nested quot scheme}.
The relative flag bundle over the derived enhancement of  $\FN_{\alpha}$
hence is a derived enhancement of
$\FN^{\mathrm{fl},\ell-\mathrm{st}}_{\alpha,\mathbf{d}}$.
The distinguished triangle relating the cotangent complexes of the flag bundle, the base, and the fibers of these derived enhancements then yields the first row in \eqref{compatibility diagram}.
Restricting to the classical truncation yields the desired perfect obstruction theory and diagram \eqref{compatibility diagram}.
For the second part, \eqref{compatibility diagram} defines a compatibility datum in the sense of \cite[Definition 5.8]{BF}. Hence the claim follows by \cite[Proposition 5.10]{BF}. 
\end{proof}

The formulas of  \eqref{virtual class pullback} are actually an empty statements since 
both sides of the equalities vanish.
The virtual classes of $\Quot^{\flat}, \Quot^{\sharp}$ both vanish due to everywhere surjective cosections $(E^{\bullet}_{\FN_\alpha})^{\vee} \to \CO[-1]$, see Lemma~\ref{lemma:pot reduction on Quot}. These cosections pull back to natural cosections on the perfect obstruction theory of $\FN^{\mathrm{fl},\ell-\mathrm{st}}_{\alpha,\mathbf{d}}$. 
However, after reduction, we obtain the non-trivial identities:
\begin{equation} \label{reduced virtual class pullback}
[ \FN^{\mathrm{fl},0-\mathrm{st}}_{\alpha,\mathbf{d}} ]^{\red}
= \Pi^{\ast} [ \Quot^{\flat}_{\alpha} ]^{\red},
\quad
[ \FN^{\mathrm{fl},(N+1)-\mathrm{st}}_{\alpha,\mathbf{d}} ]^{\red}
= \Pi^{\ast} [ \Quot^{\sharp}_{\alpha} ]^{\red}.
\end{equation}

\begin{prop} \label{label:prop obst thy}
The master space $\BM_{\ell,\ell+1}$ has a perfect obstruction theory $E^{\bullet} \to \BL_{\BM_{\ell,\ell+1}}$ with virtual tangent bundle fitting into the distinguished triangle
\[ 
\begin{tikzcd}
T_{\widetilde{\Pi}} \ar{r} & (E^{\bullet})^{\vee} \ar{r} & 
\Pi^{\ast} (E^{\bullet}_{\FN_\alpha})^{\vee}
\ar{r}{{[1]}} & \ldots\, ,
\end{tikzcd}
\]
where $\widetilde{\Pi} : \BM_{\ell,\ell+1} \to \FN_{\alpha}$ is the associated smooth projection.
\end{prop}
\begin{proof}
Let $\CW = \pi_{\ast}(F(\kappa))$ be the vector bundle on $\FN_{\alpha}$ with fiber $H^0(F(k))$ over a moduli point $$(s : I \to F)\in \FN_\alpha\, .$$ Then, $\BM_{\ell,\ell+1}\subset \widetilde{\mathrm{Fl}}(\CW)$ is an open subset.
The derived enhancement of $\FN_{\alpha}$ hence gives a natural derived enhancement of $\BM_{\ell,\ell+1}$ by
restriction from $\widetilde{\mathrm{Fl}}(\CW)$.
\end{proof}

As before, the cosection on the obstruction theory $E^{\bullet}_{\FN_\alpha} \to \BL_{\FN_{\alpha}}$ pulls back to a cosection of the perfect obstruction theory of $\BM_{\ell,\ell+1}$ and  defines a reduced virtual class.

\begin{prop}
Let $M=\FQ^{\mathrm{pl},\mathrm{st}}_{n,\mathbf{d}}$ and let
$\Pi : M \to S^{[n]}$ be the forgetful morphism.
There is a perfect obstruction theory $E^{\bullet} \to \BL_{M}$ with virtual tangent bundle fitting into the distinguished triangle
\[
\begin{tikzcd}
T_{\widetilde{\Pi}} \ar{r} & (E^{\bullet})^{\vee} \ar{r} & R\Hom_{\rho}(\CI_{\CZ},\CO_{\CZ})^{\vee} \ar{r}{{[1]}} & \ldots\, ,
\end{tikzcd}
\]
where $\rho : S^{[n]} \times S \to S^{[n]}$ is the projection
and $\CI_{\CZ}$ is the ideal sheaf of the universal subscheme $\CZ \subset S^{[n]} \times S$.
\end{prop}
\begin{proof}
The Hilbert scheme
has a well-known perfect obstruction theory $$R\Hom_{\rho}(\CI_{\CZ},\CO_{\CZ})^{\vee} \to \BL_{S^{[n]}}\, ,$$
which is induced from an underlying derived structure on the Hilbert scheme \cite{MPR}.
So arguing as above yields the desired perfect obstruction theory on the flag bundle.
\end{proof}

The standard cosection of the perfect obstruction theory on $S^{[n]}$ pulls back to a cosection
for the perfect obstruction theory of the flag bundle and defines a reduced virtual class.

\subsection{Torus action and fixed locus}
There is a $\BC^{\ast}$-action on $\BM_{\ell,\ell+1}$ given by
scaling $\rho_{-1}$:
\[ \lambda \cdot (I,(V_{-1},V), (\sigma,\rho_{-1},\rho))
= (I,(V_{-1},V), (\sigma,\lambda \rho_{-1},\rho))\, . \]
The perfect obstruction theory on $\BM_{\ell,\ell+1}$ admits a canonical $\BC^{\ast}$-equivariant lift.
We have the following determination of $\BC^*$-fixed loci including the virtual normal bundles:
\begin{thm}(See \cite[Proposition 5.2.6]{KLT}) \label{fixed locus description}
The $\BC^{\ast}$-fixed locus of $\BM_{\ell,\ell+1}$ is the disjoint union of the following spaces:
\begin{enumerate}
\item[(a)] Let $Z_{\rho_{-1}=0} = \{ \rho_{-1} = 0 \} \subset \BM_{\ell,\ell+1}$. There is a natural isomorphism
\[ Z_{\rho_{-1}=0} \cong \FN^{\mathrm{fl}, (\ell+1)-\mathrm{st}}_{\alpha}. \]
The fixed obstruction theory of $\BM_{\ell,\ell+1}$ is the perfect obstruction theory of
$\FN^{\mathrm{fl}, (\ell+1)-\mathrm{st}}_{\alpha}$.
The virtual normal bundle is $(V_{\ell+1}/V_{\ell})^{\vee} \otimes \Ft$.
\item[(b)] Let $Z_{\sigma=0} = \{ \sigma = 0 \} \subset \BM_{\ell,\ell+1}$. There is a natural isomorphism
\[ Z_{\sigma=0} \cong \FN^{\mathrm{fl}, \ell-\mathrm{st}}_{\alpha}. \]
The fixed obstruction theory of $\BM_{\ell,\ell+1}$ is the perfect obstruction theory of
$\FN^{\mathrm{fl}, \ell-\mathrm{st}}_{\alpha}$.
The virtual normal bundle is $V_{\ell+1}/V_{\ell} \otimes \Ft^{\ast}$.
\item[(c)] For every splitting type appearing in Lemma~\ref{lemma:splitting of l l+1 stability}, there is a component of the $\BC^*$-fixed locus of the form
\[ Z_{\alpha',\mathbf{d'}} \cong \FN_{\alpha',\mathbf{d'}}^{\mathrm{fl},\ell-\mathrm{st}} \times \FQ^{\mathrm{pl},\mathrm{st}}_{m'',\mathbf{d}''} \, .\]
Moreover, the fixed obstruction theory is the induced perfect obstruction theory on the product via the factors.
\end{enumerate}
\end{thm}
\begin{proof}[Sketch of Proof]
Consider first the open subset $D(\sigma) \subset \BM_{\ell,\ell+1}$ of objects where $\sigma$ is not zero. Since $\sigma$ is then an isomorphism, after an isomorphism of the objects we may assume that $V_{-1} = \BC$ and $\sigma = \id$.
Picking $\rho_{-1}$ is hence just an element of the dual of $V_{\ell+1}/V_{\ell}$ and we obtain an isomorphism
\[ D(\sigma) \cong \mathrm{Tot}(\CV_{\ell+1}/\CV_{\ell})^{\ast} \setminus \iota(\FN^{\mathrm{fl},(\ell,\ell+1)-\mathrm{ss}}_{\alpha,\mathbf{d}} \setminus \FN^{\mathrm{fl},(\ell+1)-\mathrm{st}}_{\alpha,\mathbf{d}}) \]
where $\iota$ is the zero section of the total space of $(\CV_{\ell+1}/\CV_{\ell})^{\ast}$ over $\FN^{\mathrm{fl},(\ell,\ell+1)-\mathrm{ss}}_{\alpha,\mathbf{d}}$.
The $\BG_m$-action acts just by scaling the fibers.
The first fixed locus is just the zero-section, so the statements in (a) are clear.

In (b) we can use the open subset where $\rho_{-1}$ is non-vanishing:
\[ D(\rho_{-1}) \cong 
\mathrm{Tot}(\CV_{\ell+1}/\CV_{\ell}) 
\setminus \hat{\iota}(\FN^{\mathrm{fl},(\ell,\ell+1)-\mathrm{ss}}_{\alpha,\mathbf{d}} \setminus \FN^{\mathrm{fl},\ell-\mathrm{st}}_{\alpha,\mathbf{d}}) \]
where $j$ is the zero section of the total space of $\CV_{\ell+1}/\CV_{\ell}$ over $\FN^{\mathrm{fl},(\ell,\ell+1)-\mathrm{ss}}_{\alpha,\mathbf{d}}$; the $\BG_m$-acts by inverse of the scaling action.
Thus (b) follows.

Assume that $(I,(V_{-1},V), (\sigma,\rho_{-1},\rho))$ is $\BC^{\ast}$-fixed and $\rho_{-1}$ and $\sigma$ are both non-zero. After an isomorphism we can assume that $\sigma=\id$ and $V_{-1} = \BC$.
By assumption for every $\lambda \in \BC^{\ast}$, there is an isomorphism
\[ (\varphi,\psi) : (I,(V_{-1},V), (\sigma,\rho_{-1},\rho)) \to (I,(V_{-1},V), (\sigma,\lambda \rho_{-1},\rho)). \]
Since $\sigma=\id$, we must have $\psi=\id$, so $\Fr(\varphi_1)$ must act on $V_{\ell+1}/V_{\ell}$ by multiplication by $\lambda^{-1}$. Thus $(I,V,\rho)$ admits an extra isomorphism, and we are in the situation studied in Lemma~\ref{lemma:splitting of l l+1 stability}.
More precisely Lemma~\ref{lemma:splitting of l l+1 stability} and the family version of it, see \cite[Lemma 4.28]{KT}, describes the closed substack in $\FN^{\mathrm{fl},(\ell,\ell+1)-\mathrm{st}}$ with $\BG_m$-stabilizer as a union of products 
\[ \mathfrak{Z} = \FN_{\alpha',\mathbf{d}'}^{\mathrm{fl},\ell-\mathrm{st}} \times \FQ_{m,\mathbf{d}''}^{\mathrm{fl},\mathrm{st}}. \]
The restrictions of the line bundle $\CL := \CV_{\ell+1}/\CV_{\ell}$ to $\mathfrak{Z}$ has $\BG_m$-weight $1$ under the stabilizer group. This implies that $\mathfrak{Z}$ is a trivial $\BG_m$-gerbe over its $\BG_m$-rigidification, which is given by
\[ Z = \FN_{\alpha',\mathbf{d}'}^{\mathrm{fl},\ell-\mathrm{st}} \times \FQ_{m,\mathbf{d}''}^{\mathrm{pl},\mathrm{st}}. \]
Consider the $\BG_m$-torsor defined by $\CL$,
\[ p : D(\sigma) \cap D(\rho_{-1}) = \mathrm{Tot}(\CL)^{\times} \to \FN_{\alpha,\mathbf{d}}^{\mathrm{fl},(\ell,\ell+1)-\mathrm{ss}}. \]
Then the section $Z \to \mathfrak{Z}$ which defines the trivialization of the $\BG_m$-gerbe can be identified with the map
\[ (\mathrm{Tot}(\CL)^{\times})^{\BG_m} \to \mathfrak{Z} \subset \FN^{\mathrm{fl},(\ell,\ell+1)-\mathrm{ss}}. \]
Moreover, if $\CV_i$ is a vector bundle on $\FN^{\mathrm{fl},(\ell,\ell+1)-\mathrm{st}}$ then the restriction of $p^{\ast}(\CV_i)$ to the fixed locus as an equivariant $\BG_m$-bundle is identified with the restriction to $\mathfrak{Z}$. The restriction of the perfect obstruction theory on $\BM_{\ell,\ell+1}$ to the fixed locus hence decomposes naturally according to the restriction of the perfect obstruction theory $E^{\bullet}$ of $\FN^{\mathrm{fl},(\ell,\ell+1)-\mathrm{ss}}$ to $\mathfrak{Z}$ plus a $1$-dimensional piece coming from the fiber direction. When restricted to $\mathfrak{Z}$, the fixed part of $E^{\bullet}$ over a point with $(s:\CI \to F = F' \oplus Z)$ must be a quadratic in either $(\CI \to F')$ or $Z$. It therefore splits exactly according to the product as a direct sum from each factor and yields the claim.
\end{proof}

\subsection{Conclusion of the proof of Theorem \ref{thm:wallcrossing}}
Consider an arbitrary monomial in descendent classes
\[ \mathsf{P} = 
\prod_{i=1}^{\ell} \ch_{a_i}^{b(i)}(\gamma_i)\,.
\]
We may assume that the weighted (complex cohomological) degree satisfies
\[ \deg \mathsf{P} = \text{reduced vdim} \Quot^{\flat}_{\alpha} = \text{reduced vdim} \Quot^{\sharp}_{\alpha}. \]

Consider the flag bundle
$\Pi : \FNfl_{\alpha,\mathbf{d}} \to \FN_{\alpha}$,
with $\mathbf{d}=(d_1,\ldots,d_{\Fr})$.
 Let $\BT_{\Pi} = \BL_{\Pi}^{\vee}$
 be the relative tangent complex with top Chern class
\[ \mathsf{Q} = \mathsf{Q}(c_j(\CV_i)) = c_{r}( \BT_{\Pi} )\, , \quad r= \mathrm{rk}(\BT_{\Pi})\, . \]
It restricts to the locus of $0$-stable and $d_{\Fr}$-stable objects as the 
top Chern classes of the relative tangent bundles of the forgetful morphisms
\[ \Pi_0 : \FN^{\mathrm{fl},0-\mathrm{st}}_{\alpha,\mathbf{d}} \to \Quot^{\flat}_{\alpha}\, ,
\quad
\Pi_{d_{\Fr}}: \FN^{\mathrm{fl},d_{\Fr}-\mathrm{st}}_{\alpha,\mathbf{d}} \to 
= \Quot^{\sharp}_{\alpha}\, .
\]
Since $\Pi_0$ and $\Pi_{d_{\Fr}}$ are smooth morphisms with isomorphic fibers, we have
\[ (\Pi_0)_{\ast}\left( \mathsf{Q}|_{\FN^{\mathrm{fl},0-\mathrm{st}}_{\alpha,\mathbf{d}}} \right) 
= \mathsf{c}, \quad
(\Pi_{d_{\Fr}})_{\ast}\left( \mathsf{Q}|_{\FN^{\mathrm{fl},d_{\Fr}-\mathrm{st}}_{\alpha,\mathbf{d}}} \right) 
= \mathsf{c}
\]
for a positive constant $\mathsf{c} \in \BZ$.
Using \eqref{reduced virtual class pullback}, we obtain the following relations:
\begin{eqnarray*}
\mathsf{c}\int_{ \left[ \Quot^{\flat}_{\alpha} \right]^{\red} }
\mathsf{P}
&=&
\int_{[\FN^{\mathrm{fl},0-\mathrm{st}}_{\alpha,\mathbf{d}}]^{\red}}
\mathsf{P} \cdot \mathsf{Q} \, .\\
\mathsf{c}\int_{ \left[ \Quot^{\sharp}_{\alpha} \right]^{\red} }
\mathsf{P}
&=&
\int_{[\FN^{\mathrm{fl},d_{\Fr}-\mathrm{st}}_{\alpha,\mathbf{d}}]^{\red}}
\mathsf{P} \cdot \mathsf{Q} \, ,
\end{eqnarray*}
where we have suppressed the pullback by $\Pi$ of $\mathsf{P}$ and the restriction of $\mathsf{Q}$ to the stable loci in the notation.
To relate the above integrals over the $0$- and $d_{\Fr}$-stable loci,
we relate the integrals over $\ell$- and $(\ell+1)$-stable loci for every $\ell$.

To do so, consider the equivariant integral
\[ I = \int_{[ \BM_{\ell,\ell+1} ]^{\red}} \mathsf{P} \cdot \mathsf{Q}\, , \]
where we endow the integrands with the natural equivariant structure.
We have
\[ \deg( \mathsf{P} \cdot \mathsf{Q} ) = \text{reduced vdim} (\FN^{\mathrm{fl},\ell-\mathrm{st}})
=
\text{reduced vdim}( \BM_{\ell,\ell+1} ) - 1\, . \]
Since $\BM_{\ell,\ell+1}$ is proper, the integral  $I$ vanishes.
We apply the virtual localization formula.
By Theorem~\eqref{fixed locus description},
there are contributions from three types of components.
In the third type (c), the fixed perfect obstruction theory
splits as a direct sum of perfect obstruction theories on the factors 
$\FN_{\alpha',\mathbf{d'}}^{\mathrm{fl},\ell-\mathrm{st}}$ and $\FQ^{\mathrm{pl},\mathrm{st}}_{m'',\mathbf{d}''}$.
Each factor has a non-trivial cosection.
However, for the reduced virtual class of $\BM_{\ell,\ell+1}$, we have removed only one cosection.
Hence, the reduced fixed virtual class of the third component vanishes
and  does {\em not} contribute.
The localization formula therefore yields:
\[
0 = I = \int_{[ \FN^{\mathrm{fl}, (\ell+1)-\mathrm{st}}_{\alpha} ]^{\red}} \frac{\mathsf{P} \cdot \mathsf{Q}}{
t - c_1(\CV_{\ell+1}/\CV_{\ell})}
+
\int_{[ \FN^{\mathrm{fl}, (\ell)-\mathrm{st}}_{\alpha} ]^{\red}} \frac{\mathsf{P} \cdot \mathsf{Q}}{
-t + c_1(\CV_{\ell+1}/\CV_{\ell})}\, ,
\]
where $t = c_1(\Ft)$. After
taking the $1/t$-coefficient, we obtain
\[
\int_{[ \FN^{\mathrm{fl}, (\ell+1)-\mathrm{st}}_{\alpha} ]^{\red}} \mathsf{P} \cdot \mathsf{Q}
=
\int_{[ \FN^{\mathrm{fl}, \ell-\mathrm{st}}_{\alpha} ]^{\red}} \mathsf{P} \cdot \mathsf{Q}
\]
as desired. The proof of Theorem~\ref{thm:wallcrossing} is complete. \qed

\section{Hilbert schemes of curves}

\subsection{Conventions}
Let $S$ be a nonsingular projective surface, and let $\beta \in H_2(S,\BZ) \cong H^2(S,\BZ)$ be an effective curve class.\footnote{We identify here homology with cohomology using Poincar\'e duality.}
Let $H_{\beta}$ be the Hilbert scheme of curves parameterizing Cartier divisors with class $\beta$. 
The moduli space $H_{\beta}$ carries a perfect obstruction theory with virtual tangent bundle
\[ T^{\mathrm{vir}} = R \pi_{\ast}( \CO_{\CD_{\beta}}(\CD_{\beta}) )\, , \]
where $\pi : S \times H_{\beta} \to H_{\beta}$ is the projection and
$\CD_{\beta} \subset S \times H_{\beta}$ is the universal Cartier divisor, see \cite{DKO,KT1}. 

We will assume\footnote{In \cite[Appendix]{KoolSeiberg}, a reduced obstruction theory is constructed under the weaker assumption that $H^2(L)=0$ for every effective $L$ with $c_1(L)=\beta$. However, universality then does not hold, so we assume here the stronger condition ($\dag$).}:
\begin{equation} \tag{$\dagger$} H^2(L) = 0 \text{ for every line bundle } L \text{ with } c_1(L) = \beta. \end{equation}
Then, the cohomology sequence
associated to 
$$0 \to \CO \to \CO(\CD_{\beta}) \to \CO(\CD_{\beta}|_{\CD_{\beta}}) \to 0$$
yields an everywhere surjective cosection of the obstruction sheaf of rank $h^2(S,\CO_S)$
which gives rise to a reduced perfect obstruction theory \cite{DKO}
and an associated reduced virtual class $[H_{\beta}]^{\red}$ of dimension $$\frac{1}{2}(\beta^2 + c_1(S)\cdot\beta ) + h^2(S,\CO_S)\, .$$

For $\gamma \in H^{\ast}(S,\BQ)$ and $k \geq 0$, descendent classes on the Hilbert scheme are defined by
\[ \ch_k(\gamma) = \pi_{\ast}( \rho^{\ast}(\gamma) \cdot \ch_k(\CO(\CD_{\beta})) )\, , \]
where $\rho : S \times H_{\beta} \to S$ is the projection.

\subsection{Universality: symplectic surfaces}
For effective classes $\beta$ on symplectic surfaces, the condition ($\dag$) is always satisfied.
We then have the following universality principle for descendent integrals over $H_{\beta}$.

\begin{thm} \label{thm:universality Hilbert scheme symplectic}
Let $S$ be a symplectic surface. Let $\beta \in H^2(S,\BZ)$ be effective,
$\gamma_1,\ldots,\gamma_n \in H^{\ast}(S,\BQ)$, and $k_1,\ldots,k_n \geq 0$. 
Then,  for every symplectic surface $S'$ and degree-preserving $\mathbb{Q}$-algebra isomorphism $$\varphi : H^{\ast}(S, \BQ) \to H^{\ast}(S',\BQ)$$ for which $\varphi(\pt)=\pt$ and $\varphi(\beta)$ is effective, we have
\[
\int_{[H_{\beta}]^{\red}}
\ch_{k_1}(\gamma_1)\cdots \ch_{k_n}(\gamma_n)
=
\int_{[H_{\varphi(\beta)}]^{\red}}
\ch_{k_1}(\varphi(\gamma_1))\cdots \ch_{k_n}(\varphi(\gamma_n))\, .
\]
\end{thm}

\begin{example}
Let $S = E_1 \times E_2$ and $\beta = r \mathsf{f}_1$ be a multiple of the fiber class $\mathsf{f}_1$.
Then $H_{\beta} \cong \Sym^r(E_1)$, where the isomorphism is given by sending $z \in \Sym^r(E_1)$ to the subscheme $\pr_1^{-1}(z)$.
If $\CZ \subset E_1 \times \Sym^r(E_1)$ is the universal subscheme and $\tilde{\pi} : S \times \Sym^r(E_1) \to \Sym^r(E_1)$ is the projection, then 
\[ R^1 \pi_{\ast}(\CO_{\CD_{\beta}}(\CD_{\beta})) = \tilde{\pi}_{\ast}\CO_{\CZ}(\CZ) \cong T_{\Sym^r(E_1)}\, . \]
The obstruction sheaf of the reduced obstruction theory then can be identified with $\tilde{\pi}_{\ast} \CO(\CZ) \cong T_{p}$, where $p : \Sym^{r}(E_1) \to E_1$ is the sum map.
Moreover, we have
\[
\ch_{1}(dx_1 \mathsf{f}_2) = p^{\ast}(dx_1), \quad \ch_1(dy_1 \mathsf{f}_2) = p^{\ast}(dy_1),
\quad \ch_1(\pt) = D_r\, ,
\]
where $D_r \in H^2(\Sym^{r}(E_1))$ is the class of the locus of subschemes containing a fixed point. So
$$D_r = j_{\ast} [\Sym^{r-1}(E_1)]\, ,$$ where, for a point $x \in E_1$,  $$j : \Sym^{r-1}(E_1) \to \Sym^r(E_1)\, ,\ \ \ z' \mapsto z'+x$$ is the inclusion.
We obtain
\begin{gather*}
\int_{[H_{\beta}]^{\red}} \ch_{1}(dx_1 \mathsf{f}_2)\cdot  \ch_{1}(dy_1 \mathsf{f}_2)
= \int_{E_1} dx_1 \cdot dy_1 \cdot p_{\ast}(c_{r-1}(T_p)) = r
\end{gather*}
and
\begin{gather*}
\int_{[H_{\beta}]^{\red}} \ch_{1}(\pt) = \int_{\Sym^r(E_1)} c_{r-1}(T_p)\cdot D_r
= \int_{\Sym^r(E_1)} c(T_{\Sym^r(E_1)}) \cdot D_r \\
= \int_{\Sym^{r-1}(E_1)} c(T_{\Sym^{r-1}(E_1)}) \cdot D_{r-1} = \ldots = \int_{\Sym^1(E_1)} D_1 = 1\, ,
\end{gather*}
where we use that the normal bundle to $j$ satisfies $c_1(N_j) = D_{r-1}$.

The calculation matches  Theorem~\ref{thm:universality Hilbert scheme symplectic}
for the $\BQ$-algebra isomorphism $$\varphi : H^{\ast}(E_1 \times E_2) \to H^{\ast}(E_1 \times E_2)$$ given by
$\varphi(dx_1) = dx_1$, $\varphi(dy_1) = 1/r \cdot dy_1$, $\varphi(d x_2)=dx_2$, $\varphi(dy_2) = r \cdot dy_2$ :
\begin{align*}
\int_{[H_{r \mathsf{f}_1}]^{\red}} \ch_{1}(dx_1 \mathsf{f}_2)\cdot \ch_{1}(dy_1 \mathsf{f}_2)
& =
\int_{[H_{\mathsf{f}_1}]^{\red}} \ch_{1}(\varphi(dx_1 \mathsf{f}_2))\cdot \ch_{1}(\varphi(dy_1 \mathsf{f}_2)) \, ,\\
\int_{[H_{r \mathsf{f}_1}]^{\red}} \ch_{1}(\pt)
& = 
\int_{[H_{\mathsf{f}_1}]^{\red}} \ch_{1}(\pt)\, .
\end{align*}
\end{example}

\subsection{General case}
The universality statement for symplectic surfaces 
follows from a more general universality result.
For the statement, we first record several structures on the cohomology ring $H^{\ast}(S,\BQ)$.

Consider the morphism
\[ \xi : H^{\ast}(S,\BQ) \to \wedge^{\ast} H^1(S,\BQ)^{\ast} \]
which is defined on $H^{4-k}(S,\BQ)$ as the composition
\[ \xi|_{H^{4-k}(S,\BQ)} : H^{4-k}(S,\BQ) \xrightarrow{\mathrm{PD}} H^{k}(S,\BQ)^{\ast} \xrightarrow{} \wedge^{k} H^1(S,\BQ)^{\ast} \xrightarrow{(-1)^{\frac{k (k-1)}{2}}} \wedge^{k} H^1(S,\BQ)^{\ast}, \]
where the first map is Poincar\'e duality,
the second is the dual to the natural inclusion 
$$\wedge^{k} H^1(S,\BQ) \to H^{k}(S,\BQ)\, ,$$ and the third is multiplication by $(-1)^{\frac{k(k-1)}{2}}$.
Moreover, there is a canonical isomorphism\footnote{We work here with integral cohomology modulo torsion.}
\[ \wedge^{2 q(S)} H^1(S,\BZ)^{\ast} \cong \BZ \]
where the sign is fixed by choosing a basis of $H^1(S,\BZ) \subset H^1(S,\CO_S)$ compatible with the complex structure. After
tensoring with $\BQ$, we obtain a degree map
\[ T : \wedge^{\ast} H^1(S,\BQ)^{\ast} \to \wedge^{2q(S)} H^{1}(S,\BQ)^{\ast} \to \BQ\, ,\]
where the first map is projection onto the factor of degree $2q(S)$.
The {\em $T$-degree} of $\alpha \in \wedge^{\ast} H^1(S,\BQ)^{\ast}$   is $T(\alpha)$.

\begin{thm} \label{thm:univ general}
Let $S$ be a nonsingular projective
surface. Let $\beta \in H^2(S,\BZ)$ be an effective class  satisfying condition ($\dagger$), $\gamma_1,\ldots,\gamma_n \in H^{\ast}(S,\BQ)$ be homogeneous, and $k_1,\ldots,k_n \geq 0$.
Then, the integral
\[
\int_{[H_{\beta}]^{\red}}
\ch_{k_1}(\gamma_1)\cdots \ch_{k_n}(\gamma_n)
\]
depends upon $(S,\beta,\gamma_1,\ldots,\gamma_n)$ only through $p_g(S)$, $q(S)$, the degrees of $\gamma_i$, and the $T$-degrees of arbitrary
products of the list of classes
\begin{equation} \label{wfe3434}
\xi\big(\beta^a c_1(S)^b c_2(S)^c\big)\, , \
\xi\big(\gamma_i \beta^a c_1(S)^b c_2(S)^c\big)\, ,\end{equation}
for all $i,a,b,c$
such that 
$\deg_{\BC}(\gamma_i) + a + b + 2c \leq 2$ . \end{thm}

A parallel universality result for stable pair invariants of surfaces with point insertions was obtained by Kool-Thomas  \cite[Theorem 1.2]{KT2}. We apply similar methods in the proof below.
As  explained in \cite[Appendix]{KoolSeiberg}, descendent integrals over $[S_{\beta}]^{\vir}$ 
are not universal in the sense of Theorem~\ref{thm:univ general} if ($\dag$) is not satisfied.

\begin{proof}[Proof of Theorem~\ref{thm:universality Hilbert scheme symplectic} from Theorem~\ref{thm:univ general}]
If $S$ is a $K3$ surface, the map $\xi$ is zero unless $k=4$. Then, $\xi$ is a degree map, and
$$T : \wedge^{\ast} H^1(S)^{\ast} = \wedge^0 H^1(S)^{\ast} = \BQ \to \BQ$$ is the identity.
The statement says that the descendent integrals just depend on the intersection pairings
of $$\int_S \beta^2\, ,\ \ \int_S c_2(S)\, ,\ \  \int_S \beta \gamma_i\, , \ \ \int_S \pt \gamma_i\, ,\ \ \int_{S} \gamma_i\, .$$
Since the $\BQ$-algebra isomorphism $\varphi$ sends the point class to the point class, $\varphi$ preserves intersection pairings.
We obtain Theorem~\ref{thm:universality Hilbert scheme symplectic} in the $K3$ case.

If $S$ is an abelian surface, then $H^{\ast}(S,\BQ) = \wedge^{\ast} H^1(S,\BQ)$, so the map $\xi$ is, up to a sign,  Poincar\'e duality. 
Let $\varphi : H^{\ast}(S,\BQ) \to H^{\ast}(S',\BQ)$ be a degree-preserving $\BQ$-algebra isomorphism such that $\varphi(\pt) = \pt$. 
Then, there is a unique $\BQ$-algebra isomorphism 
\[ \widetilde{\varphi} : H^{\ast}(S,\BQ)^{\ast} \to H^{\ast}(S',\BQ)^{\ast} \]
such that $\widetilde{\varphi}(\lambda)(\varphi(v)) = \lambda(v)$ for all $\lambda \in H^{\ast}(S,\BQ)^{\ast}$ and $v \in H^{\ast}(S,\BQ)$.
Indeed, $\widetilde{\varphi}$ is the map dual to $\varphi^{-1}$.
Since $\varphi$ preserve the point class, $\varphi$ preserves the intersection pairing on $H^{\ast}(S,\BQ)$, so $$\widetilde{\varphi} \circ \mathrm{PD} = \mathrm{PD} \circ \varphi \ \ \text{and}\ \ \ 
 \widetilde{\varphi} \circ \xi = \xi \circ \varphi\, .$$
Moreover, the map $T$ is just evaluation of the point class: 
\[ \forall \alpha \in \wedge^4 H^1(S,\BQ)^{\ast}, \quad T(\alpha) = \alpha(\pt)\, . \]
Hence, 
$ T(\tilde{\varphi}(\alpha)) = \tilde{\varphi}(\alpha)(\pt)  = \alpha(\varphi(\pt)) = \alpha(\pt) = T(\alpha)$.
For all homogeneous classes $\gamma_i \in H^{\ast}(S,\BQ)$ and integers $a_i \geq 0$, we have
\[ T\left( \prod_i \xi\big(\varphi(\gamma_i) \varphi(\beta)^{a_i}\big) \right)
= T\left( \prod_i \widetilde{\varphi} \xi\big(\gamma_i \beta^{a_i}\big) \right)
= T \circ \widetilde{\varphi}\left( \prod_i \xi\big(\gamma_i \beta^{a_i}\big) \right)
= T \left( \prod_i \xi\big(\gamma_i \beta^{a_i}\big)\right). \]
The $T$-degrees of the products of classes \eqref{wfe3434}
are therefore the same as their images under $\varphi$. We obtain Theorem \ref{thm:universality Hilbert scheme symplectic}  in the  abelian surface case.
\end{proof}
\subsection{Picard variety}
We start with explaining the geometric meaning of the $T$-degrees.
Let $\Pic_{\beta}$ be the Picard variety parameterizing line bundles with first Chern class $\beta$.
Let $\CL_{\beta}$ be the Poincar\'e bundle on $S \times \Pic_{\beta}$ normalized by
$$\CL_{\beta}|_{\Pic_{\beta} \times \{ x \}} \cong \CO_{\Pic_{\beta}}$$ for a fixed point $x \in S$.
Consider the decomposition of $c_1(\CL_{\beta})$ in terms of bi-degree:
\[ c_1(\CL_{\beta}) = (\beta, \eta, 0) \in H^2(S,\BZ) \oplus (H^1(S,\BZ) \otimes H^1(\Pic_{\beta}),\BZ) \otimes H^2(\Pic_{\beta},\BZ)\, . \]
The middle class $\eta$ can be identified with the identity morphism under the canonical isomorphism
\[ H^1(\Pic_{\beta},\BZ) = H^1(S,\BZ)^{\ast}\,, \]
see \cite[Section 4]{KT2}. Moreover, we have
\[ H^{\ast}(\Pic_{\beta}) = \wedge^{\ast} H^{\ast}(S,\BQ)^{\ast}\, . \]
The map $\xi$ can be written as
\begin{equation}\label{gr4}
    \xi(\gamma) = \pi_{\ast}( \rho^{\ast}(\gamma) \exp(\eta))\, , 
    \end{equation}
where $\rho : S \times \Pic_{\beta} \to S$ and $\pi : S \times \Pic_{\beta} \to \Pic_{\beta}$ are the projections\footnote{Let $e_1,\ldots,e_n$ be a basis of $H^1(S,\BQ)$, and let $f_1,\ldots,f_n$ be a dual basis of $H^1(S,\BQ)$. Then, $\eta = \sum_i e_i \boxtimes f_i$, and \eqref{gr4} follows by a direct computation. The sign $(-1)^{\frac{k (k-1)}{2}}$ arises from reordering the $e_i$ and $f_i$ in the power $\eta^k$.}.
Finally, the morphism $$T: \wedge H^{\ast}(S,\BQ)^{\ast} \to \BQ$$ is just the degree map on $\Pic_{\beta}$.

\subsection{Proof of Theorem~\ref{thm:univ general}}
For simplicity, we first assume  $H^1(L) = 0$ for every line bundle $L \in \Pic_{\beta}$. The condition holds, for example, for ample $\beta$ on symplectic surfaces $S$ by Kodaira vanishing. 
Then, $E = \pi_{\ast}(\CL_{\beta})$ is a vector bundle,  we have an isomorphism
\[ H_{\beta} = \p(E)\, , \]
and the bundle map $p : \p(E) \to \Pic_{\beta}$ is the Abel-Jacobi map.
The canonical inclusion $$\CO_{\p(E)}(-1) \hookrightarrow p^{\ast}(E)$$ yields a section $\CO_{S\times H_\beta} \to \CL_{\beta}(1)$ on $S \times H_{\beta}$,
where we have omitted the pullback by $\id_S \times p$, whose zero locus defines the universal Cartier divisor $\CD_{\beta}$.
Since
\[ \dim H_{\beta} = \rk(E) - 1 + \dim \Pic_{\beta} = 
h^2(S,\CO_S)+
\frac{1}{2} (\beta^2 + c_1(S)\cdot \beta) \, ,\]
we see
$[H_{\beta}] = [H_{\beta}]^{\red}$.
For $z = c_1(\CO_{\p(E)}(1))$, we have
\[ p_{\ast}(z^i) = s_{i - (\rk(E)-1)}(E) \, ,\]
where $s_j(E)$ are the Segre classes. 
By Grothendieck-Riemann-Roch, there exist  universal polynomials
$P_{j}$ such that
\[ s_j(E) = P_j(\xi(1), \xi(c_1(S)), \xi(c_2(S)), \xi(\beta), \xi(\beta^2), \xi(\beta c_1(S)))\, . \]

For a formal variable $x$, we write
\[ \ch_x(E) = \sum_{k \geq 0} \ch_k(E) x^k\]
and correspondingly for the descendent classes,
\[ \ch_x(\gamma) = \sum_{k \geq 0} \ch_k(\gamma) x^k = \pi_{\ast}( \rho^{\ast}(\gamma) \ch_x(\CO(\CD_{\beta})))\, . \]
We also write
$$\xi_x(\gamma) = \sum_{k \geq 0}  \pi_{\ast}\left(  \rho^{\ast}(\gamma)
\frac{\eta^k}{k!}
\right)x^k\, . $$
Then, we obtain
\begin{align*}
\int_{H_{\beta}} \ch_{x_1}(\gamma_1) \cdots \ch_{x_n}(\gamma_n)
& = 
\int_{H_{\beta} \times S^n} \pr_{01}^{\ast}(\ch_{x_1}(\CO(\CD_{\beta})))
\cdots 
\pr_{0n}^{\ast}( \ch_{x_n}(\CO(\CD_{\beta})))\,
\pr_1^{\ast}(\gamma_1) \cdots \pr_{n}^{\ast}(\gamma_n) \\
& = 
\int_{H_{\beta} \times S^n} 
\pr_0^{\ast}( e^{(x_1+\ldots+x_n) z})
\pr_{01}^{\ast}( e^{x_1 \eta} ) \cdots \pr_{0n}^{\ast}(e^{x_n \eta})
\, \pr_1^{\ast}( e^{x_1 \beta} \gamma_1) \cdots \pr_n^{\ast}( e^{x_n \beta} \gamma_n) \\
& = 
\int_{\Pic_{\beta} \times S^n}
\pr_0^{\ast}(p_{\ast}(e^{(x_1+\ldots+x_n) z}))
\pr_{01}^{\ast}( e^{x_1 \eta} ) \cdots \pr_{0n}^{\ast}(e^{x_n \eta})
\, \pr_1^{\ast}( e^{x_1 \beta} \gamma_1) \cdots \pr_n^{\ast}( e^{x_n \beta} \gamma_n) \\
& = \int_{\Pic_{\beta}}
\left( \sum_{j \geq 0} P_{j-\mu}
\frac{(x_1+\ldots+x_n)^{j}}{j!} \right)
\xi_{x_1}(e^{x_1 \beta} \gamma_1) \cdots \xi_{x_n}(e^{x_n \beta} \gamma_n),
\end{align*}
where the factors of $H_{\beta} \times S^n$ are indexed by $(0,\ldots,n)$
and $$\mu = \rk(E) - 1 = p_g(S)-q(S) + \frac{\beta^2 + \beta \cdot c_1(S)}{2}\, .$$


In case $L \in \Pic_{\beta}$ may have non-vanishing
$H^1(L)$,
we pick an ample divisor $A \subset S$ such that $$H^i(L(A)) = 0$$ for all $L \in \Pic_{\beta}$ and $i>0$. We define $$E_0 = \pi_{\ast}(\CL_{\beta}(A))\, , \ \ \ E_1 = \pi_{\ast}(\CL_{\beta}(A)|_{A})\, $$
on $\Pic_\beta$.
By applying $\pi_{\ast}$ to the exact sequence $$0 \to \CL_{\beta} \to \CL_{\beta}(A) \to \CL_{\beta}(A)|_{A} \to 0$$ and considering the associated long exact sequence, we find
\begin{enumerate}
\item[(i)] $E_1$ is a vector bundle since $R^2 \pi_{\ast} \CL_{\beta} = 0$ by condition ($\dag$),
\item[(ii)] $E_\bullet=[E_0 \to E_1] \cong R \pi_{\ast}(\CL_{\beta})$ in $D^b(\Pic_{\beta})$.
\end{enumerate}
Consider the effective curve class $\gamma = \beta + c_1(A)$. Then $H_{\gamma} = \p(E_0)$ where 
the universal Cartier divisor $$\CD_{\gamma}\subset 
S \times H_\gamma
$$ is given by the natural section of $\CL_{\beta}(A)(1)$.
There exists a natural section
\[ s: \CO_{\p(E_0)} \xrightarrow{t} p^{\ast}(E_0)(1) = \pi_{\ast}( \CO(\CD_{\gamma})) \to
p^{\ast}(E_1)(1) = \pi_{\ast}( \CO(\CD_{\gamma})|_{A})\, , \]
where the second map is the restriction to $A$. The section $s$ vanishes exactly on the locus of Cartier divisors containing $A$, so there is an isomorphism 
\begin{equation} \label{fvvf3} H_{\beta} \cong Z(s)\subset H_\gamma
\end{equation}
given by $C \mapsto C+A$. As shown in \cite[Appendix]{KT1}, the refined Euler class of $s$ is exactly the reduced virtual class $[H_{\beta}]^{\red}$.

Let $j : S\times H_{\beta} \to S\times H_{\gamma}$ be the inclusion
\eqref{fvvf3}, then
\[ \CO(\CD_{\beta}) = j^{\ast} \CO(\CD_{\gamma}) \otimes \rho^{\ast}\CO_S(-A)
= j^{\ast}( \CL_{\beta}(A)(1) \otimes \rho^{\ast}(\CO_S(-A)) ) = j^{\ast}( \CL_{\beta}(1))\, , \]
so we obtain
\[ \ch_x(\CO(\CD_{\beta})) = e^{x(\rho^{\ast}(\beta) + p^{\ast}(\eta) + z)}|_{S\times H_{\beta}}. \]
For $z=c_1(\CO_{\p(E_0)(1)})$, 
we have
\begin{align*}
p_{\ast}(z^{i} \cap [H_{\beta}]^{\red}) & = 
p_{\ast}( z^{i} c_{\rk(E_1)}(p^*(E_1)(1))) \\
& = \sum_{j} p_{\ast}(z^{i+j}) c_{\rk(E_1)-j}(p^*(E_1)) \\
& = \sum_{j} c_{\rk(E_1)-j}(E_1) s_{i+j-(\rk(E_0)-1)}(E_0) \\
& = s_{i-\rk(E_{\bullet})+1}(E_{\bullet})\, .
\end{align*}
Moreover, with the same polynomial $P_j$ as before, we have
\[ s_j(E_\bullet) = P_j(\xi(1), \xi(c_1(S)), \xi(c_2(S)), \xi(\beta), \xi(\beta^2), \xi(\beta c_1(S)))\, . \]

By the same argument as before,   
we conclude
\[
\int_{[H_{\beta}]^{\red}} \ch_{x_1}(\gamma_1) \cdots \ch_{x_n}(\gamma_n)
= \int_{\Pic_{\beta}}
\left( \sum_{j \geq 0} P_{j-\mu}
\frac{(x_1+\ldots+x_n)^{j}}{j!} \right)
\xi_{x_1}(e^{x_1 \beta} \gamma_1) \cdots \xi_{x_n}(e^{x_n \beta} \gamma_n),
\]
where $\mu = \rk(E_{\bullet}) - 1 = p_g(S)-q(S) + \frac{\beta^2 + \beta \cdot c_1(S)}{2}$.
After taking the $x_1^{k_1} \cdots x_n^{k_n}$-coefficient, the claim follows. \qed

\printbibliography

\end{document}